\documentclass[a4paper,11pt]{article}

\usepackage[T1]{fontenc}
\usepackage{lmodern, graphicx}
\usepackage[french,english]{babel}
\usepackage{csquotes}

\usepackage{amsmath}
\usepackage{amssymb}
\usepackage{amsthm}
\usepackage{enumitem}
\usepackage{tikz-cd}

\usepackage[
backend=biber,
style=numeric,
sorting=nyt
]{biblatex}

\usepackage[hidelinks]{hyperref}
\usepackage{stmaryrd}
\usepackage{indentfirst}
\usepackage{soul}
\usepackage{chngpage}
\usepackage{adjustbox}

\usepackage{GPack2b}

\usepackage{todonotes}
\setuptodonotes{color=blue!10!white}
\usepackage{tikz}
\usepackage{tikz-cd} 
\usetikzlibrary{positioning,intersections}
\usetikzlibrary{decorations.pathmorphing} 
\usetikzlibrary{arrows}

\newcommand{\Bv}{\Bc_{\V}}
\renewcommand{\t}{\tau}
\renewcommand{\u}{\upsilon}

\DeclareMathOperator{\pseudo}{pseudo}
\newcommand{\Mt}{\M_{\times}}

\DeclareMathOperator{\co}{co}

\begin{document}

\title{Sheaves on a bicategory}
\author{Olivia Caramello and Elio Pivet}
\date{}

\maketitle

\begin{abstract}

We give a detailed account of the theory of enrichment over a bicategory and show that it establishes a two-fold generalization of enrichment over both quantaloids and monoidal categories. We define complete $\Bc$-categories, a generalization of Cauchy-complete enriched categories serving as a basis for the development of sheaf theory in the enriched setting. We prove an adjunction between complete $\Bc$-categories and $2$-presheaves on the category $\Map(\Bc)$ of left adjoints in $\Bc$. We express conditions under which this adjunction becomes a left-exact reflection, yielding back the usual results linking sheaves on sites and enriched categories. We prove that our adjunction recovers the already existing results about quantaloids, and discuss the fixed points of the adjunction in the monoidal case.
\end{abstract}

\tableofcontents

\newpage
\section{Introduction}

The use of enrichment techniques to give an alternative description of sheaf conditions dates back to \cite{WaltersCahiers}, where Walters proved that sheaves on a locale $X$ or a site $(\C,J)$ could be represented as categories enriched in some object constructed from $X$ or $(\C,J)$; the object in question is a quantaloid, a multi-object generalization of quantales, which are themselves a non-commutative generalization of locales. Broadly speaking, a quantale is a locale with an additional operation which is bound to play the role of a ``non-commutative disjunction''; this concept was introduced in 1983 by Mulvey~\cite{et} in order to study $C^*$-algebras. Later, stemming from Walters' ideas, several authors gave definitions for sheaves over quantales, generalizing the already-existing notions around locales~\cite{BVDBQS,NawazThesis,BCSBCTQS,GylysQSSQ}. At first most of these articles required the quantale to satisfy several ``locale-like'' conditions, but soon enough the theory evolved to include the case of quantaloids~\cite{Gylys,VDBQNCRR,CSEQCDF,HeymansThesis}. More recently, several articles pursuing the study of sheaves over quantales and quantaloids appeared. The theory developed in \cite{TMM2k22} notably addresses the case of those associated with sites as in~\cite{WaltersCahiers}. Sheaves on a quantaloid $\Q$ are defined as categories enriched in $\Q$ which satisfy a Cauchy-completion property: in the geometric case of sheaves over a locale, the glueing and restriction conditions appear as representability conditions for certain distributors (see section~\ref{exSecPresh} and the examples in it for more details). 

On another hand, enriched category theory is usually developed over a monoidal category~\cite{Kelly}, and enrichment over a quantaloid, and albeit very akin to that notion, is not a direct application of the already existing theory for monoidal categories. In order to work in a unified setting subsuming both these theories, we have chosen to develop in the present paper a theory of enrichment over a bicategory, and extend to this setting the fibrational representation of sheaves arising in the context of quantaloids. We indeed produce such a fibrational representation, not only for ordinary, discrete-values sheaves but for $2$-sheaves (in the sense of~\cite{StreetCBS}) or enriched sheaves (in the sense of~\cite{BQTES}), paving the way for the development of a general theory of \emph{enriched toposes}. Our main result is, for a given bicategory $\Bc$, an adjunction between the category $\Catk(\Bc)$ of complete $\Bc$-categories (in a sense which generalizes the usual Cauchy-completeness for enriched categories) and a category of pseudofunctors $[\Map(\Bc)^{\coop},\Cat]$, as in the following diagram, where $\si:[\Map(\Bc)^{\coop},\Cat] \to \Cat(\Bc)$ is a kind of enriched ``Grothendieck construction'' functor, $P$ is a kind of functor of ``fibers'' and $\C$ is a bicategorical analogue of the Cauchy-completion functor (providing the role of sheafification in our context): 

\[
\begin{tikzcd}[row sep=40, column sep = 20]
{\Cat(\Bc)} 
\arrow[d,"{\C}"', shift right=2]
\arrow[d,"{\scriptstyle\dashv}"description, phantom]
&& {[\Map(\Bc)^{\coop},\Cat]} \arrow[dll,"\C\si "',bend right=12]
\arrow[ll,"{\si}"']\\ 
{\Catk(\Bc)}
\arrow[u,"{i}"', shift right=2]
\arrow[urr,"P"',bend right=12, start anchor=real east]
\arrow[urr,"{\scriptstyle \rotatebox{305}{$\dashv$}} "description, phantom]
\end{tikzcd}
\]
 
This result motivates us to define $\Catk(\Bc)$ as the ``\emph{category of sheaves on $\Bc$}'', yielding a notion of ``enriched topos''. Indeed, we can expect these categories to satisfy, at least in the case of well-behaved bicategories $\Bc$, good categorical properties generalizing those of $2$-Grothendieck toposes, notably as a consequence of the existence of such an adjunction; a systematic investigation of these properties will be the subject of future work. 

This bicategorical setting subsumes both sheaves over quantaloids and sheaves enriched over monoidal categories (the eludication of the precise relationship between the specialization of our theory in the monoidal setting and the theory of enriched sheaves over a monoidal category of \cite{BQTES} will be addressed in future work). In the case of a quantaloid $\Q$, the above adjunction restricts to a left-exact reflection $\Catsk(\Q) \to [\Map(\Q)^{\op},\Set]$, yielding back the results of~\cite{HeymansThesis,Walters}. In the case of a monoidal category $\V$, the adjunction is not always a left-exact reflection; we give some conditions in which it is, for example when $\V = \Set$ or $\V =\Ab$. 

The structure of the paper is as follows. The first part is devoted to an exposition of the theory of enrichment over a bicategory, based on~\cite{StreetEnrichment}. With respect to \cite{StreetEnrichment}, we provide more details, introduce new terminology and offer concrete proofs of coherence conditions. We also discuss a number of examples and establish a Yoneda and co-Yoneda lemma for $\Bc$-categories that generalize those of ordinary category theory. The second part is devoted to establishing the adjunction $\C\si \dashv P$ as in the above diagram. In the third and final part of the article, we look at two classes of examples: quantaloids, which were studied in~\cite{HeymansThesis} and which provide a particularly well-behaved case of our adjunction, and monoidal categories, for which we give conditions characterizing the fixed points of the adjunction.

\section{Categories enriched in a bicategory} 

This section presents an extensive account of the theory of categories enriched in a bicategory, following Street's article~\cite{StreetEnrichment} but using a more modern notation heavily inspired by~\cite{HeymansThesis}; most results of this section can already been found in~\cite{StreetEnrichment}. 

There are at least two important classes examples of bicategories to be kept in mind, which will be considered in detail in section~\ref{secEx}: these are, on one hand, \textit{quantaloids}, with their posetal $2$-structure, and, on the other, (deloopings of) \textit{monoidal categories}, with their trivial $0$-structure. Together, these orthogonal classes of examples form a good way to understand the theory of enrichment in a bicategory. \\ 

In all this paper $\Bc$ will denote a bicategory which is assumed to be locally cocomplete, meaning that all hom-categories admit all colimits, and closed, meaning that right Kan extensions and lifts exist. This is to ensure that all pre- and post- composition functors respect colimits. We will sometimes consider \textit{involutive} bicategories, meaning that $\Bc$ is equipped with a involution $(-)^{\circ}: \Bc^{\op} \to \Bc$ which is the identity on objects.

\subsection{$\Bc$-categories and $\Bc$-functors}

\begin{deft}
An \textbf{$\Ob(\Bc)$-typed set} is the data of: 
\begin{enumerate}
\itb A set $A$.
\itb A function: $t: A \to \Ob(\Bc)$, that we call the \textbf{typing function}.
\end{enumerate}
An element $a \in A$ is thus said to be \textbf{of type $x$} with $x \in \Ob(\Bc)$ if $ta = x$.
\end{deft}

\begin{ex}[Sections of a presheaf]\label{exSecPresh}
Consider $X$ a topological space and $U$ an open subset of $X$. We can, following~\cite{WaltersCahiers}, look at the bicategory $\R(X)$, of which objects are given by open subsets of $X$, arrows $U \to V$ are given by open subsets $W \subseteq U \wedge V$, and $2$-cells are given by the inclusion of open subsets. This yields what is called a quantaloid, i.e. a category enriched in the monoidal category of complete lattices and sup-preserving morphisms; alternatively, a quantaloid is simply a locally posetal, locally cocomplete closed bicategory. We can then associate to any continuous function $f: U \to \Rb$ the open subset $U$ of $X$ over which $f$ is defined. This produces a $\Ob(\R(X))$-typed set, with base set the set of all continuous functions going into $\Rb$.
More generally, if $F$ is a (pre)sheaf over a topological space $X$, the set of its sections $\coprod_{U \subseteq X} F(U)$ is naturally endowed with the structure of a $\Ob(\Bc)$-typed set~; this notion of typed set is the first step towards a fibrational view of sheaf theory. There is a general correspondence between sheaves on the topological space $X$ and enriched categories on a particular quantaloid associated to $X$~\cite{HeymansThesis}.
\end{ex}

\begin{deft}
Let $A,C$ be two $\Ob(\Bc)$-typed sets. A \textbf{$\Bc$-matrix} $\M: A \to C$ is the data, for all $a \in A$, $c \in C$, of a $1$-cell $\M(c,a): ta \to tc$ in $\Bc$.
\end{deft}

\begin{rk}
In the case of a quantaloid $\Q$ as in~\cite{HeymansThesis}, one can define the composition of $\Q$-matrices by $(\M\N)(c,a) = \bigvee_b \M(c,b)\N(b,a)$. In the case of a general bicategory $\Bc$, we could very well generalize this definition to produce a category $\Matr(\Bc)$ by using coproducts (which generalize suprema in quantaloids). However, doing that would hinder many of the properties that are desirable in this context; for example the composite of two distributors, if defined like so, would not in general yield a distributor. Moreover, the $\Bc$-matrices associated to $\Bc$-categories would not be idempotent, meaning that we would not be able to express $\Bc$-categories as monads in $\Matr(\Bc)$. This is because coproducts are too discrete of a colimit; as we shall see in section \ref{sec:distributors}, a richer diagram is needed (such a diagram is trivial in the case of quantaloids due to the locally posetal nature of these bicategories, but not in general). This is the object of definition~\ref{CompDistrib}, which properly defines composition for distributors.
\end{rk}

\begin{deft}
A \textbf{$\Bc$-category} is the data of:
\begin{enumerate}
\itb A $\Ob(\Bc)$-typed set $A$.
\itb A $\Bc$-matrix $\M: A \to A$.
\itb For all $a,b,c \in A$, an \textbf{idempotency} $2$-cell: $\iota_{abc}: \M(a,b)\M(b,c) \To \M(a,c)$.
\itb For all $a \in A$, a \textbf{reflexivity} $2$-cell: $\rho_a: \id_{ta} \To \M(a,a)$.
\end{enumerate}
This data must satisfy the following coherence properties, for all $a,b,c,d \in A$:
\begin{enumerate}
\itb \textbf{Associativity of $\iota$}: $\iota_{a,c,d} \bullet (\iota_{a,b,c} * 1_{\M(c,d)}) = \iota_{a,b,d} \bullet (1_{\M(a,b)} * \iota_{b,c,d})$.
\itb \textbf{Unitality conditions}: 
\begin{enumerate}
\item[-] $\iota_{a,b,b} \bullet (1_{\M(a,b)} * \rho_b) = 1_{\M(a,b)}$
\item[-] $\iota_{a,a,b} \bullet (\rho_a * 1_{\M(a,b)}) = 1_{\M(a,b)}$
\end{enumerate}
\end{enumerate}
\end{deft}

\begin{rk}
Let us express again the coherence conditions of a $\Bc$-category in a more visual, diagrammatic way. \\ 

\textbf{Associativity of $\iota$:}

\[
\begin{tikzcd}[column sep = 50]
ta
&tb
\arrow[l,"{\M(a,b)}"{name={11}}]
&tc
\arrow[l,"{\M(b,c)}"{name={12}}]
\arrow[ll,""{name={112}}, phantom]
&td
\arrow[l,"{\M(c,d)}"{name={13}}]\\ 
ta
&& tc
\arrow[ll,"{\M(a,c)}"{name={212}}]
& td
\arrow[l,"{\M(c,d)}"{name={23}}]
\arrow[lll,""{name={2123}},phantom]\\
ta
&&& td
\arrow[lll,"{\M(a,d)}"{name={3}}]
\arrow[Rightarrow, from={112}, to={212}, "{\iota_{abc}}", shorten >=3, shorten <=5]
\arrow[Rightarrow, from={13}, to={23}, "{1_{\M(c,d)}}", shorten >=3]
\arrow[Rightarrow, from={2123}, to={3}, "{\iota_{acd}}", shorten >=3, shorten <=5]
\end{tikzcd}
\]
Must be equal to:
\[
\begin{tikzcd}[column sep = 50]
ta
&tb
\arrow[l,"{\M(a,b)}"{name={11}}]
&tc
\arrow[l,"{\M(b,c)}"{name={12}}]
&td
\arrow[l,"{\M(c,d)}"{name={13}}]
\arrow[ll,""{name={123}}, phantom]\\ 
ta
& tb
\arrow[l,"{\M(a,b)}"{name={21}}]
&& td
\arrow[ll,"{\M(b,d)}"{name={223}}]
\arrow[lll,""{name={2123}},phantom]\\
ta
&&& td
\arrow[lll,"{\M(a,d)}"{name={3}}]
\arrow[Rightarrow, from={11}, to={21}, "{1_{\M(a,b)}}", shorten >=3]
\arrow[Rightarrow, from={123}, to={223}, "{\iota_{bcd}}", shorten >=3, shorten <=5]
\arrow[Rightarrow, from={2123}, to={3}, "{\iota_{abd}}", shorten >=3, shorten <=5]
\end{tikzcd}
\]

\textbf{Unitality conditions:}
\[
\begin{tikzcd}[column sep = 50]
ta
&tb
\arrow[l,"{\M(a,b)}"{name={11}}]
&tb
\arrow[l,"{\id_{tb}}"{name={12}}]\\
ta
&tb
\arrow[l,"{\M(a,b)}" {name={21}}]
&tb
\arrow[l,"{\M(b,b)}"{name={22}}]
\arrow[ll,""{name=2}, phantom]\\
ta
&& tb
\arrow[ll,"{\M(a,b)}"{name={3}}]
\arrow[Rightarrow, from={11},to={21},"{1_{\M(a,b)}}", shorten >=3]
\arrow[Rightarrow, from={12},to={22},"{\rho_b}", shorten >=3]
\arrow[Rightarrow, from={2},to={3},"{\iota_{abb}}", shorten >=3, shorten <=5]
\end{tikzcd}
= 
\begin{tikzcd}[column sep = 50]
ta
& tb
\arrow[l,"{\M(a,b)}"{name={1}}]\\ 
\\ 
ta 
& tb
\arrow[l,"{\M(a,b)}"{name={2}}]
\arrow[Rightarrow, from={1},to={2},"{1_{\M(a,b)}}", shorten >=3]
\end{tikzcd}
\]

\hspace*{2pt}

\[
\begin{tikzcd}[column sep = 50]
ta
&ta
\arrow[l,"{\id_{ta}}"{name={11}}]
&tb
\arrow[l,"{\M(a,b)}"{name={12}}]\\
ta
&ta
\arrow[l,"{\M(a,a)}" {name={21}}]
&tb
\arrow[l,"{\M(a,b)}"{name={22}}]
\arrow[ll,""{name=2}, phantom]\\
ta
&& tb
\arrow[ll,"{\M(a,b)}"{name={3}}]
\arrow[Rightarrow, from={11},to={21},"{\rho_a}", shorten >=3]
\arrow[Rightarrow, from={12},to={22},"{1_{\M(a,b)}}", shorten >=3]
\arrow[Rightarrow, from={2},to={3},"{\iota_{abb}}", shorten >=3, shorten <=5]
\end{tikzcd}
= 
\begin{tikzcd}[column sep = 50]
ta
& tb
\arrow[l,"{\M(a,b)}"{name={1}}]\\ 
\\ 
ta 
& tb
\arrow[l,"{\M(a,b)}"{name={2}}]
\arrow[Rightarrow, from={1},to={2},"{1_{\M(a,b)}}", shorten >=3]
\end{tikzcd}
\]
\end{rk}

\begin{ex}[Monoidal categories]
Let $\V$ be a closed monoidal category; then $\V$ can be seen as a bicategory $\Bc_{\V}$ with only one object, in which the $1$-cells correspond to the objects of $\V$, the $2$-cells to the arrows of $\V$, and the composition to the monoidal product. With this representation, the coherence relations of the monoidal product correspond exactly to the properties for $\Bv$ to be a bicategory. Then $\Bv$-categories and $\V$-categories are the same thing; the ``idempotency'' $2$-cell correspond to the composition operation, while the ``reflexivity'' $2$-cell corresponds to the choice of the identity over $a$. Through this prism, the coherence properties can be understood as follows: the associativity of $\iota$ corresponds to the fact that composition is associative, while the unitality condition assert that the identity over an object is indeed an identity with respect to composition. 
\end{ex}

\begin{ex}[Sections of a presheaf]\label{secpreshf}
Let $X$ be a topological space, and consider the bicategory $\R(X)$ already described in Example~\ref{exSecPresh}. For a presheaf $F$ on $X$, denote $\si F$ for the $\Ob(\R(X))$-typed set of its sections. This typed set can be endowed with the structure of a $\R(X)$-category by defining $\M(f,g)$ to by the greatest open subset of $X$ on which $f$ and $g$ are both defined and are equal. This is the ``local equality'' point of view that is already well-known to describe sheaves over locales; all presheaves can be described in this way, and sheaves over $X$ correspond precisely to the $\R(X)$-categories which are \textit{complete}~\cite{WaltersCahiers}.
\end{ex}

\begin{ex}[Arrows of $\Bc$]
Recall that we have asked our bicategory $\Bc$ to be closed, i.e. that right Kan extensions and lifts exist. If $\Bc$ is of the form $\Bv$, this amounts to asking that $\V$ be closed as a monoidal category. A closed monoidal category $\V$ has the property that $\V$ is itself a $\V$-category, using the internal hom, which in the delooping $\Bv$ corresponds to the right Kan extension operation. Closed bicategories therefore enjoy the property that they have a $\Bc$-category of their arrows. We denote by $\Bc_1$ the set of all $1$-cells of $\Bc$; this set can by typed through the codomain function $t = \codom: \Bc_1 \to \Ob(\Bc)$; for any $f: b_1 \to b_2$, we have $tf = b_2$. Let us consider the $\Bc$-endomatrix $[-,-]$ on the $\Ob(\Bc)$-typed set $\Bc_1$ defined as follows, for all $f,g \in \Bc_1$: if $\dom(f) = \dom(g)$, we define $[g,f]: tf \to tg$ as the right Kan extension of $g$ along $f$. Else, we define $[g,f]$ to be the initial element of $\Bc(tf,tg)$. Then $([-,-],\Bc_1)$ is a $\Bc$-category.
\end{ex}

\begin{ex}[Objects of $\Bc$]
Let $*$ be a $0$-cell in $\Bc$. Consider the $\Ob(\Bc)$-typed set $\{*\}$ with $t*=*$. Let $(\id_*)$ be the endomatrix on this $\Ob(\Bc)$-typed set $\{*\}$, with the reflexivity and idempotency $2$-cells given by $1_{\id_*}$. This yields a structure that we denote by $(\id_*,*)$. Then $(\id_*,*)$ is a $\Bc$-category. This important example will allow us to define a functor $\Bc \to \Dist(\Bc)$ and the notion of ``singleton'' for a $\Bc$-category.
\end{ex}

\begin{deft}
Let $(\M,A)$ and $(\N,C)$ be two $\Bc$-categories. A \textbf{$\Bc$-functor} $f$ from $(\M,A)$ to $(\N,C)$ consists in the data of:
\begin{enumerate}
\itb A type-preserving function $f: A \to C$, i.e. a function such that for all $a \in A$, $tf(a) = ta$.
\itb For all $a,a' \in A$, a $2$-cell $f_{aa'}: \M(a,a') \To \N(f(a),f(a'))$.
\end{enumerate}
This data must satisfy the following coherence conditions, for all $a,b,c \in A$:
\begin{enumerate}
\itb $f_{a,c}\bullet \iota_{a,b,c} = \iota_{f(a),f(b),f(c)} \bullet (f_{a,b} * f_{b,c})$
\itb $\rho_{f(a)} = f_{a,a} \bullet \rho_a$
\end{enumerate}
\end{deft}

\begin{ex}[$\V$-functors]
If $\V$ is a monoidal category, then $\Bv$-functors between $\Bv$-categories are the same things as $\V$-functors between $\V$-categories. In this paradigm where $\M$ is to be viewed as a hom-matrix, we can give the following description of the definition of $\Bc$-functors. The underlying function is the action of the functor on objects, and the $2$-cells $f_{aa'}$ encode the action of the functor on arrows, mapping an arrow between $a$ and $a'$ to an arrow between their respective images. The coherence conditions respectively state that the functor respects compositions and identities.
\end{ex}

\begin{ex}[Morphisms of presheaves]
For any topological space $X$, any morphism $f$ between two presheaves on $X$ yields a $\R(X)$-functor between the corresponding $\R(X)$ categories. In this view, where $\M$ is the ``local equality'' between sections of presheaves, the family of $2$-cells $f_{aa'}$ express the fact that the images of the sections must be equal at least where the sections are. Because $\R(X)$ is a quantaloid, the $2$-structure is trivial (posetal), and all coherence conditions between $2$-cells will always be verified. 
\end{ex}

\begin{ex}[Metric maps]\label{realnumbersquantale}
Consider $\Bc$ to be the one-object locally ordered category $\overline{\Rb}_+$, whose $1$-cells are positive real numbers together with $0$ and $\infty$: $\Ob(\Bc) = [0,\infty]$ and where the $2$-cells are given by the \textit{reverse} order; therefore $0$ is the greatest element. It is well-known~\cite{Lawvere73} that $\Bc$-categories are \textbf{generalized metric spaces}, i.e. metric space with possibly infinite distance between points. Let $E$ and $F$ be two such a $\Bc$-categories; then a $\Bc$-functor $E \to F$ is a function which needs to satisfy $d(f(a),f(b)) \leq d(a,b)$, i.e. a function which is a \textbf{metric map}.
\end{ex}

\begin{prop}
$\Bc$-categories and $\Bc$-functors between them together form a category $\Cat(\Bc)$.
\end{prop}

\begin{proof}
The identity $\Bc$-functor on a $\Bc$-category $(\M,A)$ is given by $(\id_A, (1_{\M(a,b)})_{a,b \in A})$; it is easily checked that it is indeed a $\Bc$-functor. Now for any diagram in $\Cat(\Bc)$ as follows:
\[
\begin{tikzcd}
(\M,A)
\arrow[r,"f"]
& (\N,C)
\arrow[r,"g"]
& (\Pb,D)
\end{tikzcd}
\]

Define the composite $\Bc$-functor as $(gf,(g_{f(a)f(b)} \bullet f_{ab})_{a,b \in A})$;  it is indeed a $\Bc$-functor, as we have, for all $a_1,a_2,a_3 \in A$:
\begin{align*}
(gf)_{a_1a_2} \bullet \iota_{a_1a_2a_3}     &= g_{f(a_1)f(a_2)} \bullet f_{a_1a_2} \bullet \iota_{a_1a_2a_3}\\ 
                                            &= g_{f(a_1)f(a_2)} \bullet \iota_{f(a_1)f(a_2)f(a_3)} \bullet (f_{a_1a_2} * f_{a_2a_3})\\ 
                                            &= \iota_{gf(a_1)gf(a_2)gf(a_3)} \bullet (g_{f(a_1)f(a_2)}*g_{f(a_2)f(a_3)}) \bullet (f_{a_1a_2} * f_{a_2a_3})\\ 
                                            &= \iota_{gf(a_1)gf(a_2)gf(a_3)} \bullet ((g_{f(a_1)f(a_2)} \bullet f_{a_1a_2}) * (g_{f(a_2)f(a_3)} \bullet f_{a_2a_3}))\\ 
                                            &= \iota_{gf(a_1)gf(a_2)gf(a_3)} \bullet ((gf)_{a_1a_2} * (gf)_{a_2a_3})
\end{align*}
Now for all $a \in A$, we have:
\begin{align*}
(gf)_{aa} \bullet \rho_a   &= g_{f(a)f(a)} \bullet f_{aa} \bullet \rho_a\\ 
&= g_{f(a)f(a)}\bullet \rho_{f(a)}\\ 
&= \rho_{gf(a)}
\end{align*}
To conclude, all we have left to do is to mention that it is obvious that the composite of any $\Bc$-functor by a composable identity $\Bc$-functor yields back the original $\Bc$-functor, i.e. that the identity $\Bc$-functors are well indeed neutral elements for the composition of $\Bc$-functors.
\end{proof}

\begin{deft}
Suppose that the bicategory $\Bc$ is endowed with an involution. A $\Bc$-category $(\M,A)$ is said to be \textbf{symmetric} if: 
\begin{enumerate}
\itb For all $a,a' \in A$, $\M(a,a') = \M(a',a)^{\circ}$. 
\itb For all $a \in A$, $\rho_a^{\circ} = \rho_a$.
\itb For all $a,b,c \in A$, $\iota_{abc}^{\circ} = \iota_{cba}$.
\end{enumerate}
\end{deft}

\begin{ex}[Symmetric monoidal categories]
Let $\V$ be a symmetric monoidal category. Then $\Bv$ is naturally endowed with an involution, and symmetric $\Bv$-categories correspond exactly to $\V$-enriched categories.
\end{ex}

\begin{ex}[Sections of a presheaf]
For any presheaf $F$ over a topological space, the quantaloid $\R(X)$ is naturally endowed with an involution as follows: for any open subsets $U$ and $V$ of $X$ and any ``arrow'' $W: U \to V$, i.e. $W \leq U\wedge V$, define $W^{\circ}$ as $W: V \to U$, since the ``meet'' operation is commutative. The $\R(X)$-category $\si F$ is then symmetric: the greatest open subset of $X$ on which two section $s$ and $t$ are equal does not depend on the order in which we consider $s$ and $t$.
\end{ex}

\begin{deft}
We denote by $\Cat_{\s}(\Bc)$ the full subcategory of $\Cat(\Bc)$ whose objects are symmetric $\Bc$-categories.
\end{deft}

\subsection{Distributors}\label{sec:distributors}

Distributors, also sometimes called \textit{modules}~\cite{StreetEnrichment} or \textit{profunctors}, are to $\Bc$-functors what relations are to functions. More precisely, they equip the category $\Cat(\Bc)$ with proarrows, meaning that any $\Bc$-functor yields a pair of adjoint distributors.

\begin{deft}
Let $(\M,A)$ and $(\N,C)$ be two $\Bc$-categories. A \textbf{distributor} $\phi$ from $(\M,A)$ to $(\N,C)$ is the data of:
\begin{enumerate}
\itb A ``rectangular'' $\Bc$-matrix $\phi: A \to C$,
\itb For all $a,a' \in A$, $c,c' \in C$, a $2$-cell of \textbf{double action}: $$\delta_{c,c',a',a}: \N(c,c')\phi(c',a')\M(a',a) \To \phi(c,a)$$
\end{enumerate}
Satisfying the following two coherence conditions: 
\begin{enumerate}
\itb \textbf{Unitality}: For any $c \in C$, $a \in A$: $$\delta_{ccaa} \bullet (\rho_c^{\N} * 1_{\psi(c,a)} * \rho_a^{\N}) = 1_{\psi(c,a)}$$
\itb \textbf{Associativity of $\delta$}: For any $c_1,c_2,c_3 \in C$, and any $a_1,a_2,a_3 \in A$: 
\[\delta_{c_1c_2a_2a_1} \bullet (1_{\N(c_1,c_2)} * \delta_{c_2c_3a_3a_2} * 1_{\M(a_2,a_1)}) = \delta_{c_1c_3a_3a_1} \bullet (\iota_{c_1c_2c_3}^{\N} * 1_{\psi(c_3,a_3)} * \iota_{a_3a_2a_1}^{\M})\]
\end{enumerate}
\end{deft}

The coherence condition expressing the associativity of $\delta$ could also be represented as the commutativity of a diagram of shape: 
\[
\begin{tikzcd}[row sep = 50, column sep = 20]
\N(c_1,c_2)\psi(c_2,a_2)\M(a_2,a_1)
\arrow[r,""]
& \psi(c_1,a_1)\\
\N(c_1,c_2)\N(c_2,c_3)\psi(c_3,a_3)\M(a_3,a_2)\M(a_2,a_1)
\arrow[u,""]
\arrow[r, ""]
& \N(c_1,c_3)\psi(c_3,a_3)\M(a_3,a_1)
\arrow[u,""]
\end{tikzcd}
\]

\begin{ex}[Representable distributors]
Let $f: (\M,A) \to (\N,C)$ be a $\Bc$-functor. Then $f$ defines two distributors:
\begin{align*}
f_!: (\M,A) &\to (\N,C)\\
(c,a) &\mapsto \N(c,f(a))
\end{align*} 
\begin{align*}
f^!: (\N,C) &\to (\M,A)\\
(a,c) &\mapsto \N(f(a),c)
\end{align*} 
In the bicategory $\Dist(\Bc)$, we have an adjunction $f_! \dashv f^!$: this property expresses the fact that $\Dist(\Bc)$ provides $\Cat(\Bc)$ with the structure of an equipment, hence the fact that distributors can also be sometimes called ``profunctors''. For such a $f$, we can express the family $\delta$ of double action $2$-cells for $f_!$ and $f^!$ by: 
\[\delta_{c_1c_2a_2a_1}^{f_!}:= \iota^{\N}_{c_1f(a_2)f(a_1)} \bullet (\iota^{\N}_{c_1c_2f(a_2)} * f_{a_2a_1})\]
\[\delta^{f^!}_{a_1a_2c_2c_1}:= \iota_{f(a_1)f(a_2)c_1}^{\N} \bullet (f_{a_1a_2} * \iota_{f(a_2)c_2c_1}^{\N})\]
Taking $f$ to be the identity $\Bc$-functor yields the fact that $\M$ is a distributor over the $\Bc$-category $(\M,A)$. In this case, we can easily express the family $\delta$ as: 
\[\delta_{b_1b_2a_2a_1}:= \iota_{b_1a_2a_1} \bullet (\iota_{b_1b_2a_2} * 1_{\M(a_2,a_1)}) = \iota_{b_1b_2a_1} \bullet (1_{\M(b_1,b_2)} * \iota_{b_2a_2a_1})\]
\end{ex}

\begin{ex}[$\Hom$ as a distributor]
Let there be a category $\C$,i.e. a category enriched over the cartesian monoidal category $\Set$, hence a category enriched over the corresponding delooping bicategory $\Bc_{\Set}$. As a $\Bc_{\Set}$-category, $\C$ is given by $(\Hom, \Ob(\C))$; and as we saw above, $\Hom$ is, as the structural matrix of a $\Bc_{\Set}$-category, a distributor. Note that for any $x: a_1\to a_2$ in $\C$, we can define for any $b \in \Ob(\C)$ a functor $\Hom(b,x): \Hom(b,a_1) \To \Hom(b,a_2)$ of precomposition by $f$. This construction can be extended to general case of a bicategory $\Bc$. Let $(\M,A)$ and $(\N,C)$ be two $\Bc$-categories, consider two elements $a_1,a_2 \in A$ of same type $ta$, and consider a $2$-cell $x: \id_{ta} \To \M(a_1,a_2)$ in $\Bc$. Let $\psi: (\M,A) \to (\N,C)$ and $\phi: (\N,C) \to (\M,A)$ be two distributors. Then we define the following $2$-cells, for all $c \in C$:
\begin{enumerate}
\itb $\psi(c,x): \psi(c,a_1) \To \psi(c,a_2)$, given by: \[\psi(c,x) = \delta_{cca_1a_2}^{\psi} \bullet (\rho_c^{\N} * 1_{\psi(c,a_1)} * x)\]
\itb $\phi(x,c): \phi(a_2,c) \To \phi(a_1,c)$, given by: \[\phi(x,c) = \delta_{a_1a_2cc}^{\phi} \bullet (x * 1_{\phi(a_2,c)} * \rho_c^{\N})\]
\end{enumerate}

We give here the diagram describing the construction of $\psi(c,x)$: 
\[
\begin{tikzcd}[sep=huge]
tc
& tc
\arrow[l,"{\id}"' {name={11}}]
& ta_1
\arrow[l,"{\psi(c,a_1)}"' {name={12}}]
& ta_2
\arrow[l,"{\id}"' {name={13}}] 
\\ tc
& tc
\arrow[l,"{\N(c,c)}"' {name={21}}]
& ta_1
\arrow[l,"{\psi(c,a_1)}"' {name={22}}]
& ta_2
\arrow[l,"{\M(a_1,a_2)}"' {name={23}}]
\arrow[lll, "" {name={2}}, phantom]
\\ tc
&
&
& ta_2
\arrow[lll,"{\psi(c,a_2)}"' {name={3}}]
\arrow[Rightarrow, from=11, to=21, "{\rho^{\N}_c}", shorten >=3, shorten <= 5]
\arrow[Rightarrow, from=12, to=22, "1_{\psi(c,a_1)}", shorten >=3, shorten <=5]
\arrow[Rightarrow, from=13, to=23, "x", shorten >=3, shorten <=5]
\arrow[Rightarrow, from=2, to=3, "{\delta_{cca_1a_2}^{\psi}}", shorten >=3, shorten <= 5]
\end{tikzcd}
\]
In the case of usual categories, the bicategory $\Bc_{\Set}$ only has one $0$-cell so there are no questions of ``being of the same type'', and a generalised element $x: \id \To \Hom(a_1,a_2)$ corresponds directly to a morphism $x: a_1 \to a_2$ in $\C$. In this enriched category $\iota$ is the composition functor and $\rho$ is the choice of the identity morphism on any object of $\C$. It follows that $\delta$ is simply the composition of three composable arrows, and we get that for all $b \in \Ob(\C)$, $\Hom(b,x)$ (which is a $2$-cell in $\Bc_{\Set}$, i.e. an arrow in $\Set$, i.e. a function) simply corresponds to the precomposition with $x$, by taking the formula $\Hom(b,x) = \delta\bullet (\rho * 1_{\Hom(b,a_1)} * x)$ (the horizontal composiiton corresponding to the (reversed) usual composition in $\C$).
\end{ex}

\begin{ex}[Relations between presheaves]
Let $X$ be a topological space, consider $F,G$ two sheaves on $X$ and denote by $(\M_F,\si F)$, $(\M_G,\si G)$ the corresponding $\R(X)$-categories. Consider a relation $\phi$ between $F$ and $G$, meaning for any open subset of $U$, a relation between the sets $F(U)$ and $G(U)$ which is compatible with the restriction. Then $\phi$ defines a distributor between $(\M_F,\si F)$, $(\M_G, \si G)$ by sending a pair of sections $(f,g)$ to the greatest open subset on which they are equal up to this relation.
\end{ex}

\begin{ex}[Arrows of $\Bc$]
Let $f: x \to y$ be an arrow in $\Bc$; then $f$ defines a distributors between the $\Bc$-categories $(\id_x,x)$ and $(\id_y,y)$. In this particularly simple case, we have $\delta=1_f$, and all the coherence conditions are trivially satisfied. This assignation defines a faithful functor $\Bc \to \Dist(\Bc)$.
\end{ex}

\begin{ex}[$\V$-profunctors]
Let $\V$ be a symmetric monoidal closed category. Then there is a monoidal structure on $\Cat(\V)$; for any $\V$-categories $\C$ and $\D$, the objects of $\C \otimes \D$ are $\Ob(\C) \times \Ob(\D)$, and $(\C \otimes \D)((c_1,d_1),(c_2,d_2)) = \C(c_1,c_2) \otimes \D(d_1,d_2)$ (see 6.2.9 of~\cite{HandCat2} for more details). Then a distributor between $\C$ in $\D$, i.e. an arrow in $\Dist(\Bv)$, is the same thing as a $\V$-functor $\D^{\op} \otimes \C \to \V$. Indeed, such a $\V$-functor $f$ is defined through the existence for all $(c_1,d_1),(c_2,d_2) \in \Ob(\D^{\op} \otimes \C)$ of an arrow in $\V$: $f_{d_2d_1c_1c_2}: (\D^{\op} \otimes \C)((d_1,c_1),(d_2,c_2)) \to [\phi(d_1,c_1),\phi(d_2,c_2)]$, where $[-,-]$ is the internal hom; using its universal property we get an arrow $\D(d_2,d_1) \otimes \phi(d_1,c_1) \otimes \C(c_1,c_2) \to \phi(d_2,c_2)$, and the coherence conditions for the $\V$-functoriality of $f$ coincide with the coherence conditions for it to define a distributor.
\end{ex}

\begin{ex}[$L^1$ metric]
Let $\Bc = \overline{\Rb}_+$ be the one-object locally ordered bicategory of real numbers (see Example~\ref{realnumbersquantale} for details), and let $E$ and $F$ be two $\Bc$-categories. Then a distributor $\phi: E \to F$ is a function $E \times F \to \overline{\Rb}_+$ which is so that $|\phi(c_1,a_1) - \phi(c_2,a_2)| \leq d(c_1,c_2) + d(a_1,a_2)$; now consider the product space $E \times F$ together with the $L^1$ metric $d((c_1,a_1),(c_2,a_2)) = d(c_1,c_2)+d(a_1,a_2)$, then distributors $E \to F$ are exactly metric maps (i.e. $1$-lipschitz continuous functions) from this product $E \times F$ to $\overline{\Rb}_+$ with respect to this metric.
\end{ex}

The goal of the rest of this section will now be to give the definition of the bicategory $\Dist(\Bc)$. In all that follows, we consider three $\Bc$-categories $(\M,A)$, $(\N,C)$, $(\Pb,D)$, as well as two distributors $\phi$ and $\psi$ as in the following diagram:

\[\begin{tikzcd}
(\M,A)
\arrow[r,"{\phi}"]
& (\N,C)
\arrow[r,"{\psi}"]
& (\Pb,D)
\end{tikzcd}\]

\begin{prop}
Suppose that $\Bc$ is endowed with an involution. Let $\phi: (\M,A) \to (\N,C)$ be a distributor between two \textit{symmetric} $\Bc$-categories~; then we define $\phi^{\circ}: (\N,C) \to (\M,A)$ by $\phi^{\circ}(a,c) = \phi(c,a)^{\circ}$~; then $\phi^{\circ}$ is a distributor.
\end{prop}

\begin{proof}
Let $a_1,a_2 \in A$ and $c_1,c_2 \in C$ be elements of the considered $\Bc$-categories. Define $\delta^{\psi^{\circ}}_{a_1a_2c_2c_1} = (\delta^{\psi}_{c_1c_2a_2a_1})^{\circ}$; this is a $2$-cell going from $$(\N(c_1,c_2)\phi(c_2,a_2)\M(a_2,a_1))^{\circ} = \M(a_1,a_2)\phi^{\circ}(a_2,c_2)\N(c_2,c_1)$$ to $\phi^{\circ}(a_1,c_1)$. Then the distributor conditions are a direct consequence of the coherence properties of the symmetry of $\M$ and $\N$: 
\begin{align*}
\delta^{\phi^{\circ}}_{aacc} \bullet (\rho_a * 1_{\phi^{\circ}(a,c)} * \rho_c) &= (\delta^{\phi}_{ccaa})^{\circ} \bullet (\rho_a^{\circ} * 1_{\phi(c,a)} * \rho_c^{\circ})\\ 
&= (\delta^{\phi}_{ccaa} \bullet (\rho_c * 1_{\phi(c,a)} * \rho_a))^{\circ}\\ 
&= 1_{\phi(c,a)}^{\circ}\\ 
&= 1_{\phi^{\circ}(a,c)}
\end{align*}

The other condition is treated in the same way.
\end{proof}

\begin{coro}
Suppose that $\Bc$ is endowed with an involution. Then this involution extends to $\Dist_{\s}(\Bc)$.
\end{coro}

From now on we describe a composition operation for distributors; the definition is already present, in a less detailed form, in section 3 of~\cite{StreetEnrichment}. The composition of distributors generalizes the relational composition of relations, and necessitates the use of coends. We give the proof, lacking from~\cite{StreetEnrichment}, that a composite of two distributors is still a distributor.

\begin{deft}\label{CompDistrib}
Let $a,d$ be elements of $A$ and $D$ respectively. A \textbf{cowedge of $\psi$ and $\phi$ at $(d,a)$ over $\N$} is the data of:
\begin{enumerate}
\itb A $1$-cell $w: ta \to td$ in $\Bc$.
\itb For all $c \in C$, a $2$-cell $e_c: \psi(d,c)\phi(c,a) \To w$. 
\end{enumerate}
This data must satisfy the properties expressed by the commutativity of each of the following diagrams, for all $c_1,c_2 \in C$:
\[
\begin{tikzcd}[row sep=50, column sep=100]
{\psi(d,c_1)\phi(c_1,a)}
\arrow[r,"e_{c_1}"]
& {w}\\ 
{\psi(d,c_1)\N(c_1,c_2)\phi(c_2,a)}
\arrow[u,"{1 * (\delta \bullet (1 * \rho))}" description]
\arrow[r,"{(\delta \bullet (\rho * 1)) * 1}"']
& {\psi(d,c_2)\phi(c_2,a)}
\arrow[u,"{e_{c_2}}"']
\end{tikzcd}
\]\\

The \textbf{composite} of $\phi$ and $\psi$ is defined pointwise (that is, for each $a \in A$, $d\in D$) as the colimit of all cowedge diagrams of $\phi$ and $\psi$ at $(d,a)$ over $\N$; it is the universal cowedge of $\phi$ and $\psi$ at $(d,a)$. We write:
\[
(\psi\phi)(d,a) = \int^{c: C}_{\N} \psi(d,c)\phi(c,a)
\]
\end{deft}

\begin{prop}
The composite of two distributors is a distributor.
\end{prop}

\begin{proof}
First of all, we must find a splitting $2$-cell $\delta^{\psi\phi}$, for each $d_1,d_2 \in D$, $a_1,a_2 \in A$: 
\[\delta_{d_1d_2a_2a_1}^{\psi\phi}: \Pb(d_1,d_2)(\psi\phi)(d_2,a_2)\M(a_2,a_1) \To (\psi\phi)(d_1,a_1)\]
Because $\Bc$ is closed, pre- and post-composition functors preserve colimits, hence:
\begin{align*}
\Pb(d_1,d_2)(\psi\phi)(d_2,a_2)\M(a_2,a_1) &= \Pb(d_1,d_2)\big(\int^{c:C}_{\N} \psi(d_2,c)\phi(c,a_2)\big)\M(a_2,a_1)\\
&= \int^{c:C}_{\N} \Pb(d_1,d_2)\psi(d_2,c)\phi(c,a_2)\M(a_2,a_1)
\end{align*}
Thus giving a morphism $\Pb(d_1,d_2)(\psi\phi)(d_2,a_2)\M(a_2,a_1) \To (\psi\phi)(d_1,a_1)$, is the same as giving a family of $2$-cells $f_c: \Pb(d_1,d_2)\psi(d_2,c)\phi(c,a_2)\M(a_2,a_1) \To (\psi\phi)(d_1,a_1)$, for $c \in C$, being compatible in the sense defined above. We are going to use the fact that $(\psi\phi)(d_1,a_1)$ is itself defined as a colimit: write $e_c: \psi(d_1,c)\phi(c,a_1) \To (\psi\phi)(d_1,a_1)$ for the injection of the colimiting cocone, for all $c \in \Ob(\C)$, and define:
\[f_c:= e_c \bullet ((\delta_{d_1d_2cc}^{\psi} \bullet (1_{\Pb(d_1,d_2)\psi(d_2,c)} * \rho_c^{\N})) * (\delta^{\phi}_{cca_2a_1} \bullet (\rho^{\N}_c * 1_{\phi(c,a_2)\M(a_2,a_1)})))\]
Now to prove that this family is compatible, we must show for all $c_1,c_2 \in C$:

\begin{center}
$f_{c_1} \bullet (1_{\Pb(d_1,d_2)\psi(d_2,c_1)} * (\delta_{c_1c_2a_2a_2} \bullet (1_{\N(c_1,c_2)\phi(c_2,a_2)} * \rho_{a_2}^{\M})) * 1_{\M(a_2,a_1)})$\\
$=$\\ 
$f_{c_2} \bullet (1_{\Pb(d_1,d_2)} * (\delta_{d_2d_2c_1c_2} \bullet (\rho_{d_2}^{\Pb} * 1_{\psi(d_2,c_1)\N(c_1,c_2)})) * 1_{\phi(c_2,a_2)\M(a_2,a_1)})$
\end{center}
Define $A$ as the expression $f_{c_1} \bullet \cdots$, then we have the following; we use simplified but unambiguous notations to lighten the computations:
\begin{eqnarray*}
A &=&e_{c_1} \bullet ((\delta_{d_1d_2c_1c_1}^{\psi} \bullet (1_{\Pb_{d_1d_2}\psi_{d_2c_1}} * \rho_{c_1}^{\N})) * (\delta^{\phi}_{c_1c_1a_2a_1} \bullet (\rho^{\N}_{c_1} * 1_{\phi_{c_1a_2}\M_{a_2a_1}}))) \bullet \\ 
& &(1_{\Pb_{d_1d_2}\psi_{d_2c_1}} * (\delta_{c_1c_2a_2a_2} \bullet (1_{\N_{c_1c_2}\phi_{c_2a_2}} * \rho_{a_2}^{\M})) * 1_{\M_{a_2a_1}})\\
&=& e_{c_1} \bullet ((\delta^{\psi}_{d_1d_2c_1c_1} \bullet (1_{\Pb_{d_1d_2}\psi_{d_2,c_1}} * \rho^{\N}_{c_1}) \bullet 1_{\Pb_{d_1d_2}\psi_{d_2c_1}}) * (\delta^{\phi}_{c_1c_1a_2a_1} \bullet \\ 
& & (\rho_{c_1}^{\N} * 1_{\phi_{c_1a_2}\M_{a_2a_1}}) \bullet ((\delta_{c_1c_2a_2a_2}^{\phi} \bullet (1_{\N_{c_1c_2}\phi_{c_2a_2}} * \rho_{a_2}^{\M})) * 1_{\M_{a_2a_1}})))\\
&=& e_{c_1} \bullet ((\delta \bullet (1*\rho))  * (\delta \bullet (\rho * 1_{\phi_{c_1a_2}\M_{a_2a_1}}) \bullet (1_{\id_{tc_1}} * ( \delta \bullet (1 * \rho)) * 1_{\M_{a_2a_1}})))\\
&=& e_{c_1} \bullet ((\delta \bullet (1 * \rho)) * (\delta \bullet (\rho * (\delta_{c_1c_2a_2a_2}^{\phi} \bullet (1_{\N_{c_1c_2}\phi_{c_2a_2}} * \rho_{a_2}^{\M})) * 1_{\M_{a_2a_1}})))
\end{eqnarray*}
Here we can use the distributor properties:
\begin{align*}
&\delta_{c_1c_1a_2a_1}^{\phi} \bullet (\rho_{c_1}^{\N} * (\delta_{c_1c_2a_2a_2}^{\phi} \bullet (1_{\N_{c_1c_2}\phi_{c_2a_2}} * \rho_{a_2}^{\M})) * 1_{\M_{a_2a_1}})\\
&= \delta_{c_1c_1a_2a_1}^{\phi} \bullet (1_{\N(c_1c_1)} * \delta_{c_1c_2a_2a_2} * 1_{\M_{a_2a_1}}) \bullet (\rho_{c_1}^{\N} * 1_{\N_{c_1c_2}\phi_{c_2a_2}} * \rho_{a_2}^{\M} * 1_{\M_{a_2a_1}})\\
&= \delta_{c_1c_2a_2a_1}^{\phi} \bullet (\iota_{c_1c_2c_2}^{\N} * 1_{\phi_{c_2a_2}} * \iota^{\M}_{a_2a_1a_1}) \bullet (\rho_{c_1}^{\N} * 1_{\N_{c_1c_2}\phi_{c_2a_2}} * \rho_{a_2}^{\M} * 1_{\M_{a_2a_1}}) \\
&= \delta_{c_1c_2a_2a_1}^{\phi} \bullet ((\iota_{c_1c_2c_2}^{\N} \bullet (\rho_{c_1}^{\N} * 1_{\N_{c_1c_2}})) * 1_{\phi_{c_2a_2}} * (\iota_{a_2a_1a_1}^{\M} \bullet (\rho_{a_2}^{\M} * 1_{\M_{a_2a_1}})))\\
&= \delta_{c_1c_2a_2a_1}^{\phi} \bullet 1_{\N_{c_1c_2}\phi_{c_2a_2}\M_{a_2a_1}}\\
&= \delta_{c_1c_2a_2a_1}^{\phi}
\end{align*}
Therefore:
\begin{align*}
A&= e_{c_1} \bullet ((\delta^{\psi}_{d_1d_2c_1c_1} \bullet (1_{\Pb(d_1,d_2)\psi(d_2,c_1)} * \rho_{c_1}^{\N})) * \delta_{c_1c_2a_2a_1}^{\phi})
\end{align*}
Now we express $\delta_{c_1c_2a_2a_1}^{\phi}$ in a way which ``goes through $\phi(c_2,a_1)$'', so that we can use the fact that we know that $e_{c_1}$ and $e_{c_2}$ have a cowedge compatibility condition. Consider the following expression $B$:
\begin{align*}
B &= \delta_{c_1c_2a_1a_1}^{\phi} \bullet (1_{\N_{c_1c_2}} * \delta_{c_2c_2a_2a_1}^{\phi} * 1_{\M_{a_1a_1}}) \bullet (1_{\N_{c_1c_2}} * \rho_{c_2}^{\N} * 1_{\phi_{c_2a_2}\M_{a_2a_1}} * \rho_{a_1}^{\M})\\
&= \delta_{c_1c_2a_2a_1}^{\phi} \bullet (\iota_{c_1c_2c_2} * 1_{\phi_{c_2a_2}} * \iota_{a_2a_1a_1}) \bullet (1_{\N_{c_1c_2}} * \rho_{c_2}^{\N} * 1_{\phi_{c_2a_2}\M_{a_2a_1}} * \rho_{a_1}^{\M})\\
&= \delta_{c_1c_2a_1a_2}^{\phi}
\end{align*}
We reformulate $B$ so that we can use the interchange law later:
\begin{align*}
B &= \delta_{c_1c_2a_1a_1}^{\phi} \bullet (1_{\N(c_1,c_2)\phi(c_2,a_1)} * \rho_{a_1}^{\M}) \bullet (\delta_{c_2c_2a_2a_1}^{\phi} \bullet (\rho_{c_2}^{\N} * 1_{\phi(c_2,a_2)\M(a_2,a_1)}))
\end{align*}
Plugging that into our expression of $A$, we can use the interchange law to get:
\begin{align*}
A &= e_{c_1} \bullet (1_{\psi(d_1,c_1)} * (\delta_{c_1c_2a_1a_1}^{\phi} \bullet (1_{\N(c_1,c_2)\phi(c_2,a_1)} * \rho_{a_1}^{\M}))) \bullet C
\end{align*}
Where:
\[C = (\delta_{d_1d_2c_1c_1}^{\psi} \bullet (1_{\Pb_{d_1d_2}\psi_{d_2c_1}} * \rho_{c_1}^{\N})) * 1_{\N_{c_1c_2}} * (\delta_{c_2c_2a_2a_1}^{\phi} \bullet (\rho_{c_2}^{\N} * 1_{\phi_{c_2a_2}\M_{a_2a_1}}))\]
Using the compatibility conditions between $e_{c_1}$ and $e_{c_2}$, we therefore have: 
\[A = e_{c_2} \bullet ((\delta_{d_1d_1c_1c_2}^{\psi} \bullet (\rho_{d_1}^{\Pb} * 1_{\psi(d_1,c_1)\N(c_1,c_2)})) * 1_{\phi(c_2,a_1)}) \bullet C\]
Again, we can use the interchange law, separating relatively to horizontal composition between $1_{\N(c_1,c_2)}$ and $\rho_{c_2}^{\N}$ to get:
\begin{align*}
A&= e_{c_2} \bullet (D * (\delta_{c_2c_2a_2a_1}^{\phi} \bullet (\rho_{c_2}^{\N} * 1_{\phi(c_2,a_2)\M(a_2,a_1)})))
\end{align*}
With: 
\[D = \delta_{d_1d_1c_1c_2}^{\psi} \bullet (\rho_{d_1}^{\Pb} * \delta_{d_1d_2c_1c_1}^{\psi} * 1_{\N(c_1c_2)}) \bullet (1_{\Pb(d_1,d_2)\psi(d_2,c_1)} * \rho_{c_1}^{\N} * 1_{\N(c_1,c_2)})\]
A quick computation can simplify $D$, as we did before with $B$:
\begin{align*}
D&= \delta_{d_1d_1c_1c_2}^{\psi} \bullet (1_{\Pb_{d_1d_1}} * \delta_{d_1d_2c_1c_1}^{\psi} * 1_{\N_{c_1c_2}}) \bullet (\rho_{d_1}^{\Pb} * 1_{\Pb_{d_1d_2}\psi_{d_2c_1}} * \rho * 1_{\N_{c_1c_2}})\\
&= \delta_{d_1d_2c_1c2}^{\psi} \bullet (\iota_{d_1d_1d_2} * 1_{\psi_{d_2c_1}} * \iota_{c_1c_1c_2})\bullet (\rho_{d_1} * 1_{\Pb_{d_1d_2}\psi_{d_2c_1}} * \rho_{c_1} * 1_{\N_{c_1c_2}})\\
&= \delta_{d_1d_2c_1c_2}^{\psi}
\end{align*}
We thus get:
\[A =  e_{c_2} \bullet (\delta_{d_1d_2c_1c_2}^{\psi} * (\delta_{c_2c_2a_2a_1}^{\phi} \bullet (\rho_{c_2}^{\N} * 1_{\phi(c_2,a_2)\M(a_2,a_1)})))\]
Now define:
\[A' = f_{c_2} \bullet (1_{\Pb(d_1,d_2)} * (\delta_{d_2d_2c_1c_2} \bullet (\rho_{d_2}^{\Pb} * 1_{\psi(d_2,c_1)\N(c_1,c_2)})) * 1_{\phi(c_2,a_2)\M(a_2,a_1)})\]
We can simplify it to get $A$. First start by separating horizontally between the two instances of $\rho_{c_2}$, we get (we drop the indices under identities for brevity):
\[A' = e_{c_2} \bullet ((\delta_{d_1d_2c_2c_2}^{\psi} \bullet (1 * \delta_{d_2d_2c_1c_2}^{\psi} * 1) \bullet (1 * \rho_{d_2}^{\Pb} * 1 * \rho_{c_2}^{\N})) * (\delta_{d_2d_2c_1c_2} \bullet (\rho_{d_2}^{\Pb} * 1))) \]
We recognize the pattern on the left-hand side of the expression: like $B$ and $D$ before, it simplifies, into $\delta_{d_1d_2c_1c_2}^{\psi}$. Now the right side is the same as in $A$, and therefore we have $A = A'$, so that conclude this part of the proof. Therefore the family $(f_c)_{c\in \Ob(\C)}$ defines a cowedge, and the colimiting universal property enjoyed by $\Pb(d_1,d_2)(\psi\phi)(d_2,a_2)\M(a_2,a_1)$ yields a $2$-cell:
$$\delta_{d_1d_2a_2a_1}^{\psi\phi}: \Pb(d_1,d_2)(\psi\phi)(d_2,a_2)\M(a_2,a_1) \To (\psi\phi)(d_1,a_1)$$\\ 

Now we must prove that this $\delta$ satisfies the distributor conditions. First for the identity condition, note that for all $a \in A$, $c \in C$, $d \in D$, we have the following commutative diagram: 

\[
\begin{tikzcd}[row sep = 50, column sep = 120]
\psi(d,c)\phi(c,a)
\arrow[r,"e_c"description]
& (\psi\phi)(d,a) \\ 
\Pb(d,d)\psi(d,c)\phi(c,a)\M(a,a) 
\arrow[u,"(\delta^{\psi} * \delta^{\phi}) \bullet (1_{\psi(d,c)} * \rho^{\N}_c * \rho^{\N}_c * 1_{\phi(c,a)})"description]
\arrow[r,"1_{\Pb(d,d)} * e_c * 1_{\M(a,a)}"description]
& \Pb(d,d)(\psi\phi)(d,a)\M(a,a)
\arrow[u,"\delta_{ddaa}^{\psi\phi}"description]\\
\psi(d,c)\phi(c,a) 
\arrow[u,"\rho^{\Pb}_d * 1_{\psi(d,c)\phi(c,a)} * \rho^{\M}_a"description]
\arrow[r,"e_c"description]
& (\psi\phi)(d,a)
\arrow[u,"\rho_d^{\Pb} * 1_{(\psi\phi)(d,a)} * \rho_a^{\M}"description]
\end{tikzcd}
\]

The commutativity of the upper square is ensured by the very definition of $\delta^{\psi\phi}_{ddaa}$, and that of the bottom square is trivially verified through the interchange law. This means that the whole diagram is commutative; given that all colimits are taken over diagrams of the same shape (namely above the $\psi(d,c_1)\N(c_1,c_2)\phi(c_2,a)$, we can safely assert that the composite $e_c \bullet (\delta^{\psi} * \delta^{\phi}) \bullet (1_{\psi(d,c)} * \rho^{\N}_c * \rho^{\N}_c * 1_{\phi(c,a)}) \bullet (\rho^{\Pb}_d * 1_{\psi(d,c)\phi(c,a)} * \rho^{\M}_a)$ is still compatible and yields a cowedge diagram from each $\psi(d,c)\phi(c,a)$ to $(\psi\phi)(d,a)$. By unicity of the coend, we can therefore deduce that the coend of these cowedges is equal to $\delta_{ddaa}^{\psi\phi} \bullet (\rho_d^{\Pb} * 1_{(\psi\phi)(d,a)} * \rho_a^{\M})$. Now, we have the following computation: 

\begin{align*}
&e_c \bullet (\delta^{\psi}_{ddcc} * \delta^{\phi}_{ccaa}) \bullet (1_{\psi(d,c)} * \rho^{\N}_c * \rho^{\N}_c * 1_{\phi(c,a)}) \bullet (\rho^{\Pb}_d * 1_{\psi(d,c)\phi(c,a)} * \rho^{\M}_a) \\ 
&= e_c \bullet ((\delta_{ddcc}^{\psi} \bullet (\rho_d^{\Pb} * 1_{\psi(d,c)} * \rho_c^{\N})) * (\delta_{ccaa}^{\phi} \bullet (\rho_c^{\N} * 1_{\phi(c,a)} * \rho_a^{\M})))\\
&= e_c
\end{align*}

That proves that all these cowedges from $\psi(d,c)\phi(c,a)$ to $(\psi\phi)(d,a)$ are actually trivial and equal to $e_c$; thus proving by unicity of the colimit that their coend is the identity. Therefore we have: 
\[\delta_{ddaa}^{\psi\phi} \bullet (\rho_d^{\Pb} * 1_{(\psi\phi)(d,a)} * \rho_a^{\M}) \simeq 1_{(\psi\phi)(d,a)}\]

For the second distributor condition we give a similar argument based on the functoriality of the colimit (what we did above amounted to prove that for certain particular choices, $(\colim_i f_i)(\colim_i g_i) = \colim_i (f_ig_i)$). First consider the following diagram, of which both squares are commutative because of the definition of $\delta^{\psi\phi}$.

\noindent
\adjustbox{max width=\textwidth}{
\begin{tikzcd}[row sep = 50, column sep = 50, labels={description}]
\psi_{d_1c}\phi_{ca_1}
\arrow[r,"e_c"]
&  (\psi\phi)_{d_1a_1} \\ 
\Pb_{d_1d_2}\psi_{d_2c}\phi_{ca_2}\M_{a_2a_1}
\arrow[r,"1 * e_c * 1"]
\arrow[u, "(\delta \bullet (1 * \rho)) * (\delta \bullet (\rho * 1))"]
& \Pb_{d_1d_2}(\psi\phi)_{d_2a_2}\M_{a_2a_1}
\arrow[u, "\delta^{\psi\phi}"]\\ 
\Pb_{d_1d_2}\Pb_{d_2d_3}\psi_{d_3c}\phi_{ca_3}\M_{a_3a_2}\M_{a_2a_1}
\arrow[r,"1 * e_c * 1"]
\arrow[u,"1 * (\delta \bullet (1 * \rho)) * (\delta \bullet (\rho * 1)) * 1"]
& \Pb_{d_1d_2}\Pb_{d_2d_3}(\psi\phi)_{d_3a_3}\M_{a_3a_2}\M_{a_2a_1}
\arrow[u,"1 * \delta^{\psi\phi} * 1"]
\end{tikzcd}
}\\

The commutativity of this diagram ensures that we can compute at the ``local'' level; in other words, that $\delta^{\psi\phi} \bullet (1 * \delta^{\psi\phi} * 1)$ is the colimit of the $e_c \bullet (\delta \bullet (1 * \rho)) * (\delta \bullet (\rho * 1)) \bullet (1 * (\delta \bullet (1 * \rho)) * (\delta \bullet (\rho * 1)) * 1)$, for $c$ varying. Now, denoting this term by $A$, we have, using first the interchange law and then the distributor condition on $\psi$ and $\phi$ themselves: 
\begin{eqnarray*}
A &=& e_c \bullet (\delta^{\psi}_{d_1d_2cc} * \delta^{\phi}_{cca_2a_1}) \bullet (1_{\Pb(d_1,d_2)} * \delta^{\psi}_{d_2d_3cc} * 1_{\N(c,c)\N(c,c)} * \delta^{\phi}_{cca_2a_3} * 1_{\M(a_2,a_1)}) \\
& & \bullet \ (1_{\Pb(d_1,d_2)\Pb(d_2,d_3)\psi(d_3,c)} * \rho^{\N}_c * \rho^{\N}_c * \rho^{\N}_c * \rho^{\N}_c * 1_{\phi(c,a_3)\M(a_3,a_2)\M(a_2,a_1)}) \\ 
&=&  e_c \bullet (\delta^{\psi}_{d_1d_3cc} * \delta^{\phi}_{cca_3a_1}) \bullet (\iota_{d_1d_2d_3}^{\Pb} * 1_{\psi(d_3,c)} * \iota^{\N}_{ccc} * \iota^{\N}_{ccc} * 1_{\phi(c,a_3)} * \iota^{\M}_{a_3a_2a_1}) \bullet \\ 
& &(1_{\Pb(d_1,d_2)\Pb(d_2,d_3)\psi(d_3,c)} * \rho^{\N}_c * \rho^{\N}_c * \rho^{\N}_c * \rho^{\N}_c * 1_{\phi(c,a_3)\M(a_3,a_2)\M(a_2,a_1)}) \\ 
&=&  e_c \bullet ((\delta^{\psi}_{d_1d_3cc} \bullet (\iota^{\Pb}_{d_1d_2d_3} * 1_{\psi(d_3,c)} * \rho^{\N}_c)) * (\delta^{\phi}_{cca_3a_1} \bullet (\rho^{\N}_c * 1_{\phi(c,a_3)} * \iota^{\M}_{a_3a_2a_1})))
\end{eqnarray*}

Now that we have reformulated $\delta^{\psi\phi} \bullet (1 * \delta^{\psi\phi} * 1)$ as a colimit of $A$, let us reformulate the other term of the equation, namely $\delta^{\psi\phi} \bullet (\iota * 1 * \iota)$. We have the following diagram, which is commutative; the upper part is commutative because of the definition of $\delta^{\psi\phi}$, and the lower part is trivially commutative.

\noindent
\adjustbox{max width=\textwidth}{
\begin{tikzcd}[row sep = 50, column sep = 50, labels={description}]
\psi_{d_1c}\phi_{ca_1}
\arrow[r,"e_c"]
& (\psi\phi)_{d_1a_1} \\ 
\Pb_{d_1d_3}\psi_{d_3c}\phi_{ca_3}\M_{a_3a_1}
\arrow[r,"1 * e_c * 1"]
\arrow[u, "(\delta \bullet (1 * \rho)) * (\delta \bullet (\rho * 1))"]
& \Pb_{d_1d_3}(\psi\phi)_{d_3a_3}\M_{a_3a_1}
\arrow[u, "\delta^{\psi\phi}"]\\ 
\Pb_{d_1d_2}\Pb_{d_2d_3}\psi_{d_3c}\phi_{ca_3}\M_{a_3a_2}\M_{a_2a_1}
\arrow[r,"1 * e_c * 1"]
\arrow[u,"\iota * 1 * \iota"]
& \Pb_{d_1d_2}\Pb_{d_2d_3}(\psi\phi)_{d_3a_3}\M_{a_3a_2}\M_{a_2a_1}
\arrow[u,"\iota * 1 * \iota"]
\end{tikzcd}
}\\

The commutativity of this diagram (plus the fact that everything is indeed compatible if we take the cowedges over $c_1,c_2 \in C$) proves that $\delta^{\psi\phi} \bullet (\iota * 1 * \iota)$ is the colimit of all terms $B = e_c \bullet ((\delta \bullet (1 * \rho)) * (\delta \bullet (\rho * 1))) \bullet (\iota * 1 * \iota)$. Now an easy application of the interchange law yields:

\begin{eqnarray*}
B &=& e_c \bullet ((\delta^{\psi}_{d_3d_1cc} \bullet (1_{\Pb_{d_1d_3}\psi_{d_3c}} * \rho^{\N}_c)) * (\delta^{\phi}_{cca_3a_1} \bullet (\rho^{\N}_c * 1_{\phi_{ca_3}\M_{a_3a_1}}))) \bullet\\ 
& &(\iota^{\Pb}_{d_1d_2d_3} * 1_{\psi_{d_3c}\phi_{ca_3}} * \iota^{\M}_{a_3a_2a_1})\\
&=& e_c \bullet ((\delta^{\psi}_{d_1d_3cc} \bullet (\iota^{\Pb}_{d_1d_2d_3} * 1_{\psi(d_3,c)} * \rho^{\N}_c)) * (\delta^{\phi}_{cca_3a_1} \bullet (\rho^{\N}_c * 1_{\phi(c,a_3)} * \iota^{\M}_{a_3a_2a_1})))
\end{eqnarray*}

Hence $A = B$, whence their colimits are the same, proving that $\psi\phi$ does indeed satisfy the second distributor condition.
\end{proof}

\begin{prop}\label{Mid}
Let $(\M,A)$ be a $\Bc$-category. Then the distributor $\M: (\M,A) \to (\M,A)$ is an identity for the composition of distributors; in other words, for any distributors $\psi$ and $\phi$ having respectively $(\M,A)$ as domain and codomain, we have $\psi \M = \psi$ and $\M\phi = \M$.
\end{prop}

\begin{proof}
Let $\psi: (\M,A) \to (\N,C)$ be a distributor. We will prove that for any $a \in A$, $c \in C$, $\psi(c,a)$ satisfies the colimit property defining $(\psi\M)(c,a)$. First let $b$ be an element of $A$, then there exist some arrow $f_b:= \delta^{\psi}_{ccba} \bullet (\rho^{\N}_c * 1_{\psi(c,b)\M(b,a)}): \psi(c,b)\M(b,a) \To \psi(c,a)$. We shall show that for any $b_1,b_2 \in A$, the following diagram is commutative: 

\[
\begin{tikzcd}[column sep = 70, row sep = 50, labels={description}]
\psi(c,b_1)\M(b_1,a)
\arrow[r,"f_{b_1}"]
& \psi(c,a)\\
\psi(c,b_1)\M(b_1,b_2)\M(b_2,a)
\arrow[u,"{1 * (\delta \bullet (1 * \rho))}"]
\arrow[r,"{(\delta \bullet (\rho * 1)) * 1}"]
& \psi(c,b_2)\M(b_2,a)
\arrow[u,"f_{b_2}"]
\end{tikzcd}
\]

This will prove that $\psi(c,a)$ is indeed a cowedge, making it a fair candidate for the role of $(\psi\M)(c,a)$. The commutativity of the diagram is a standard use of the distributor property on $\psi$, since $\delta^{\M}$ is given through $\iota$; the up-left side of the diagram is equal to:

\begin{align*}
&\delta^{\psi}_{ccb_1a} \bullet (\rho_c^{\N} * 1_{\psi(c,b_1)\M(b_1,a)}) \bullet (1_{\psi(c,b_1)} * (\delta^{\M}_{b_1b_2aa} \bullet (1_{\M(b_1,b_2)\M(b_2,a)} * \rho_a^{\M}))) \\
&= \delta^{\psi}_{ccb_1a} \bullet (\rho_c^{\N} * 1_{\psi(c,b_1)\M(b_1,a)}) \bullet (1_{\psi(c,b_1)} * \iota^{\M}_{b_1b_2a})\\
&= \delta^{\psi}_{ccb_1a} \bullet (\rho^{\N}_c * 1_{\psi(c,b_1)} * \iota_{b_1b_2a}^{\M})
\end{align*}

And the bottom-right side is:

\begin{align*}
&\delta^{\psi}_{ccb_2a} \bullet (\rho_c^{\N} * 1_{\psi(c,b_2)\M(b_2,a)}) \bullet  ((\delta^{\psi}_{ccb_1b_2} \bullet (\rho^{\N}_c * 1_{\M(b_1,b_2)})) * 1_{\M(b_2,a)})\\
&= \delta^{\psi}_{ccb_2a} \bullet (1_{\N(c,c)} * \delta^{\psi}_{ccb_1b_2} * 1_{\M(b_2,a)}) \bullet (\rho_c^{\N} * \rho_c^{\N} * 1_{\psi(c,b_1)\M(b_1,b_2)\M(b_2,a)})\\
&= \delta^{\psi}_{ccb_1a} \bullet (\iota^{\N}_{ccc} * 1_{\psi(c,b_1)} * \iota^{\M}_{b_1b_2a}) \bullet (\rho^{\N}_c * \rho_c^{\N} * 1_{\psi(c,b_1)\M(b_1,b_2)\M(b_2,a)})\\
&= \delta^{\psi}_{ccb_1a} \bullet (\rho^{\N}_c * 1_{\psi(c,b_1)} * \iota_{b_1b_2a}^{\M})
\end{align*}

This shows that the above diagram is commutative, proving that $\psi(c,a)$ is indeed a cowedge over all $\psi(c,b)\M(b,a)$. We now have to prove that it is a \textit{universal} cowedge. Let there be some other cowedge $k \in \Bc(ta,tc)$ with $f_b: \psi(c,b)\M(b,a) \To k$ for all $b \in A$, and such that for any $b_1,b_2 \in A$, $g_{b_1} \bullet (1_{\psi(c,b_1)} * (\delta^{\M}_{b_1b_2aa} \bullet (1_{\M(b_1,b_2)\M(b_2,a)} * \rho_a^{\M}))) = g_{b_2} \bullet ((\delta^{\psi}_{ccb_1b_2} \bullet (\rho_c^{\N} * 1_{\M(b_1,b_2)})) * 1_{\M(b_2,a)})$. Then we can define $g_a \bullet (1_{\psi(c,a)} * \rho^{\M}_a): \psi(c,a) \To k$. To conclude, we only need to prove that for any $b \in A$, $g_b = g_a \bullet (1_{\psi(c,a)} * \rho_a^{\M}) \bullet f_b$, that is that the following diagram is commutative: 

\[
\begin{tikzcd}[sep=huge,labels={description}]
& k \\
& \psi(c,a) 
\arrow[u,"g_a \bullet (1 * \rho_a)"]\\
\psi(c,b)\M(b,a) 
\arrow[uur,"g_b", bend left=25]
\arrow[ru,"f_b"]
& & \psi(c,a)\M(a,a)
\arrow[uul,"g_a", bend right=25]
\arrow[lu,"f_a"] \\ 
& \psi(c,b)\M(b,a)\M(a,a)
\arrow[ul,"1 * \iota"]
\arrow[ur,"(\delta \bullet (\rho * 1))*1"]
\end{tikzcd}
\]

An easy computation yields that the following square is commutative: 

\[
\begin{tikzcd}[row sep = 40, column sep =70, labels=description]
\psi(c,b)\M(b,a)
\arrow[r,"f_b"]
\arrow[d,"1 * \rho"]
& \psi(c,a) 
\arrow[d,"1 * \rho"]\\
\psi(c,b)\M(b,a)\M(a,a)
\arrow[r,"(\delta \bullet (\rho * 1 ))*1"]
& \psi(c,a)\M(a,a)
\end{tikzcd}
\]

And we have $(1_{\psi(c,b)} * \iota^{\M}_{baa}) \bullet (1_{\psi(c,b)\M(b,a)} * \rho^{\M}_a) = 1_{\psi(c,b),\M(b,a)}$; therefore: 

\begin{align*}
g_b &= g_b \bullet (1_{\psi(c,b)} * \iota^{\M}_{baa}) \bullet (1_{\psi(c,b)\M(b,a)} * \rho^{\M}_a)\\
&= g_a \bullet ((\delta^{\psi}_{ccba} \bullet (\rho^{\N}_c * 1_{\psi(c,b)\M(b,a)})) * 1_{\M(a,a)}) \bullet (1_{\psi(c,b)\M(b,a)} * \rho^{\M}_a)\\
&= g_a \bullet (1_{\psi(c,a)} * \rho_a^{\M}) \bullet f_b
\end{align*}

This shows that indeed $\psi(c,a)$ satisfies the necessary colimit condition and is thus equal to $(\psi\M)(c,a)$. Now this proves that the underlying matrices of $\psi$ and $\psi\M$ are the same; to prove that both are one same distributor we need to prove $\delta^{\psi\M} = \delta^{\psi}$. To prove that, we use the unicity of the coend, by showing that the following diagram is commutative: 

\[
\begin{tikzcd}[row sep = 50, column sep = 100, labels={description}]
\psi_{c_1b}\M_{ba} 
\arrow[r,"{\delta^{\psi} \bullet (\rho * 1)}"]
& \psi_{c_1a_1}\\
\N_{c_1c_2}\psi_{c_2b}\M_{ba_2}\M_{a_2a_1} 
\arrow[u,"(\delta^{\psi} \bullet (1 * \rho)) * \iota"]
\arrow[r,"{1 * (\delta^{\psi} \bullet (\rho * 1)) * 1}"]
& \N_{c_1c_2}\psi_{c_2a_2}\M_{a_2a_1}
\arrow[u,"\delta^{\psi}"]
\end{tikzcd}
\]

Which is only a direct case of the second distributor condition. This shows that $\psi\M = \M$; the other sense works in the same way.

\end{proof}

\begin{coro}
Let $(\M,A)$ be a $\Bc$-category. Then $\M$ is idempotent, i.e. $\M^2 = \M$.
\end{coro}

\begin{coro}
$\Bc$-categories and distributors between them together form a category $\Dist(\Bc)$. If $\Bc$ is endowed with an involution, we denote by $\Dist_{\s}(\Bc)$ the full subcategory of $\Dist(\Bc)$ whose objects are symmetric $\Bc$-categories.
\end{coro}

\begin{deft}
Let $\phi_1, \phi_2: (\M,A) \to (\N,C)$ be two distributors between two $\Bc$-categories. A \textbf{morphism of distributors} from $\phi_1$ to $\phi_2$ is the data for all $c \in C$, $a \in A$, of a $2$-cell $f_{ca}: \phi_1(c,a) \To \phi_2(c,a)$ in $\Bc$, satisfying the following coherence conditions, for all $a,a_1,a_2 \in A$, $c,c_1,c_2 \in C$:
\begin{enumerate}
\item $f_{c_1a} \bullet \delta_{c_1c_2aa}^{\phi_1} \bullet (1_{\N_{c_1c_2}(\phi_1)_{c_2a}} * \rho_a^{\M}) \simeq \delta_{c_1c_2aa}^{\phi_2} \bullet (1_{\N_{c_1c_2}(\phi_2)_{c_2a}} * \rho_a) \bullet (1_{\N_{c_1c_2}} * f_{c_2a})$
\item $f_{ca_2} \bullet \delta_{cca_1a_2}^{\phi_1} \bullet (\rho_c^{\N} * 1_{(\phi_1)_{ca_1}\M_{a_1a_2}}) \simeq \delta_{cca_1a_2}^{\phi_2} \bullet (\rho_c * 1_{(\phi_2)_{ca_1}\M_{a_1a_2}}) \bullet (f_{ca_1} * 1_{\M_{a_1a_2}})$
\end{enumerate}
\end{deft}

\begin{prop}
$\Bc$-categories, distributors between them and morphisms between distributors together form a bicategory that we also denote by $\Dist(\Bc)$. If $\Bc$ is endowed with an involution, then $\Dist_{\s}(\Bc)$ is an involutive bicategory.
\end{prop}

\begin{proof}
We give here the definition for the horizontal composition of morphisms of distributors. Supposes one has the following diagram of distributors:
\[
\begin{tikzcd}[sep=huge]
(\Pb,D)
& (\N,C)
\arrow[l,"\psi_1" {description, name={s1}}]
& (\M,A)
\arrow[l,"\phi_1" {description, name={h1}}]
\\
(\Pb,D)
& (\N,C)
\arrow[l,"\psi_2" {description, name={s2}}]
& (\M,A)
\arrow[l,"\phi_2" {description, name={h2}}]
\arrow[Rightarrow, from={h1}, to={h2}, "f", shorten <= 5, shorten >= 3]
\arrow[Rightarrow, from={s1}, to={s2}, "g", shorten <= 5, shorten >= 3]
\end{tikzcd}
\]
For $a \in A$, $d \in D$, we define the component $(g*f)_{da}$ of the horizontal composite of the morphisms of distributors $f$ and $g$, as the canonical morphism obtained through the consideration of the family of the following cowedges:

\[
\begin{tikzcd}[sep=huge]
& (\psi_2)_{dc_1}(\phi_2)_{c_1a} 
\arrow[r]
& (\psi_2\phi_2)_{da} \\ 
(\psi_1)_{dc_1}(\phi_1)_{c_1a}
\arrow[ru,"{g_{dc_1} * f_{c_1a}}" {description}]
& (\psi_2)_{dc_1}\N_{c_1c_2}(\phi_2)_{c_2a}
\arrow[u,"{1 * (\delta \bullet (1 * \rho))}" {description}]
\arrow[r,"{(\delta \bullet (\rho * 1)) * 1}"]
& (\psi_2)_{dc_2}(\phi_2)_{c_2a}
\arrow[u] \\ 
(\psi_1)_{dc_1}\N_{c_1c_2}(\phi_1)_{c_2a}
\arrow[u,"{1 * (\delta \bullet (1 * \rho))}"]
\arrow[r,"{(\delta \bullet (\rho * 1)) * 1}"]
\arrow[ru,"{g_{dc_1} * 1 * f_{c_2a}}" {description}]
& (\psi_1)_{dc_2}(\phi_1)_{c_2a}
\arrow[ru,"{g_{dc_2} * f_{c_2a}}" {description}]
\end{tikzcd}
\]\\

The commutativity of the above diagram being ensured by the coherence conditions for morphisms of distributors. The vertical composite is more easily defined, simply as a pointwise composition: consider three distributors $\phi_1,\phi_2,\phi_3: (\M,A) \to (\N,C)$, and consider two morphisms $f: \phi_1 \to \phi_2$ and $g: \phi_2 \to \phi_3$. Then $g \bullet f$ is pointwise defined: $(g \bullet f)_{ca} = g_{ca} \bullet f_{ca}$. \\

Now let us prove that the vertical composite $gf$ is still a morphism of distributors; for any $c_1,c_2 \in C$, $a \in A$, we have the following; we simplify the notation for the sake of clarity: 

\begin{align*}
(gf)_{c_1a} \bullet \delta \bullet (1 * \rho) &\simeq g_{c_1a} \bullet f_{c_1a} \bullet \delta \bullet (1 * \rho)\\ 
& \simeq g_{c_1a} \bullet \delta \bullet (1 * \rho) \bullet (1 * f_{c_2a})\\ 
& \simeq \delta \bullet (1 * \rho) \bullet (1 * g_{c_2a}) \bullet (1 * f_{c_2a})\\ 
& \simeq \delta \bullet (1 * \rho) \bullet (1 * (gf)_{c_2a})
\end{align*}

The other condition works in the same way. For the horizontal composition however, as it is defined using the universal property of the colimit, it is more difficult. We want to prove that the $2$-cell:
$(gf)_{d_1a} \bullet \delta_{d_1d_2aa}^{\psi_1\phi_1} \bullet (1_{\Pb_{d_1,d_2}(\psi_1\phi_1)_{d_2,a}} * \rho_a)$ in $\Bc$ is isomorphic to the $2$-cell $\delta_{d_1d_2aa}^{\psi_2\phi_2} \bullet (1_{\Pb_{d_1,d_2}(\psi_2\phi_2)_{d_2,a}} * \rho_a) \bullet (1_{\Pb_{d_1,d_2}} * (gf)_{d_2a})$. To understand better, let us draw the commutative square which defines them at any $c \in C$; the first term is obtained through:

\[
\begin{tikzcd}[row sep = 50, column sep = 50, labels={description}]
\psi_2(d_1,c)\phi_2(c,a)
\arrow[r,"{e_c^{d_1a}}"]
&(\psi_2\phi_2)(d_1,a)\\ 
\psi_1(d_1,c)\phi_1(c,a)
\arrow[r]
\arrow[u,"{g_{d_1c} * f_{ca}}"]
& (\psi_1\phi_1)(d_1,a)
\arrow[u]\\ 
\Pb(d_1,d_2)\psi_1(d_2,c)\phi_1(c,a)\M(a,a)
\arrow[r]
\arrow[u,"{(\delta * \delta) \bullet (1 * \rho * \rho * 1)}"]
& \Pb(d_1,d_2)(\psi_1\phi_1)(d_2,a)
\arrow[u]\\ 
\Pb(d_1,d_2)\psi_1(d_2,c)\phi_1(c,a)
\arrow[r]
\arrow[u,"{1 * \rho}"]
&\Pb(d_1,d_2)(\psi_1\phi_1)(d_2,a)
\arrow[u]
\end{tikzcd}
\]

The diagram to obtain the second term is given by:

\[
\begin{tikzcd}[row sep = 50, column sep = 50, labels={description}]
\psi_2(d_1,c)\phi_2(c,a)
\arrow[r,"{e_c^{d_1a}}"]
&(\psi_2\phi_2)(d_1,a)\\ 
\Pb(d_1,d_2)\psi_2(d_2,c)\phi_2(c,a)\M(a,a)
\arrow[r]
\arrow[u,"{(\delta * \delta) \bullet (1 * \rho * \rho * 1)}"]
& \Pb(d_1,d_2)(\psi_2\phi_2)(d_2,a)\M(a,a)
\arrow[u]\\ 
\Pb(d_1,d_2)\psi_2(d_2,c)\phi_2(c,a)
\arrow[r]
\arrow[u,"{1 * \rho}"]
& \Pb(d_1,d_2)(\psi_2\phi_2)(d_2,a)
\arrow[u]\\ 
\Pb(d_1,d_2)\psi_1(d_2,c)\phi_1(c,a)
\arrow[r]
\arrow[u,"{1 * g_{d_2c} * f_{ca}}"]
&\Pb(d_1,d_2)(\psi_1\phi_1)(d_2,a)
\arrow[u]
\end{tikzcd}
\]

Now we can restrict to proving that in these two diagrams the left part are isomorphic, i.e. that $(g_{d_1c} * f_{ca}) \bullet (\delta * \delta) \bullet (1 * \rho * \rho * 1) \bullet (1 * \rho)$ is isomorphic to $(\delta * \delta) \bullet (1 * \rho * \rho * 1) \bullet (1 * \rho) \bullet (1 * g_{d_2c} * f_{ca})$. A quick computation yields that the rightmost $\delta$ in both expression (the one which relates to $\phi$) can be cancelled with $\rho_a$ and $\rho_c$, and the equality that we want to prove becomes: \[(\delta \bullet (1 * \rho) \bullet g_{d_2c}) * f_{ca} \simeq (g_{d_1c} \bullet \delta \bullet (1 * \rho)) * f_{ca}\]
This is precisely the coherence condition on $g$. The other coherence condition is obtained in the same way, through the one on $f$.
\end{proof}

\begin{ex}[$2$-cells of $\Bc$]
Let there be a $2$-cell $\alpha: f \To g$ between two arrows $f,g: x \to y$ in $\Bc$. Then $\aa$ defines a morphism between the distributors defined by $f$ and $g$. This assignation yields a strict pseudofunctor $\Bc \to \Dist(\Bc)$.
\end{ex}

\begin{ex}[$\V$-natural transformations]
Let $\V$ be a monoidal closed category, and denote by $\Bv$ its delooping. Recall (see~\cite{Kelly} for example for more details) that a $\V$-natural transformation $\alpha$ between two $\V$-functors $f,g: \C \to \D$ is the data for any $c \in \Ob(\C)$ of an arrow $I \to \D(f(c),g(c))$. Now let us denote $(\N,C)$ and $(\Pb,D)$ for the $\Bv$-categories corresponding to $\C$ and $\D$; from $f$ and $g$ we can get four representable distributors $f_!,g_!$ and $f^!,g^!$ between these $\Bv$-categories. Then $\aa$ defines two morphisms of distributors: $\aa_!: f_! \To g_!$ and $\aa^!: f^! \to g^!$. As an example, we describe $\aa_!$; let there be $c \in \Ob(\C)$ and $d \in \Ob(\D)$, then $f_!(d,c) = \D(d,f(c))$ and $g_!(d,c) = \D(d,g(c))$. Define $(\aa_!)_{dc}$ as the composite:
\[\D(d,f(c)) = \D(d,f(c)) \otimes I \to \D(d,f(c)) \otimes \D(f(c),g(c)) \to \D(d,g(c))\] 
Then $\aa_!$ is a morphism of distributors.
\end{ex}

\begin{rk}
Note that we have not given any definition for natural transformation of $\Bc$-functors; however, since $\Dist(\Bc)$ equips $\Cat(\Bc)$ with proarrows, we could in theory define the $2$-structure of $\Cat(\Bc)$ from that of $\Dist(\Bc)$, stating that $2$-cells between $\Bc$-functors are given by $2$-cells between the corresponding representable distributors.
\end{rk}

\subsection{Complete $\Bc$-categories}

In the case of a locale $X$, sheaves over $X$ can be expressed as a subcategory of the category of $X$-categories; that is the subject of~\cite{WaltersCahiers}, which identifies ``Cauchy-complete'' $X$-categories as the good class to recover the topos of sheaves on $X$. Similar results have been obtained for general sites using quantaloids (a particular class of locally ordered bicategories);~\cite{HeymansThesis} shows that the topos $\wh{\C}_{J}$ of sheaves on a site $(\C,J)$ can be obtained as the category of symmetrically complete $\Q$-categories, where $\Q$ is a quantaloid obtained directly from the site $(\C,J)$. We present here, in a style more akin to~\cite{NawazThesis,Gylys} (i.e. using singletons), the notion (already present in~\cite{StreetEnrichment}) of complete $\Bc$-category.

\begin{deft}
A $\Bc$-category $(\M,A)$ is said to be (\textbf{Cauchy-})\textbf{complete} whenever any distributor $\phi: (\N,C) \to (\M,A)$ which has a right adjoint in $\Dist(\Bc)$ is representable, for any $\Bc$-category $(\N,C)$.
\end{deft}

As shown below, the completeness of a $\Bc$-category can be checked on a smaller, particularly important class of distributors, namely the singletons.

\begin{deft}
Let $(\M,A)$ be a $\Bc$-category. A \textbf{presingleton} of $(\M,A)$ is the data of a $0$-cell $* \in \Ob(\Bc)$ and of a distributor $\sigma: (\id_*,*) \to (\M,A)$.
\end{deft}

\begin{rk}\label{SingCondReformulation}
As $\Bc$-categories of the form $(\id_*,*)$ are rather simple, the distributor conditions acquire a simpler expression in the case of presingletons: given a $\Bc$-category $(\M,A)$, a presingleton $\sigma: (\id_*,*) \to (\M,A)$ is equivalently given by
\begin{enumerate}
\itb For all $a \in A$, a $1$-cell $\sigma(a): * \to ta$,
\itb For all $a,b \in A$, a $2$-cell $\delta_{ab}: \M(a,b)\sigma(b) \To \sigma(a)$
\end{enumerate}
subject to the following conditions: 
\begin{enumerate}
\itb For all $a \in A$, $\delta_{aa} \bullet (\rho_a^{\M} * 1_{\sigma(a)}) = 1_{\sigma(a)}$.
\itb For all $a,b,c \in A$, $\delta_{ab} \bullet (1_{\M(a,b)} * \delta_{b,c}) = \delta_{ac} \bullet (\iota_{abc}^{\M} * 1_{\sigma(c)})$.\\
\end{enumerate} 

Similarly, a ``precosingleton'', i.e. a distributor $\tau: (\M,A) \to (\id_*,*)$, is equivalent to the data of
\begin{enumerate}
\itb For any $a \in A$, a $1$-cell $\tau(a): ta \to *$,
\itb For any $a,b \in A$, a $2$-cell $\delta_{a,b}: \tau(a)\M(a,b) \To \tau(b)$
\end{enumerate}
such that
\begin{enumerate}
\itb For any $a \in A$, $\delta_{aa} \bullet (1_{\tau(a)} * \rho_a^{\M}) = 1_{\tau(a)}$
\itb For any $a,b,c \in A$, $\delta_{ba} \bullet (\delta_{cb} * 1_{\M(b,a)}) = \delta_{ca} \bullet (1_{\tau(c)} * \iota_{cba}^{\M})$
\end{enumerate}
\end{rk}

\begin{prop}[$\Bc$-category of presingletons]
Let $(\M,A)$ be a $\Bc$-category, and denote by $\cc PA$ the set of all its presingletons.
\begin{enumerate}
\itb For any $\s: (\id_x,x) \to (\M,A)$, we write $t\s = x$. This defines a typing function $\cc PA \to \Ob(\Bc)$.
\itb For any two presingletons $\s,\t$ of $(\M,A)$, write $\Pb(\s,\t)$ for the colimit of the following diagram in $\Bc(t\t,t\s)$: 
\begin{enumerate}
\itb The objects of the diagram are $1$-cells $t\t \to t\s$ in $\Bc$ for which there exists a morphism of distributors $\s f \to \t$, taken as many times as there is such a morphism of distributors. In other words, they are the image of the projection from the set of pairs $(f,\aa_f)$ with $f: t\t \to t\s$ a $1$-cell in $\Bc$ and $\aa_f: \s f \To \t$ a $2$-cell in $\Dist(\Bc)$.
\itb The arrows of the diagram are the $2$-cells $\beta: f \to g$ for which $\aa_g \bullet (1_{\s} * \beta) \simeq \aa_f$.
\end{enumerate}
\end{enumerate}
Then $(\Pb,\cc PA)$ is a $\Bc$-category.
\end{prop}

\begin{proof}
Let $\s,\tau,\upsilon$ be three presingletons on $(\M,A)$. Define $\rho_{\s}^{\Pb}$ as the canonical injection of $\id_{t\s}$ in the colimit $\Pb(\s,\s)$. Let $(f,\aa_f: \s f \to \tau)$ and $(g,\aa_g: \tau g \to \upsilon)$ be two elements at the base of the colimits defining $\Pb(\s,\tau)$ and $\Pb(\tau,\upsilon)$; then $(fg,\aa_g \bullet (\aa_f * 1_g))$ is at the base of the colimit defining $\Pb(\s,\upsilon)$. Since $\Bc$ is closed, this suffices to obtain a $2$-cell $\iota_{\s\tau\upsilon}^{\Pb}: \Pb(\s,\tau)\Pb(\tau,\upsilon) \To \Pb(\s,\upsilon)$. Now verifying the coherence conditions is only standard calculus with colimits, see the end of the proof of~\ref{GrothConstrDef} for a very similar proof pattern.
\end{proof}

\begin{ex}[$1$-cells of $\Bc$]
For any objects $x,y$ in $\Bc$, a distributor $(\id_x,x) \to (\id_y,y)$ is the same thing as a $1$-cell $f: x \to y$ in $\Bc$. In particular, any $1$-cell $f: x \to y$ defines a presingleton of $(\id_y,y)$.
\end{ex}

\begin{deft}
A presingleton is said to be a \textbf{singleton} if it has a right adjoint in the bicategory $\Dist(\Bc)$. 
\end{deft}

\begin{rk}
An adjunction $\s \dashv \tau$ is given by:
\begin{enumerate}
\itb A unit $\eta: \id_* \To \int^{a:A}_{\M} \tau(a)\s(a)$
\itb A counit, given for all $a,b\in A$ by a $2$-cell $\s(a)\tau(b) \To \M(a,b)$ 
\end{enumerate}
subject to the usual triangular identities.
\end{rk}

\begin{rk}\label{AdjHomSet}
An adjunction can also be described in a more usual hom-set-theoretical way; for instance, it is easy to prove that a presingleton $\s: (\id_{t\s},t\s) \to (\M,A)$ and a precosingleton $\s^*: (\M,A) \to (\id_{t\s},t\s)$ are adjoint in $\Dist(\Bc)$ if and only if for any presingleton $\tau: (\id_{t\tau},t\tau) \to (\M,A)$ and any $1$-cell $f: t\tau \to t\s$ (i.e. any singleton of type $t\tau$ of $(\id_{t\s},t\s)$), we have the following isomorphism: \[\Pb_{\M,A}(\s f,\tau) \simeq \Pb_{(\id_{t\s},t\s)}(f,\s^*\tau)\]
\end{rk}

\begin{rk}\label{MateFormula}
Even more generally, the notion of adjunction in a bicategory can be reformulated in ``hom-set-theoretical terms'' by considering the notion of \textit{mates}. A rigorous treatment is presented in section 6.1 of~\cite{JohnsonYau2DimCats}; we recall here the main points adapted to our case, of the bicategory $\Dist(\Bc)$. Consider four $\Bc$-categories $(\M_1,A_1)$, $(\M_2,A_2)$, $(\N_1,C_1)$, $(\N_2,C_2)$, and four distributors $\phi_1: (\M_1,A_1) \to (\N_1,C_1)$, $\phi_2: (\M_2,A_2)\to (\N_2,C_2)$, $\psi_{\M}: (\M_1,A_1) \to (\M_2, A_2)$ and $\psi_{\N}: (\N_1,C_1) \to (\N_2,C_2)$. Suppose that $\phi_1$ and $\phi_2$ are maps in $\Dist(\Bc)$, with right adjoints $\phi_1^*$ and $\phi_2^*$. Then there is a bijection of hom-sets: 
\begin{center}
$
\Hom_{\Dist(\Bc)((\M_1,A_1),(\N_2,C_2))}(\phi_2\psi_{\M},\psi_{\N}\phi_1)$

$\cong$

$\Hom_{\Dist(\Bc)((\N_1,C_1),(\M_2,A_2))}(\psi_{\M}\phi_1^*,\phi_2^*\psi_{\N})$
\end{center}

This isomorphism sends any morphism of distributors $\alpha: \phi_2\psi_{\M} \To \psi_{\N}\phi_1$ to its \textbf{mate} which is defined by the following diagram:

\[
\begin{tikzcd}[column sep = 30, row sep = 50]
(\M_2,A_2)
& {}
& (\M_2,A_2)
\arrow[ll,"{\M_2 = \id_{(\M_2,A_2)}}"' {name={112}, description}]
& (\M_1,A_1)
\arrow[l,"{\psi_{\M}}"'{name={13}, description}]
& (\N_1,C_1)
\arrow[l,"{\phi_1^*}"'{name={14}, description}]\\ 
(\M_2,A_2)
& (\N_2,C_2)
\arrow[l,"{\phi_2^*}"{name={21}, description}]
& (\M_2,A_2)
\arrow[l,"{\phi_2}"{name={22}, description}]
\arrow[ll,"{}"{name={212}}, phantom]
& (\M_1,A_1)
\arrow[l,"{\psi_{\M}}"{name={23}, description}]
\arrow[ll,"{}"{name={223}}, phantom]
& {}\\ 
{}
& (\N_2,C_2)
& (\N_1,C_1)
\arrow[l,"{\psi_{\N}}"{name={32}, description}]
& (\M_1,A_1)
\arrow[l,"{\phi_1}"{name={33}, description}]
\arrow[ll,"{}"{name={323}}, phantom]
& (\N_1,C_1)
\arrow[l,"{\phi_1^*}"{name={34}, description}]
\arrow[ll,"{}"{name={334}}, phantom]\\ 
(\M_2,A_2)
& (\N_2,A_2)
\arrow[l,"{\phi_2^*}"{name={41}, description}]
& (\N_1,C_1)
\arrow[l,"{\psi_{\N}}"{name={42}, description}]
& {}
& (\N_1,C_1)
\arrow[ll,"{\N_1 = \id_{(\N_1,C_1)}}"{name={434}, description}]
\arrow[Rightarrow, from={112}, to={212}, "{\eta_{\phi_2}}", shorten >=3, shorten <=5]
\arrow[Rightarrow, from={13}, to={23}, "{1}", shorten >=3, shorten <=5]
\arrow[Rightarrow, from={14}, to={34}, "{1}", shorten >=3, shorten <=5]
\arrow[Rightarrow, from={21}, to={41}, "{1}", shorten >=3, shorten <=5]
\arrow[Rightarrow, from={223}, to={323}, "{\alpha}", shorten >=3, shorten <=5]
\arrow[Rightarrow, from={32}, to={42}, "{1}", shorten >=3, shorten <=5]
\arrow[Rightarrow, from={334}, to={434}, "{\e_{\phi_1}}", shorten >=3, shorten <=5]
\end{tikzcd}
\]

The mate of $\aa$ is therefore given by $(1 * 1 * \e_{\phi_1}) \bullet (1 * \alpha * 1) \bullet (\eta_{\phi_2} * 1 * 1)$. A similar construction exists in the other sense.

\end{rk}

\begin{deft}
Suppose that $\Bc$ is endowed with an involution. A singleton $\s$ of a symmetric $\Bc$-category is said to be \textbf{symmetric} if it is a symmetric map in $\Dist(\Bc)$, i.e. if its right adjoint is $\s^{\circ}$.
\end{deft}

\begin{ex}[$\Bc$-category of singletons]
Let $(\M,A)$ be a $\Bc$-category; we have its $\Bc$-category of presingletons $(\Pb,\cc PA)$. If we restrict $\cc PA$ to only the singletons of $(\M,A)$, we get another $\Bc$-category $(\Sb,\C A)$ which we call the $\Bc$-category of singletons. If $(\M,A)$ is symmetric, we can restrict $(\Pb,\cc PA)$ to symmetric singletons and get a symmetric $\Bc$-category $(\Sb,\C_{\s} A)$. Of course, $\Sb(\tau,\s) = \Pb(\tau,\s)$ when $\tau$ and $\s$ are singletons; however, in this case the computation of the colimit can be simplified as we have $\Sb(\s,\tau) = \s^*\tau$, where $\s^*$ denotes the right adjoint of $\s$; this operation is a composition of distributors, meaning that it is equal to $\int^{a:A}_{\M} \s^*(a)\tau(a)$. This is because in the colimit $\Pb(\s,\tau)$ of all $f: t\tau \to t\s$ for which there exists a $2$-cell $\alpha: \s f \To \tau$ in $\Dist(\Bc)$, the fact that $\s$ has a right adjoint gives an equivalence between $2$-cells $\s f \To \tau$ and $2$-cells $f \To \s^* \tau$. This means that the above formula is still true even if $\tau$ is not a singleton but only a presingleton. 
\end{ex}

\begin{ex}[Representable singletons]
Let $(\M,A)$ be a $\Bc$-category. Choosing an element $a$ of $A$ is the same thing as giving a $\Bc$-functor $f_a: (\id_{ta},ta) \to (\M,A)$. From this $\Bc$-functor, we get a pair of representable distributors $(f_a)_!$ and $f_a^!$ which are adjoint in $\Dist(\Bc)$ and go between $(\id_{ta},ta)$ and $(\M,A)$; hence they form a singleton that we denote by $\M(-,a)$; this family of singletons are called the \textbf{representable singletons} of $(\M,A)$. The value at $b$ of the representable singleton $\M(-,a)$ is of course given by $\M(b,a)$, while the right adjoint takes the value $\M(a,b)$. Note that the unit of the adjunction $\M(-,a) \dashv \M(a,-)$ is given by $\rho_a$ while the counit is given by $\iota^{\M}_{-a-}$. This assignation defines a $\Bc$-functor from $(\M,A)$ to the $\Bc$-category $(\Sb,\C A)$ of its singletons: we call it the \textbf{Yoneda $\Bc$-functor}. If $(\M,A)$ is symmetric, then this Yoneda functor restricts to the $\Bc$-category $(\Sb,\C_{\s} A)$ of symmetric singletons. 
\end{ex}

\begin{ex}[$\V$-presheaves]\label{exvpresheavessing}
Let $\V$ be a monoidal category, and let $\C$ be a $\V$-category. Then any $\V$-presheaf $F: \C^{\op} \to \V$ yields a $\Bv$-presingleton on the $\Bv$-category $(\N,\C)$ associated to $\C$; with any $c \in C$, we associate an object $F(c)$ of $\V$, and the $2$-cell in $\Bv$ (i.e. arrow in $\V$) $\delta_{ab}: \C(a,b)\otimes F(b) \to F(a)$ is simply the restriction morphism. This is an equivalence: all presingletons produce $\V$-presheaves; on the other hand distributors $(\N,C) \to (I,*)$ yield copresheaves $\C \to \Set$. Among these presheaves and copresheaves we have of course the representable ones; for $c \in \Ob(\C)$, we get the adjoint pair $\C(-,c)$ and $\C(c,-)$, which yields a representable singleton of $\cal C$. In general, not all the singletons are representable; when $\V = \Set$, this is the case if and only if $\C$ is Karoubi-complete.
\end{ex}

\begin{ex}[Restriction of sections]
Let $F$ be a presheaf on a topological $X$, and consider the $\R(X)$-category $(\N,\si F)$ of its sections. Let $s$ be a section of $F$ defined on some open subset $U$ of $X$, and consider $V \subseteq U$ another open subset of $X$. Because $F$ is a presheaf, there exists a restriction $s_{|V}$ of $s$ to $V$, and for any other section $t$ of $F$, we have $\N(t,s_{|t}) = V \wedge \N(t,s)$. Now given any $\R(X)$-category $(\M,A)$, we can define for any $a \in A$ and any open subset $U$ of $X$ a \textbf{restriction} singleton by $\s(b) = U \wedge \M(b,a)$; it is easily checked that this defines a symmetric singleton on $(\M,A)$. The representatibility of this singleton is equivalent to the existence of an element $a_{|U}$ in $A$ which is to be understood as the restriction of $a$ to $U$.
\end{ex}

\begin{ex}[Glueing of sections]
Let $F$ be a presheaf on a topological $X$, and consider the $\R(X)$-category $(\N,\si F)$ of its sections. Let $s$ be a section of $F$ defined on some open subset $U$ of $X$, and consider a cover $U = \bigvee_{i\in I} U_i$ of $U$; for each $i\in I$, write $s_i:= s_{|U_i}$. Then for any other section $t$ of $F$, we have $\N(t,s) = \bigvee_{i\in I} \N(t,s_i)$. Now, given any $\R(X)$-category $(\M,A)$ (in which all restriction singletons are representable, i.e. such that any element $a \in A$ has a restriction $a_{|U}$ along any open subset $U \subseteq ta$), we can define for any ``matching family'' $(a_i)_{i\in I}$ (i.e. such that $\M((a_i)_{|ta_i \wedge ta_j},(a_j)_{|ta_i \wedge ta_j}) = ta_i \wedge ta_j$) a \textbf{glueing} singleton by $\s(b) = \bigvee_{i\in I} \N(b,a_i)$; it is easily checked that this defines a symmetric singleton on $(\M,A)$. The representability of this singleton is equivalent to the existence of an element $a \in A$ such that $ta = \bigvee_{i\in I} ta_i$ in $A$, which is to be understood as the glueing of the compatible family $(a_i)_{i\in I}$. 
\end{ex}

\begin{ex}[Maps of $\Bc$]
By a map $\g$ in $\Bc$ we mean a $1$-cell which has a right adjoint $\g^*$ in the bicategorical sense, meaning that there are unit and counit morphism and which satisfy the triangular identities. If $\g: x \to y$ is such a map, then it defines a singleton $\g: (\id_x,x) \to (\id_y,y)$. As it is a distributor, it is possible to compose it with other singletons to obtain again singletons: for instance, if $(\M,A)$ is a $\Bc$-category and $a$ is an element of $A$, then for any map $\g: x \to ta$ we can define a singleton $\M(-,a)\g$. If $\g$ is a symmetric map, meaning that its right adjoint is $\g^{\circ}$ then this singleton $\M(-,a)\g$ is itself a symmetric singleton.
\end{ex}

\begin{lem}[$\Bc$-enriched Yoneda lemma]\label{YonedaSingleton}
Let $(\M,A)$ be a $\Bc$-category, and write $(\Sb,\C A)$ for its $\Bc$-category of singletons. Then for any $a \in A$ and any $\sigma \in \C A$, we have:
\[\Sb(\sigma,\M(-,a)) \simeq \sigma^*(a)\]
\[\Sb(\M(-,a),\sigma) \simeq \sigma(a)\]
\end{lem}

\begin{proof}
For any a $\in A$, we have $\Sb(\s,\M(-,a)) = \s^* \M(-,a) = \int^{b: A}_{\M} \s^*(b)\M(b,a)$, this is equal to $(\s^*\M)(a)$; because $\M$ is an identity distributor over $(\M,A)$ this is also equal to $\s^*(a)$. The same argument proves the second equation. Note that because $\M(-,a)$ is a singleton, the equation $\Sb(\M(-,a),\s) = \s$ still holds if $\s$ is a presingleton.
\end{proof}

\begin{rk}
We say that a $\Bc$-functor $f: (\M,A) \to (\N,C)$ is fully faithful whenever $f_{aa'}: \M(a,a') \to \N(f(a),f(a'))$ is an isomorphism for all $a,a' \in A$. The $\Bc$-enriched Yoneda lemma ensures that the Yoneda $\Bc$-functor $\yf: (\M,A) \to (\Sb,\C A)$ is fully faithful in this sense.
\end{rk}

\begin{ex}[Categorical Yoneda lemma]
In the case $\Bv = \Bc_{\Set}$ we recover the usual Yoneda lemma; let $\C$ be a small category then representable singletons of $\C$ are the $\Hom(-,c)$ for all the objects $c$ of $\C$. Now recall from Example~\ref{exvpresheavessing} that any presheaf $F: \C^{\op} \to \Set$ yields a presingleton of $\C$. The second equation of the $\Bc$-categorical Yoneda lemma therefore applies (even if the presheaf does not yield a singleton) and yields the usual Yoneda lemma.
\end{ex}

\begin{deft}
Let $(\M,A)$ be a $\Bc$-category. We say that $(\M,A)$ is \textbf{skeletal} if the following two propositions are equivalent for all $x,x' \in A$:
\begin{enumerate}
\itb $x=x'$
\itb $\M(-,x) = \M(-,x')$
\end{enumerate}
\end{deft}

\begin{ex}
Not all $\Bc$-categories are skeletal; for an artificial example consider the $\Bc$-category $(\M,A)$ with $A =\{a_1,a_2\}$ and $\M(a_i,a_j) = \id_x$, with $i,j \in \{1,2\}$ and $ta_1 = ta_2 = x$. Many examples that naturally arise in practice are, however, skeletal.
\end{ex}

\begin{prop}\label{PreSk}
For any $\Bc$-category $(\M,A)$, the category of its presingletons $(\Pb,\cc PA)$ is skeletal.
\end{prop}

To prove the above proposition, we will need a technical lemma which is extremely useful to simplify computations in $(\Sb,\C A)$.

\begin{lem}\label{InclCanonSing}
Let $\s,\tau$ be two singletons of a $\Bc$-category $(\M,A)$. Let $f: t\tau \to t\s$ be a $1$-cell in $\Bc$ and suppose that there is some $2$-cell $\alpha_f: \s f \to \tau$ in $\Dist(\Bc)$. Then the canonical inclusion $e_f$ of $f$ in $\Pb(\s,\tau) = \s^*\tau$ is: $$e_f \simeq (1_{\s^*} * \aa_f) \bullet (\eta_{\s} * 1_f)$$
\end{lem}

\begin{proof}
We have seen in Remark~\ref{AdjHomSet} that there is an isomorphism between $\Pb(\s,\tau)$ and $\Pb(\id_{t\s},\s^*\tau)$ obtained by sending any $1$-cell $f: t\tau \to t\s$ in the base of the colimiting cocone definint $\Pb(\s,\tau)$, that is, with a $2$-cell $\alpha_f: \s f \to \tau$, to its mate $\beta_f = (1_{\s^*},\aa_f) \bullet (\eta_{\s} * 1_f)$ which goes from $f$ to $\s^*\tau$. This mate formula is directly obtained from Remark~\ref{MateFormula}. Now we can construct a whole colimiting cocone going to $\Pb(\id,\s^*\tau)$ by considering all such $(f,\beta_f)$; the fact that the mate construction yields a bijection ensures that we capture in this way all required $1$-cells. Moreover, this diagram does commute as needed: the considered $2$-cells $\zeta: f \To g$ are precisely those such that $\beta_g \bullet (1_{\id} * \zeta) = \beta_f$, i.e. those such that $\beta_g \bullet \zeta = \beta_a$, this shows that $e_f$ is isomorphic to $\beta_f$.

\end{proof}

\begin{proof}[Proof of the proposition~\ref{PreSk}]
This is a direct consequence of the above Yoneda lemma~\ref{YonedaSingleton}, recall that we established in the proof that the formula $\Pb(\s,\M(-,a)) = \s(a)$ for any presingleton $\s$. Now suppose that $\s_1$ and $\s_2$ are two presingletons of $(\M,A)$ such that $\Pb(-,\s_1) = \Pb(-,\s_2)$, then for any $a \in A$, $\s_1(a) = \Pb(\M(-,a),\s_1) = \Pb(\M(-,a),\s_2) = \s_2(a)$. Proving that $\delta^{\s_1} = \delta^{\s_2}$ is more difficult; let us first fix $\s$ and remark that $\delta_{-b}^{\s} \bullet (\iota^{\M}_{-ab} * 1) = \delta^{\s}_{-a} \bullet (1_{\M(-,a)} * \delta^{\s}_{ab})$; this shows that the $2$-cell $\delta^{\s}_{ab}$ is in the base diagram of the colimit defining $\Pb(\M(-,a),\s)$ as a $2$-cell going from $\M(a,b)\s(b)$ to $\s(a)$, which are both vertices of the base diagram. Now, denoting by $g = \M(a,b), f = \s(b)$ and $h = \s(a)$, we have the following commutative diagram: 
\[
\begin{tikzcd}[row sep = 50, column sep = 50, labels=description]
\s(a)
\arrow[r,"{e_h}"]
& \Pb(\M(-,a),\s)\\ 
\M(a,b)\s(b) 
\arrow[u,"{\delta_{ab}^{\s}}"]
\arrow[ru,"{e_{gf}}"]
\arrow[r,"{e_g * e_f}"]
& \Pb(\M(-,a),\M(-,b))\Pb(\M(-,b),\s)
\arrow[u,"{\iota^{\Pb}_{\M(-,a)\M(-,b)\s}}"]
\end{tikzcd}
\]

Now we have $\delta^{\Pb(-,\s_1)}_{\M(-,a)\M(-,b)} = \iota^{\Pb}_{\M(-,a)\M(,-b)\s_1} = \iota^{\Pb}_{\M(-,a)\M(,-b)\s_2} = \delta^{\Pb(-,\s_2)}_{\M(-,a)\M(-,b)}$; to conclude we only need to prove that $e_h = 1_{\s(a)}$. To this end we can apply lemma~\ref{InclCanonSing} to the pair $(h = \s(a), \aa_h = \delta^{\s}_{-a})$ which gives $e_h = (1_{\M(a,-)} * \aa_h) \bullet (\eta_{\M(-,a)} * 1_{h})$, which is isomorphic to $(1_{\M(a,-)} * \delta_{-a}^{\s}) \bullet (\rho_a * 1_{\s(a)})$. A quick computation shows that the horizontal composite of morphisms of distributors $1_{\M(a,-)} * \delta_{-a}^{\s}$ is equal to $\delta_{aa}^{\s}$, meaning that $e_h = \delta^{\s}_aa \bullet (\rho_a * 1_{\s(a)}) = 1_{\s(a)}$, thus concluding the proof.

\end{proof}

We can now give the following alternative definition of completeness (resp. of symmetric completeness) for skeletal (resp. skeletal and symmetric) $\Bc$-categories, the latter meaning completeness with respect to all the symmetric distributors (i.e. those maps in $\Dist(\Bc)$ whose involute are their right adjoint).

\begin{deft}
Let $(\M,A)$ a skeletal $\Bc$-category. We say that $(\M,A)$ is \textbf{complete} if every of its singletons is representable, in other words if for any singleton $\sigma: (\id_*,*) \to (\M,A)$, there is some $x \in A$ such that $* = tx$ and $\sigma = \M(-,x)$. If $(\M,A)$ is symmetric, we say that it is \textbf{symmetrically complete} if every of its \textit{symmetric} singletons is representable.
\end{deft}

\begin{proof}
It suffices to remark that, for any distributor $\phi: (\N,C) \to (\M,A)$ with a right adjoint, $\phi(-,c): (\id_{ta},ta) \to (\M,A)$ is a singleton whose representability is equivalent to that of $\phi$.
\end{proof}

\begin{ex}[Sheaves on a locale]
We have pointed out before that for a presheaf $F$ on a topological space $X$, any open subset $U$ of $X$ defines a symmetric singleton $\s$ which maps any section $s$ of $X$ to its local behaviour on $U$; and we have stated that this singleton $\s$ is representable if and only if there is only one possible local behaviour at $U$. Thus the $\R(X)$-category $\si F$ of sections of $F$ is symmetrically complete if and only if $F$ is a sheaf.
\end{ex}

\begin{lem}\label{RhoIotaSing}
Let $(\M,A)$ be a $\Bc$-category, and write $(\Sb,\C A)$ for its category of singletons. Then for any $a,b,c \in A$, we have: 
\begin{enumerate}
\item For any singleton $\s$ of $(\M,A)$, $\rho^{\Sb}_{\s} = \eta_{\s}$.
\item For any singletons $\s,\tau,\upsilon$ of $(\M,A)$, $\iota^{\Sb}_{\s\tau\upsilon} = 1_{\s^*} * \e_{\tau} * 1_{\upsilon}$.
\item For any $a \in \A$, $\rho_a^{\M} = \rho_{\M(-,a)}^{\Sb}$.
\item For any $a,b,c \in A$, $\iota_{abc}^{\M} = \iota_{\M(-,a)\M(-,b)\M(-,c)}^{\Sb}$.
\end{enumerate}
\end{lem}

\begin{proof}
$(1)$ Recall that $\rho_{\s}^{\Pb}$ is defined as the canonical inclusion of $\id_{t\s}$ in $\Pb(\s,\s)$ associated to the $2$-cell $1_{\delta_{\s}}: \delta(\s) \to \delta(\s)$, where $\delta(\s) = (\id_{t\s},t\s)$. We have seen in Remark~\ref{AdjHomSet} that there is an isomorphism $\Pb_{\M,A}(\s,\s) \simeq \Pb_{\delta(\s)}(\id_{t\s},\s^*\s)$, which is obtained thanks to an isomorphism between the bases of the respective colimiting cocones. This isomorphism is obtained through the mate construction explicited in Remark~\ref{MateFormula}; now the morphism of distributors $1_{\s}: \delta_{\s} \to \delta_{\s}$, seen as an element of $\Hom(\s \cdot \id_{\delta(\s)}, \s \cdot \id_{\delta(\s)})$ has a mate which is an element $\beta$ of $\Hom(\id_{\delta(\s)} \cdot \id_{\delta(\s)}, \s^*\s)$, this mate is given by the formula:
\begin{align*}
\beta &= (1_{\s^*} * 1_{\s} * \e_{\id}) \bullet (1_{\s^*} * 1_{\s} * 1_{\s^*}) \bullet (\eta_{\s} * 1_{\id} *  1_{\id})\\ 
&= (1_{\s^*\s} * 1_{\id}) \bullet 1_{\s^*\s\s^*} \bullet (\eta_{\s})\\ 
&= \eta_{\s}
\end{align*}

This proves that $(\id_{t\s}, 1_{\delta(\s)})$ is sent to $(\id_{t\s},\eta_{\s})$ by the mate construction between the base of the colimiting diagrams which define $\Pb(\s,\s) \simeq \Pb(\id_{\delta(\s)},\s^*\s)$. Now to conclude, we shall show that the inclusion of $(\id_{t\s},\eta_{\s})$ in $\Pb(\id_{\delta(\s)},\s^*\s)$ is isomorphic to $\eta_{\s}$. This is a general phenomenon as we can construct a whole colimiting cocone satisfying this property: elements of the base must be $1$-cells $g: t\s \to t\s$ together with $2$-cells $\alpha: g \To \s^*\s$, then the compatibility condition for the considered $2$-cells in the base ensure that everything does commute. More clearly, for any pairs $(f,\alpha_f)$ and $(g,\alpha_g)$ with $\alpha_f: f \to \s^*\s$ and $\alpha_g: g \to \s^*\s$, then the considered $2$-cells $\beta: f \To g$ in the base of the colimiting diagram are exactly those which are such that $\alpha_g \bullet (1_{\id} * \beta) = \alpha_f$, i.e. those such that $\alpha_g\bullet \beta = \alpha_f$. This means that the consideration of all $(f_i,\alpha_{f_i})_{i\in I}$ yields a colimiting cocone for $\Pb_{\delta(\s)}(\id_{\delta_{\s}},\s^*\s)$; thus the canonical inclusion of any $(f,\alpha_f)$ in the base of the colimiting cocone is $\alpha_f$. Applied to the case $(\id_{t\s},\eta_{\s})$, this proves that the canonical inclusion is $\eta_{\s}$; because $\rho_{\s}^{\Pb_{(\M,A)}}$ is defined as the canonical inclusion of $(\id_{t\s},1_{\delta(\s)})$, which is sent to $(\id_{t\s},\eta_{\s})$ by the isomorphism between the base of the colimiting cocones, we finally obtain $\rho_{\s} \simeq \eta_{\s}$. \\ 

$(2)$ Recall that $\iota_{\s\t\u}$ is defined through the consideration of the following commutative diagram: 

\[
\begin{tikzcd}[row sep = 30, column sep = 50, labels={description}]
fg
\arrow[r,"{e_{fg}}"]
& \Pb(\s,\u)\\ 
fg 
\arrow[r,"{e_f * e_g}"]
\arrow[u,"{1_{fg}}"]
& \Pb(\s,\t)\Pb(\t,\u)
\arrow[u,"{\iota_{\s\t\u}}"]
\end{tikzcd}
\]

For all $f: t\t \to t\s$, $g: t\u \to t\t$, with $\aa_f: \s f \to \t$ and $\aa_g: \t g \to \u$. Here $e_{fg}$ is yielded by $\aa_{fg} = \aa_g \bullet (\aa_f * 1_g)$. To conclude we only need to prove that $(1_{\s^*} * \e_{\t} * 1_{\u}) \bullet (e_f * e_g) \simeq e_{fg}$; recall that lemma~\ref{InclCanonSing} gives expressions for $e_f$, $e_g$ and $e_{fg}$ using $\aa_f$ and $\aa_g$. From that we can compute: 
\begin{eqnarray*}
& &(1_{\s^*} * \e_{\t} * 1_{\u}) \bullet (e_f * e_g) \\ 
&\simeq &(1_{\s^*} * \e_{\t} * 1_{\u}) \bullet (((1_{\s^*} * \aa_f) \bullet (\eta_{\s} * 1_f)) * (1_{\t^*} * \aa_g) \bullet (\eta_{\t} * 1_g))\\ 
&\simeq& (1_{\s^*} * \e_{\t} * 1_{\u}) \bullet (1_{\s^*} * \aa_f * 1_{\t^*} * \aa_g) \bullet (\eta_{\g} * 1_f * \eta_{\t} * 1_g)\\ 
&\simeq& (1_{\s^*} * ((1_{\s} * \e_{\t}) \bullet (\aa_f * 1_{\t})) * \aa_g) \bullet (\eta_{\g} * 1_f * \eta_{\t} * 1_g)\\ 
&\simeq& (1_{\s^*} * (\aa_{g} \bullet (\e_{\t} * 1_{\t g}) \bullet (1_{\t} * \eta_{\t} * 1_g) \bullet (\aa_f * 1_g))) \bullet (\eta_{\s} * 1_{fg})\\
&\simeq& (1_{\s^*} * (\aa_g \bullet (\aa_f * 1_g))) \bullet (\eta_{\s} * 1_{fg})\\ 
&\simeq& (1_{\s^*} * \aa_{fg}) \bullet (\eta_{\s} * 1_{fg})\\ 
&\simeq& e_{fg}
\end{eqnarray*}

$(3)$ We have $\rho^{\Sb}_{\M(-,a)} = \eta_{\M(-,a)} = \rho_a^{\M}$. \\

$(4)$ Recall that we have $(\e_{\M(-,x)})_{ab} = \iota_{axb}$. We therefore want to compute the horizontal composite: $1_{\M(a,-)} * \iota_{-b-} * 1_{\M(-,c)}$. We will first compute $1_{\M(a,-)} * \iota_{-b-}$. Let us write a setup diagram:

\[
\begin{tikzcd}[sep=huge]
(\id_{ta},ta)
& (\M,A)
\arrow[l,"{\M(a,-)}"'{name={11}}]
& (\M,A)
\arrow[l,"{\M(-,b)\M(b,-)}"'{name={21}}]\\
(\id_{ta},ta)
& (\M,A)
\arrow[l,"{\M(a,-)}"{name={12}}]
& (\M,A)
\arrow[l,"{\M}"{name={22}}]
\arrow[Rightarrow, "{1_{\M(a,-)}}", from=11, to=12,  shorten <= 5, shorten >= 3]
\arrow[Rightarrow, "{\iota_{-b-}}", from=21, to=22,  shorten <= 5, shorten >= 3]
\end{tikzcd}
\]

From this setup diagram, we can draw the defining diagram for the horizontal composite $1_{\M(a,-)} * \iota_{-b-}$ at some $u_1,u_2,v \in A$; because all $2$-cells in the diagram are just compositions of $\iota$, $\rho$ and $1$ in which all $\rho$ can ultimately be cancelled using the coherence relations for $\Bc$-categories, we do not label them. 

\[
\begin{tikzcd}[row sep = 25]
& \M_{au_1}\M_{u_1v}
\arrow[r]
& \M_{av} \\ 
\M_{au_1}\M_{u_1b}\M_{bv}
\arrow[ru]
& \M_{au_1}\M_{u_1u_2}\M_{u_2v}
\arrow[r]
\arrow[u]
& \M_{au_2}\M_{u_2v}
\arrow[u]\\
\M_{au_1}\M_{u_1u_2}\M_{u_2b}\M_{bv}
\arrow[u]
\arrow[r]
\arrow[ru]
& \M_{au_2}\M_{u_2b}\M_{bv}
\arrow[ru]
\end{tikzcd}
\]

The consideration of these diagrams shows that we have: $1_{\M(a,-)} * \iota_{-b-} \simeq \iota_{ab-}: \M(a,b)\M(b,-) \to \M(a,-)$. Similar considerations for the other side of the horizontal composite then show the desired isomorphism.
\end{proof}

\begin{prop}
For any $\Bc$-category $(\M,A)$, the $\Bc$-category $(\Sb, \C A)$ of its singletons is complete. If $(\M,A)$ is symmetric, the $\Bc$-category $(\Sb,\C_{\s} A)$ of its symmetric singletons is symmetrically complete.
\end{prop}

\begin{proof}
Given a singleton $\psi: (\id_{t\psi},t\psi) \to (\Sb,\C A)$ of $(\Sb,\C A)$, we want to show that there is some singleton $\phi: (\id_{t\phi},t\phi) \to (\M,A)$ such that for any other singleton $\sigma: (\id_{t\s},t\s) \to (\M,A)$ of $A$, we have $\psi(\sigma) = \sigma^{*}\phi$. Consider the distributor $\phi: (\id_{t\psi},t\psi) \to (\M,A)$ defined by: \[\phi(a) = \psi(\M(-,a)).\]
Then, denoting by $\Sb_2$ the structural matrix of the category of singletons of $(\Sb,\C A)$, we have: 
\begin{align*}
\Sb(\sigma,\phi) &= \sigma^* \phi\\ 
&= \int^{a: A}_{\M} \sigma^*(a)\psi(\M(-,a))\\ 
&= \int^{a:A}_{\M} \sigma(a) \Sb_2(\Sb(-,\M(-,a)),\psi)\\
&= \int^{a:A}_{\M} \sigma^*(a)(\Sb(\M(-,a),a)\psi)\\
&= \int^{a:A,\ \gamma:\C A}_{\M,\Sb} \sigma^*(a)\Sb(\M(-,a),\gamma)\psi(\gamma)\\
&= \int^{a:A,\ \gamma:\C A}_{\M,\Sb} \sigma^*(a)\gamma(a)\psi(a)\\
&= \int^{\gamma:\C A}_{\Sb} (\sigma^*\gamma)\psi(\gamma)\\
&= \int^{\gamma:\C A}_{\Sb} \Sb(\sigma,\gamma)\psi(\gamma)\\
&= \Sb_2(\Sb(-,\sigma),\psi)\\
&= \psi(\sigma)
\end{align*}

Now we must prove that $\phi$ is indeed a map in $\Dist(\Bc)$. First, to prove that it is a distributor, let us define $\delta^{\phi}_{ab}:= \delta^{\psi}_{\M(-,a)\M(-,b)}$. The distributor coherence conditions of $\phi$ are direct consequences of those of $\psi$. We define the right adjoint $\phi^*$ of $\phi$ as $\phi^*(a) = \psi^*(\M(-,a))$; note that $\psi^*\psi = \int^{\s: \C A}_{\Sb} \psi^*(\s)\psi(\s) = \int^{\s: \C A} \phi^*\s \s^* \phi = \phi^*\phi$, hence we can define the unit $\eta^{\phi}$ to be $\eta^{\psi}$. The same happens with the counit since $\phi(a)\phi^*(b) = \psi(\M(-,a))\psi^*(\M(-,b))$ and $\e^{\psi}_{ab}$ goes into $\Sb(\M(-,a),\M(-,b))$, which is equal to $\M(a,b)$; therefore, $\phi$ is indeed a singleton of $(\M,A)$. This shows that $(\Sb,\C A)$ is complete. This also proves the symmetric case, as $\phi$ is symmetric whenever $\psi$ is.
\end{proof}

\begin{lem}[$\Bc$-enriched co-Yoneda lemma]
A $\Bc$-category is complete if and only if it is the colimit of all its singletons in $\Dist(\Bc)$. A symmetric $\Bc$-category is complete if and only if it is the colimit of all its symmetric singletons in $\Dist_{\s}(\Bc)$.
\end{lem}   

\begin{proof}
Suppose that $(\M,A)$ is a complete $\Bc$-category, and let there be a diagram as follows in $\Dist(\Bc)$: 
\[
\begin{tikzcd}
(\id_{t\s},t\s)
\arrow[dr,"{\s}",bend right=10]
\arrow[ddr,"{\s'}"',bend right = 20]
\arrow[rr,"f"]
&& (\id_{t\tau},t\tau) 
\arrow[dl,"{\tau}"', bend left = 10]
\arrow[ddl,"{\tau'}", bend left = 20]\\
& (\M,A)
\arrow[d,"\phi",dashed]\\
& (\N,C)
\end{tikzcd}
\]

Note that the cells $f: (\id_{t\s},t\s) \to (\id_{t\tau},t\tau)$ considered are precisely those that make the triangle with $\s$ and $\tau$ commute up to a $2$-cell: i.e. there exists some $2$-cell $\alpha: \tau f \To \s$. In other words, they are the $1$-cells in $\Bc$ which are the base for the colimit which defines $\Pb(\tau,\s)$. Now the goal is to produce the distributor $\phi$ which makes the whole diagram commute. Any such $\phi: (\M,A) \to (\N,C)$ would need to satisfy $\s' = \phi\s$, i.e. for any $c \in C$, $\s'(c) = \int^{a: A}_{\M} \phi(c,a)\s(a) = \int^{a:A}_{\M} \phi(c,a)\M(a,a_{\s}) = \phi(c,a_{\s})$, where $a_{\s}$ is the element representing $\s$ by completeness of $(\M,A)$. Now define $\phi$ by $\phi(c,a) = \s_{a}(c)$ for all $a \in A$, $c \in C$, where $\s_a$ is the presingleton on $(\N,C)$ corresponding to $\M(-,a)$. To prove that it is a distributor we need to exhibit a $2$-cell $\delta^{\phi}_{c_1c_2a_2a_1}$ for each $a_1,a_2 \in A$, $c_1,c_2 \in C$. Because $\iota$ defines a morphism of distributors $\M(-,a_2)\M(a_2,a_1) \To \M(-,a_1)$, we have a $2$-cell $\alpha: \s_{a_2} \M(a_2,a_1) \To \s_{a_1}$ in $\Dist(\Bc)$, i.e. a morphism of distributors; in return this enables us to define $\delta_{c_1c_2a_2a_1}^{\phi}$ as $\alpha_{c_1} \bullet (\delta^{\s_{a_2}}_{c_1c_2} * 1_{\M(a_2,a_1)})$, which is isomorphic to $\delta^{\s_1}_{c_1c_2} \bullet (1_{\N(c_1,c_2)} * \aa_{c_2})$; this is easily proven using the coherence conditions of $\aa$ as a morphism of distributors. The coherence conditions of $\delta^{\phi}$ are direct consequences of that of $\s_a$ and of the fact that $\aa$ is a morphism of distributors; we do not prove them here.\\ 

Now suppose that $(\M,A)$ is the colimit of all its singletons, denote by $\yf: (\M,A) \to (\Sb,\C A)$ the Yoneda $\Bc$-functor, then it defines a representable distributor $\yf_!: (\Sb,\C A) \to (\M,A)$ by $\yf_!(a,\s) = \s(a)$. We will show that $(\Sb,\C A)$ satisfies the universal property of the colimit; let there be a $\Bc$-category $(\N,C)$ as in the above diagram, then there exists a distributor $\phi: (\M,A) \to (\N,C)$ making the diagram commute. Then consider the distributor $\phi\yf_!: (\Sb,\C A) \to (\N,C)$, we have $\phi\yf_!(c,\s) = \int^{a: A }_{\M} \phi(c,a)\s(a) = (\phi\s)(c,a)$, thus proving that the diagram commute and therefore that $(\M,A) \simeq (\Sb,\C A)$, i.e. that $(\M,A)$ is complete.

\end{proof}

\begin{coro}
Let $(\M,A)$ be a $\Bc$-category. Then the completion of $(\M,A)$ is obtained as the colimit of all singletons of $(\M,A)$.
\end{coro}

\begin{deft}
We denote by $\Catk(\Bc)$ the full subcategory of $\Cat(\Bc)$ whose objects are the complete $\Bc$-categories. We denote by $\Catsk(\Bc)$ the full subcategory of $\Cat_{\s}(\Bc)$ whose objects are the symmetrically complete symmetric $\Bc$-categories.
\end{deft}

\begin{ex}
Taking into account all of our above examples, we can now give a list:
\begin{enumerate}
\item $\Catsk(\Bc_{\Set})$ is the category of all small Karoubi-complete categories.
\item $\Catsk(\R(X))$ is the category of sheaves of sets over the topological space $X$.
\item $\Catk(\R(X))$ is the category of sheaves of posets over the topological space $X$.
\item For any site $(\C,J)$, denoting by $\R(C,J)$ the quantaloid of $J$-closed sieves associated to $(\C,J)$, $\Catsk(\R(\C,J))$ is the topos $\Sh(\C,J)$ of sheaves over the site. See~\cite{HeymansThesis} for more details.
\end{enumerate}
\end{ex}

\begin{deft}
The assignment $(\M,A) \mapsto (\Sb,\C A)$ defines a functor $\C: \Cat(\Bc) \to \Catk(\Bc)$. For any $\Bc$-functor $f: (\M,A) \to (\N,C)$, define $\C f(\s)$ by $f_! \s$ for any singleton $\s$ of $(\M,A)$. If $\Bc$ is endowed with an involution, this functor restricts to a functor $\C_{\s}: \Cat_{\s}(\Bc) \to \Catsk(\Bc)$.
\end{deft}

\begin{prop}
The functor $\C$ is left adjoint to the inclusion functor $i: \Catk(\Bc) \to \Cat(\Bc)$. The functor $\C_{\s}$ is left adjoint to the inclusion functor $i_{\s}: \Catsk(\Bc) \to \Cat_{\s}(\Bc)$.
\end{prop}

\begin{proof}
We shall give the unit and counit of this adjunction:
\begin{enumerate}
\itb For any complete $\Bc$-category $(\M,A)$, the counit sends any singleton to the unique element representing it:
\begin{align*}
\e_{(\M,A)}: (\Sb, \C A) &\to (\M,A)\\
\sigma = \M(-,a) &\mapsto a
\end{align*}
\itb For any $\Bc$-category $(\M,A)$, the unit sends any element $a \in A$ to its representable singleton:
\begin{align*}
\eta_{(\M,A)}: (\M,A) &\to (\Sb, \C A)\\
a &\mapsto \M(-,a)
\end{align*}
\end{enumerate}
We must show that these are indeed $\Bc$-functors, and then that they satisfy the triangle identities. Let us fix $(\M,A)$ so that we can stop writing it as index for a moment, and let us start with $\eta$. For any $a,b \in A$, we have: 
\begin{enumerate}
\itb $\eta(a) = \M(-,a)$
\itb $\eta_{ab} = 1_{\M(a,b)}: \M(a,b) \To \Sb(\M(-,a),\M(-,b)) = \M(a,b)$
\end{enumerate}

Now the $\Bc$-functor conditions are rewritten as:

\begin{enumerate}
\itb $\iota_{abc}^{\M} \simeq \iota_{\M(-,a)\M(-,b)\M(-,c)}^{\Sb}$
\itb $\rho_{a}^{\M} \simeq \rho_{\M(-,a)}^{\Sb}$
\end{enumerate}

Which is precisely the object of lemma~\ref{RhoIotaSing}. Now for $\e$, we have: 

\begin{enumerate}
\itb $\e(\M(-,a)) = a$
\itb $\e_{\M(-,a)\M(-,b)} = 1_{\M(a,b)}: \Sb(\M(-,a),\M(-,b)) = \M(a,b) \To \M(a,b)$
\end{enumerate}

And the $\Bc$-functor conditions are the same. Therefore, both $\eta_{(\M,A)}$ and $\e_{(\M,A)}$ are indeed $\Bc$-functors, meaning that $\eta$ and $\e$ are well-defined natural transformations. Now we must show that they satisfy the triangle identities. Those are very straightforward to prove, as they amount to:

\begin{enumerate}
\itb For any $\Bc$-category $(\M,A)$, the following composite is the identity $\Bc$-functor: 
\[
\begin{tikzcd}[row sep = -5, column sep = 30]
(\Sb,\C A) 
\arrow[r, "{\C(\eta_{(\M,A)})}"]
& (\Sb_2, \C \C A) 
\arrow[r, "{\e_{(\Sb_2,\C\C A)}}"]
& (\Sb, \C A) \\
\s
\arrow[r, mapsto,  shorten <= 11, shorten >= 5]
& \Sb(-,\s) 
\arrow[r, mapsto,  shorten <= 5, shorten >= 11]
& \s
\end{tikzcd}
\]
\itb For any skeletal complete $\Bc$-category $(\M,A)$, the following composite is the identity $\Bc$-functor:
\[
\begin{tikzcd}[row sep = -5, column sep = 30]
(\M,A)
\arrow[r, "{\eta_{(\M,A)}}"]
& (\Sb, \C A)
\arrow[r, "{\e_{(\M,A)}}"]
& (\M,A)\\
a
\arrow[r, mapsto,  shorten <= 11]
& \M(-,a)
\arrow[r, mapsto,  shorten >= 11]
& a
\end{tikzcd}
\]

\end{enumerate}

The proof restricts without issue to the symmetric case.
\end{proof}

\begin{coro}
The completion of an already complete $\Bc$-category yields the same original $\Bc$-category. For any $\Bc$-category $(\M,A)$, its completion $(\Sb,\C A)$ satisfies the following universal property: any $\Bc$-functor $(\M,A) \to (\N,C)$, with $(\N,C)$ a complete $\Bc$-category, factorizes uniquely through $\yf_{(\M,A)}: (\M,A) \to (\Sb, \C A)$.
\end{coro}

\section{Complete $\Bc$-categories and $2$-presheaves over $\Map(\Bc)$}

The goal of this section is to define a diagram of categories as follows, where $\si$ is some kind of Grothendieck construction functor and $P$ is some kind of pseudofunctor of ``fibers'', and show that there is an adjunction $\C\si \dashv P$:

\[
\begin{tikzcd}[row sep=40, column sep = 20]
{\Cat(\Bc)}
\arrow[d,"{\C}"', shift right=2]
\arrow[d,"{\scriptstyle\dashv}"description, phantom]
&& {[\Map(\Bc)^{\coop},\Cat]}
\arrow[ll,"{\si}"]\\ 
{\Catk(\Bc)}
\arrow[u,"{i}"', shift right=2]
\arrow[urr,"P",bend right=30, start anchor=real east]
\end{tikzcd}
\]

This construction also exists in a symmetric version, the interest of which will be demonstrated in the next section of examples. By $\Map(\Bc)$ we mean the bicategory of $\Bc$ with the same objects as $\Bc$, but in which, for $a,b$ in $\Ob(\Bc)$, $\Map(\Bc)(a,b)$ is the full subcategory of $\Bc(a,b)$ whose objects are $1$-cells with a right adjoint. \\ 

Throughout all this section, $[\Map(\Bc)^{\op},\Cat]$ will be the category of \textit{pseudofunctors} going from $\Map(\Bc)^{\op}$ to $\Cat$, and of \textit{oplax natural transformations} between them. We do not consider here the $2$-structure of that category, which would require delving into the calculus of modifications between natural transformations.

\subsection{The Grothendieck construction $\si$}

Recall that the usual Grothendieck construction (see, for instance,~\cite{CZStacks} for a thorough review) takes an indexed category, i.e. a functor $\C^{\op} \to \Cat$ and yields a category (a fibration) over $\C$. We shall modify this definition to construct, from a pseudofunctor $\Map(\Bc)^{\coop} \to \Cat$ an associated $\Bc$-category. 

\begin{deft}\label{GrothConstrDef}
Let $F: \Map(\Bc)^{\coop} \to \Cat$ be a pseudofunctor. Consider the $\Ob(\Bc)$-typed set $\si F$ of elements of $F$ defined by $$\si F = \bigsqcup_{x \in \Ob(\Bc)} \Ob(F(x)).$$ For any $a,b \in \si F$, define $\N(a,b)$ as the colimit of the following diagram in $\Bc(tb,ta)$:
\begin{enumerate}
\itb The objects are maps $f: tb \to ta$ such that there is an arrow $u: b \to F(f)(a)$ in $F(tb)$ (taken as many times as there is such a $u$). 
\itb For any two such pairs $(f_1,u_1)$ and $(f_2,u_2)$, the arrows considered are the $2$-cells $\alpha: f_1 \To f_2$ such that $F(\alpha)(a)u_2 \simeq u_1$.
\end{enumerate}
Then $(\N,\si F)$ is a $\Bc$-category.
\end{deft}

\begin{proof}
We start by defining the necesary families of $2$-cells $(\rho_a)_{a \in A}$ and $(\iota_{abc})_{a,b,c \in A}$, for $A = \si F$. Let there be some $a \in A$. Since $F$ is a functor, we have $F(\id_a)(a) = a$, hence $\id_a$ is an object of the colimiting diagram defining $\N(a,a)$ and hence there is a canonical $2$-cell $\rho_a: \id_a \To \N(a,a)$.

Now consider three elements $a,b,c \in A$; we want to define $$\iota_{abc}: \N(a,b)\N(b,c) \To \N(a,c).$$ Since $\N(a,c)$ is defined as a colimit, we can use, as we did for $\rho$, the canonical injections. However, we will also describe $\N(a,b)\N(b,c)$ as a colimit, using the fact that pre- and post-composition commute with colimits: indeed, $\N(a,b)\N(b,c)$ is a colimit of the diagram
\[
\begin{tikzcd}
& \N(a,b)\N(b,c) \\ 
f_1g_1 
\arrow[ur]
\arrow[rr,"\alpha * \beta"]
&& f_2g_2
\arrow[ul]
\end{tikzcd}
\]
in the category $\Bc(tc,ta)$, indexed also by the following data:
\begin{enumerate}
\itb $u_1: b \to F(f_1)(a)$
\itb $u_2: b \to F(f_2)(a)$
\itb $v_1: c \to F(g_1)(b)$
\itb $v_2: c \to F(g_2)(b)$
\end{enumerate}
(so that $(f_1,u_1)$ and $(f_2,u_2)$ are objects of the colimiting diagram defining $\N(a,b)$ and that $(g_1,v_1)$ and $(g_2,v_2)$ are objects of the colimiting diagram defining $\N(b,c)$). We therefore have $\alpha: f_1 \To f_2$ and $\beta: g_1 \To g_2$ satisfying $F(\alpha)(a)u_2 \simeq u_1$ and $F(\beta)(b)(v_2) \simeq v_1$.

Now we want to give $2$-cells $f_1g_1 \To \N(a,c)$ and $f_2g_2 \To \N(a,c)$ which commute with $\alpha * \beta$. First, let us drop the index for a moment as we produce a $w: c \to F(fg)(a)$. We have $F(g)(u): F(g)(b) \to F(g)(F(f)(a))$; up to isomorphism (because $F$ is a pseudofunctor and not a strict $2$-functor) this is an arrow $F(g)(b) \to F(fg)(a)$; define $w = F(g)(u)\cdot v: c \to F(fg)(a)$.

The next step is to show that we have $F(\alpha * \beta)(a) w_2 \simeq w_1$ (we resume taking the indexes, $w_1$ and $w_2$ are respectively the $w$ constructed as above for $f_1g_1$ and $f_2g_2$). We start by expliciting $F(\alpha * \beta)(a)$: because $F$ goes into $\Cat$, that amounts at taking the \textit{Godement product} of $F(\beta)$ and $F(\alpha)$; we get: 
\begin{align*}
F(\alpha * \beta)(a)        &=(F(\beta) * F(\alpha))(a) \\
                            &\simeq F(\beta)(F(f_2)(a)) \cdot F(g_1)(F(\alpha)(a))\\
                            &\simeq F(g_2)(F(\alpha)(a)) \cdot F(\beta)(F(f_1)(a))
\end{align*}

What we want to show is therefore:

\[F(\beta)(F(f_1)(a))F(g_2)(F(\alpha)(a))F(g_2)(u_2)v_2 \simeq F(g_1)(u_1)v_1\]

So we compute: 

\begin{eqnarray*}
& & F(\beta)(F(f_1)(a)) \cdot F(g_2)(F(\alpha)(a)) \cdot F(g_2)(u_2) \cdot v_2\\  
&=& F(\beta)(F(f_1)(a)) \cdot F(g_2)(F(\alpha)(a) \cdot u_2) \cdot v_2\\
&\simeq& F(\beta)(F(f_1)(a)) \cdot F(g_2)(u_1) \cdot v_2
\end{eqnarray*}

Now consider the following commutative diagram of naturality of the natural transformation $F(\beta): F(g_2) \To F(g_1)$ along the arrow $u_2: b \to F(f_2)(a)$; the diagram is in the category $F(tc)$:

\[\begin{tikzcd}[row sep = 20, column sep = 60]
F(g_2)(b)
\arrow[r,"F(\beta)(b)"]
\arrow[d,"F(g_2)(u_1)"']
& F(g_1)(b)
\arrow[d,"F(g_1)(u_1)"]\\
F(f_1g_2)(a)
\arrow[r,"F(\beta)(F(f_1)(a))"']
& F(f_1g_1)(a)
\end{tikzcd}\]

The naturality of $F(\beta)$ shows that we have:

\begin{align*}
F(\beta)(F(f_1)(a)) \cdot F(g_2)(u_1) \cdot v_2 &\simeq F(g_1)(u_1) \cdot F(\beta)(b) \cdot v_2\\
&\simeq F(g_1)(u_1) \cdot v_1\\
&= w_1
\end{align*}

This shows that $F(\alpha * \beta)(a)w_2 \simeq w_1$, thus proving that there are canonical injections $f_1g_1 \to \N(a,c)$ and $f_2g_2 \to \N(a,c)$ making the following diagram commute:

$$\begin{tikzcd}
& \N(a,c) \\
f_1g_1
\arrow[ur]
\arrow[rr,"\alpha * \beta"']
&& f_2g_2
\arrow[ul]
\end{tikzcd}$$

Using the universal property of $\N(a,b)\N(b,c)$ as a colimit, this yields a $2$-cell $\iota_{abc}: \N(a,b)\N(b,c) \To \N(a,c)$. \\

Let us sum up the definitions of $\iota$ and $\rho$. For $f: tb \to ta$ with some $u: b \to F(f)(a)$, we will denote by $e_f: f \To \M(b,a)$ the canonical inclusion. Then $\rho_a = e_{\id_a}$; moreover $\iota_{abc}$ is defined by the commutativity of the following diagram: 

\[
\begin{tikzcd}[row sep = 30, column sep = 40]
& \M(a,c)\\ 
fg 
\arrow[ur,"{e_{fg}}", bend left = 20, end anchor=west]
\arrow[r,"{e_f * e_g}"']
& \M(a,b)\M(b,c)
\arrow[u,"{\iota_{abc}}"']
\end{tikzcd}
\]

Now proving the coherence relations becomes easier, notably because we can effectively check them at the base level of the colimits. For any $f: tb \to ta$ with some $u: b \to F(f)(a)$ in $F(tb)$, we have:
\begin{eqnarray*}
\iota_{aab} \bullet (\rho_a * 1_{\M(a,b)}) \bullet e_f &=& \iota_{aab}\bullet (e_{\id_{ta}} * e_f)\\ 
&=& e_{\id_{ta} \cdot f}\\
&=& e_f
\end{eqnarray*}

Hence $\iota_{aab} \bullet (\rho_a * 1) = 1_{\M(a,b)}$; a similar computation shows the other unitality condition. Now for the associativity condition, consider three morphisms $f: tb \to ta$, $g: tc \to tb$ and $h: td \to tc$ in the base of the colimit, with their respective canonical inclusions $e_f,e_g,e_h$. Then:

\begin{eqnarray*}
& &\iota_{acd} \bullet (\iota_{abc} * 1_{\M(c,d)}) \bullet (e_f * e_g * e_h)\\
&=& \iota_{acd} \bullet ((\iota_{abc} \bullet (e_f * e_g)) * e_h)\\ 
&=& \iota_{acd} \bullet (e_{fg} * e_h)\\ 
&=& e_{fgh}\\ 
&=& \iota_{abd} \bullet (e_f * e_{gh})\\ 
&=& \iota_{abd} \bullet (e_f * (\iota_{bcd} \bullet (e_g * e_h)))\\ 
&=& \iota_{abd} \bullet (1_{\M(a,b)} * \iota_{bcd}) \bullet (e_f * e_g * e_h)
\end{eqnarray*}

\end{proof}

\begin{prop}
Let $F$ and $G$ be two pseudofunctors $\Map(\Bc)^{\coop} \to \Cat$, and let $\alpha: F \to G$ be an oplax-natural transformation between them. Consider the following function $\si \alpha: \si F \to \si G$: 
\begin{enumerate}
\itb For all $a \in \si F$, $(\si \alpha)(a) = \alpha_{ta}(a) \in \si G$.
\end{enumerate}
Then $\si \alpha$ defines a $\Bc$-functor.
\end{prop}

\begin{proof}
First of all, note that $\si \alpha$ is indeed a type preserving function, since $\alpha_{ta}(a) \in G(ta) = (\si G)_{ta}$. Now we want to define a $2$-cell: $$\N^F(a,b) \To \N^G(\alpha_{ta}(a),\alpha_{tb}(b))$$
for all $a,b \in \si F$. Because $\N^F(a,b)$ is defined as a colimit, it is the same as giving a $2$-cell $f \To \N^G(\alpha_{ta}(a),\alpha_{tb}(b))$ for any map $f: tb \to ta$ in $\Bc$ for which there is a $u: b \to F(f)(a)$ in $F(tb)$. Now the oplax-naturality of $\alpha$ is expressed by a $2$-cell (i.e. a natural transformation) $\alpha_{tb} F(f) \To G(f) \alpha _{ta}$ as in the following diagram in $\Cat$: 

\[
\begin{tikzcd}[sep=huge]
F(ta)
\arrow[r,"\alpha_{ta}"]
\arrow[d,"F(f)"']
& G(ta)
\arrow[d,"G(f)"]\\ 
F(tb)
\arrow[ru,Rightarrow,"\alpha(f)", shorten <= 10, shorten >= 10]
\arrow[r,"\alpha_{tb}"']
& G(tb)
\end{tikzcd}
\]

Now consider $\alpha_{tb}(u): \alpha_{tb}(b) \to \alpha_{tb}(F(f)(a))$; composing with $\alpha(f)(a): \alpha_{tb}(F(f)(a))$ $\To G(f)(\alpha_{ta}(a))$ yields an arrow $v:= \alpha(f)(a) \cdot \alpha_{tb}(u)$ in $G(tb)$, going from $\alpha_{tb}(b)$ to $G(f)(\alpha_{ta}(a))$. This make $f$ into an element of the base of the colimiting cocone defining $\N^{G}(\alpha_{ta}(a),\alpha_{tb}(b))$ and gives a canonical injection $e_f^G: f \To \N^{G}(\alpha_{ta}(a),\alpha_{tb}(b))$, ultimately producing a $2$-cell $(\si \aa)_{ab}: \N^{F}(a,b) \To \N^{G}(\si \aa (a),\si \aa (b))$.

Now we show the coherence properties for $\si \aa$. First note that the following diagram in $\Bc(tb,ta)$ is commutative, for any map $f: tb \to ta$ with a $u: b \to F(f)(a)$ in $F(tb)$, by definition of $(\si \aa)_{ab}$:

\[
\begin{tikzcd}[row sep = 30, column sep = 40]
& \N^{G}(\si \aa (a),\si \aa (b))\\ 
f 
\arrow[ur,"{e_f^G}", bend left = 20, end anchor=west]
\arrow[r,"{e_f^F}"']
& \N^F(a,b)
\arrow[u,"{(\si \aa)_{ab}}"']
\end{tikzcd}
\]

From that we can compute, with $f: tb \to ta$ and $g: tc \to tb$ maps in $\Bc$ in the base of the colimiting cocone defining $\N^F(a,b)$:

\begin{eqnarray*}
& &(\si \aa)_{ac} \bullet \iota_{abc} \bullet (e_f^F * e_g^F)\\ 
&=&(\si \aa)_{ac} \bullet (e_{fg}^F)\\ 
&=& e^G_{fg}
\end{eqnarray*}

And in the other sense: 

\begin{eqnarray*}
& &\iota_{\si\aa(a)\si\aa(b)\si\aa(c)} \bullet ((\si \aa)_{ab} * (\si\aa)_{bc}) \bullet (e_f^F * e_g^F)\\ 
&=&\iota_{\si\aa(a)\si\aa(b)\si\aa(c)} \bullet (((\si \aa)_{ab} \bullet e_f^F) * ((\si\aa)_{bc} \bullet e_g^F))\\ 
&=&\iota_{\si\aa(a)\si\aa(b)\si\aa(c)} \bullet (e_f^G * e_g^G)\\ 
&=&e_{fg}^G
\end{eqnarray*}

This shows that $(\si\aa)$ respects composition. Now for the unitality part it is even easier:

\begin{eqnarray*}
& &(\si\aa)_{aa} \bullet \rho_a\\ 
&=&(\si\aa)_{aa} \bullet e_{\id_{ta}}^F\\ 
&=&e_{\id_{ta}}^G\\ 
&=&\rho_{\si\aa(a)}
\end{eqnarray*}

\end{proof}

\begin{prop}
The above propositions define a functor: $$\si: [\Map(\Bc)^{\coop},\Cat] \to \Cat(\Bc)$$
\end{prop}

\begin{proof}
We need to prove that $\si$ respects composition and identities. Let $\alpha: F \to G$ and $\beta: G \to H$ be two oplax-natural transformations between pseudofunctors. Then for any $a \in \si F$, $\si(\beta \bullet \aa)(a) = (\beta \bullet \aa)_{ta}(a) = \beta_{ta} (\aa_{ta}(a)) = (\si\beta)\cdot(\si \aa)(a)$. Moreover, for any map $f: tb \to ta$ in $\Bc$ with $u: b \to F(f)(a)$, we have: 

\begin{eqnarray*}
& &(\si(\beta \aa))_{ab} \bullet e_f^F\\ 
&=&e_f^H\\ 
&=&(\si \beta)_{\si\aa(a)\si\aa(b)} \bullet e_f^G\\ 
&=&(\si\beta)_{\si\aa(a)\si\aa(b)}\bullet(\si\aa)_{ab}\bullet e_f^F\\ 
&=&((\si \beta)(\si\aa))_{ab} \bullet e_f^F
\end{eqnarray*}

This shows that as $\Bc$-functors we have an equality $(\si \beta\aa) = (\si\beta)(\si\aa)$. Now denoting by $1_F$ the identity oplax-natural transformation on a pseudofunctor $F: \Map(\Bc)^{\coop} \to \Cat$, we have $\si (1_F) = 1_{\si F}$; it is immediate at the level of elements of $\si F$, and for any $a,b \in \si F$ and any map $f: tb \to ta$ in the base of the colimiting cocone defining $\N(a,b)$, we have $(\si 1_F)_{ab} \bullet e_f = e_f$, yielding $(\si 1_F)_{ab} = 1_{\N(a,b)}$.

\end{proof}

\subsection{The pseudofunctor of ``fibers'' of a complete $\Bc$-category}

Given a complete $\Bc$-category $(\M,A)$, we can define a pseudofunctor $\pma: \Map(\Bc)^{\coop} \to \Cat$ by considering its ``fibers'' and the functors between them which are canonically induced by the completeness property.  

\begin{deft}\label{CompFib}
Let $\Bc$ be a closed locally cocomplete involutive bicategory. Let $(\M,A)$ be a complete $\Bc$-category. The \emph{pseudofunctor $\pma: \Map(\Bc)^{\coop} \to \Cat$ of fibers of $\Bc$} is defined as follows: 
\begin{enumerate}
\itb For any $x \in \Ob(\Bc)$, $\pma(x)$ is a category, described by:
\begin{enumerate}
\itb Objects of $\pma(x)$ are elements of $A$ which are of type $x$.
\itb For any two objects $a,b$ of $\pma(x)$, an arrow $a \to b$ is a $2$-cell $\id_x \To \M(b,a)$ in $\Bc$.
\itb The identity arrow on an object $a$ of $\pma(x)$ is given by the $2$-cell $\rho_a: \id_x \To \M(a,a)$.
\itb The composite of two arrows $\theta: a \to b$ and $\chi: b \to c$ in $\pma(x)$ is given by the $2$-cell $\iota_{cba} \bullet (\chi * \theta): \id_x \To \M(c,a)$
\end{enumerate}
\itb For any map $\gamma: x_1 \to x_2$ in $\Bc$, $\pma(\gamma)$ is a functor going from $\pma(x_2)$ to $\pma(x_1)$, defined by:
\begin{enumerate}
\itb For any object $a \in \pma(x_2)$, $\pma(\gamma)(a)$ is the unique element $\gamma \cdot a$ which represents the singleton $\M(-,a)\gamma$. We therefore have: \[\M(-,\pma(\gamma)(a)) = \M(-,a)\gamma\]
\itb For any arrow $\theta: a \to b$ in $\pma(x)$, $\pma(\g)(\theta)$ is the $2$-cell $(1_{\g^*} * \theta * 1_{\g}) \bullet \eta_{\g}$.
\end{enumerate}
\itb For any $2$-cell $\alpha: \g_1 \To \g_2$ in $\Bc(x,y)$ with $x,y \in \Ob(\Bc)$, $\pma(\alpha)$ is a natural transformation of functors $\pma(\alpha): \pma(\g_2) \To \pma(\g_1)$. It is defined by:
\begin{enumerate}
\itb For all $a \in \Ob(\pma(y))$, $\pma(\alpha)_a$ is an arrow from $\pma(\g_2)(a)$ to $\pma(\g_1)(a)$ in $\pma(x)$, given by $(1_{\g_1^{*}} * \rho_a * \alpha) \bullet \eta_{\g_1}$.
\end{enumerate}
\end{enumerate}
\end{deft}

Here the fact that $(\M,A)$ is complete is crucial to the definition of $\pma$, as it is precisely what enables us to define the actions of the maps on the fibers. The proof that $\pma$ is well-defined relies on the following lemma, which expresses the fact that this action of the maps on the fibers is well-behaved.

\begin{lem}
Let $(\M,A)$ be a complete $\Bc$-category. Consider a map $\gamma: x_1 \to x_2$ in $\Bc$, with its right adjoint $\gamma^*$. Then there is an action of $\gamma$ on the fibers of $(\M,A)$, going from $A_{x_2}$ to $A_{x_1}$, defined as follows: for any $a \in A$ of type $x_2$, $a\cdot \g$ is defined as the unique element which represents the singleton $\M(-,a)\gamma$. We therefore have $\M(b,a\cdot \g) = \M(b,a)\gamma$ for any $b \in A$. Moreover, for any $b \in A$, we have $\M(a\cdot \g,b) = \gamma^*\M(a,b)$. This action satisfies the following properties for all $a,b,c \in A$: 
\begin{enumerate}
\item $\iota_{abc\cdot \g} = \iota_{abc} * 1_{\g}$
\item $\iota_{a\cdot\g bc} = 1_{\g^*} * \iota_{abc}$
\item $\iota_{ab\cdot\g c} = \iota_{abc} \bullet (1_{\M(a,b)} * \e_{\g} * 1_{\M(b,c)})$
\item $\rho_{a\cdot\g} = (1_{\g^*} * \rho_a * 1_{\g}) \bullet \eta_{\g}$
\end{enumerate}
\end{lem}

\begin{proof}
First of all let us prove that $\M(a\cdot \g,b) = \g^* \M(a,b)$. Because we have an equality of singletons $\M(-,a)\g = \M(-,a\cdot \g)$, we also have equality between their right adjoints $\g^*\M(a,-) = \M(a\cdot \g,-)$.

Now for the properties of the action, we are going to prove them using the fact that $\M(-,x)\cdot \gamma$ is a singleton, which is represented by $x\cdot\g$, and therefore we have an equality between singletons $\M(-,x)\cdot\g = \M(-,x\cdot\g)$. This implies more than just an equality at the level of $1$-cells, but also at the level of $2$-cells. In what follows, let us denote by $\s$ and $\tau$ the two singletons: $\s = \M(-,x)\g$ and $\tau = \M(-,x\cdot \g)$. We have, by definition of these distributors, for any two $a,b \in A$:
\begin{enumerate}
\itb $\delta^{\sigma}_{ab} = \iota_{abx} * 1_{\g}$
\itb $\delta^{\tau}_{ab} = \iota_{ab(x\cdot \g)}$
\end{enumerate}

This immediately yields the first proposition. For the second, we must also insist that because they are equal as singletons, they are also equal as maps in $\Dist(\Bc)$, hence $\s^* = \tau^*$, and the equality of the natural transformations $\delta^{\s^*}$ and $\delta^{\tau^*}$ gives the second proposition. For the third and fourth properties, we need to use that there must be an equality between the unities and counities of the adjunctions $\s \dashv \s^*$ and $\tau \dashv \tau^*$. We have: 

\begin{enumerate}
\itb $\e^{\s}_{ab} = \iota_{axb} \bullet (1_{\M(a,x} * \e_{\g} * 1_{\M(x,b)})$
\itb $\eta^{\s} = (1_{\g^*} * \rho_x^{\M} * 1_{\g})$
\itb $\e^{\tau}_{ab} = \iota_{a(x\cdot \g)b}$
\itb $\iota^{\tau} = \rho_{x\cdot \g}^{\M}$
\end{enumerate}

And the last two properties immediately follow.
\end{proof}

\begin{proof}[Proof of the definition~\ref{CompFib}]
We first show that $\pma(\g)$ is indeed a functor. For all $a \in \Ob(\pma(x))$, $\pma(\g)(\id_a) = (1_{\g^*} * \rho_a * 1_{\g}) \bullet \eta_{\g} = \rho_{a \cdot \g)} = \id_{\pma(\g)(a)}$. For all $a,b,c \in \Ob(\pma(x))$, for all $\theta: a \to b$, $\chi: b \to c$, we have:
\begin{align*}
\pma(\g)(\chi\cdot\theta) &= \pma(\g)(\iota_{cba}\bullet (\chi * \theta))\\
&= (1_{\g^*} * (\iota_{cba} \bullet (\chi * \theta)) * 1_{\g}) \bullet \eta_{\g}\\
&= (1_{\g^*} * \iota_{cba} * 1_{\g}) \bullet (1_{\g^*} * \chi * \theta * 1_{\g}) \bullet \eta_{\g}\\
&= \iota_{c\cdot\g b a\cdot \g} \bullet (1_{\g^*} * \chi * \theta * 1_{\g}) \bullet \eta_{\g}
\end{align*}

And:
\begin{align*}
&\pma(\g)(\chi) \cdot \pma(\g)(\theta)\\
&= \iota_{c\cdot\g b\cdot\g a\cdot\g} \bullet (\pma(\g)(\chi) * \pma(\g)(\theta))\\
&= \iota_{c\cdot\g b\cdot\g a\cdot\g} \bullet (((1_{\gc} * \chi * 1_{\g}) \bullet \eta_{\g}) * ((1_{\gc} * \theta * 1_{\g}) \bullet \eta_{\g}))\\
&= \iota_{c\cdot\g b\cdot\g a\cdot\g} \bullet ((1_{\gc} * \chi * 1_{\g\gc} * \theta * 1_{\g}) \bullet (\eta_{\g} * \eta_{\g})\\
&= \iota_{c\cdot\g b a\cdot\g} \bullet (1_{\M(c\cdot g,b)} * \e_{\g} * 1_{\M(b,a\cdot \g)}) \bullet ((1_{\gc} * \chi * 1_{\g\gc} * \theta * 1_{\g}) \bullet (\eta_{\g} * \eta_{\g})\\
&= \iota_{c\cdot\g b a\cdot\g} \bullet (1_{\g^*} * \chi * \e_{\g} * \theta * 1_{\g}) \bullet (\eta_{\g} * \eta_{\g})\\
&= \iota_{c\cdot\g b a\cdot\g} \bullet (1_{\g^*} * \chi * \e_{\g} * \theta * 1_{\g}) \bullet (1_{\g^*} * 1_{\g} * \eta_{\g}) \bullet \eta_{\g}\\
&= \iota_{c\cdot\g b a\cdot\g} \bullet (((1_{\g^*} * \chi) \bullet 1_{\g^*}) * ((\e_{\g} * \theta * 1_{\g}) \bullet (1_{\g} * \eta_{\g}))) \bullet \eta_{\g}\\ 
&= \iota_{c\cdot\g b a\cdot\g} \bullet ((1_{\g^*} * \chi) * ((1_{\id} * \theta * 1_{\g}) \bullet (\e_{\g} * 1_{\id} * 1_{\g}) \bullet (1_{\g} * \eta_{\g}))) \bullet \eta_{\g}\\ 
&= \iota_{c\cdot\g b a\cdot\g} \bullet (1_{\g^*} * \chi * \theta * 1_{\g}) \bullet \eta_{\g}\\ 
\end{align*}

Now we must prove that $\pma(\alpha)$ is indeed a natural transformation, i.e. that for any $x,y \in \Ob(\Bc)$, for any $\g_1,\g_2: x \to y$ maps in $\Bc$, for any $a,b \in A_y$ and for any $\theta: \id_y \To \M(b,a)$, the following diagram commutes in $\pma(x)$ (this mean that the composition is not the same as in $\Bc$):

\[
\begin{tikzcd}[row sep = 70, column sep = 100, labels={description}]
a\cdot \g_1
\arrow[r,"(1_{\g_1^{*}} * \theta * 1_{\g_1}) \bullet \eta_{\g_1}"]
& b \cdot \g_1\\
a \cdot\g_2
\arrow[u,"(1_{\g_1^*} * \rho_a^{\M} * \aa) \bullet \eta_{\g_1}"]
\arrow[r,"(1_{\g_2^{*}} * \theta * 1_{\g_2}) \bullet \eta_{\g_2}"]
& b \cdot \g_2
\arrow[u,"(1_{\g_1^*} * \rho_b * \aa) \bullet \eta_{\g_1}"]
\end{tikzcd}
\]

A quick computation shows that the up/left term is equal to: \[(1_{\g_1^*} * \theta * \aa) \bullet \eta_{\g_1}\]

And the down/right term is equal to:\[(1_{\g_1^{*}} * (\e_{\g_2} \bullet (\alpha * 1_{\g_2^*})) * 1_{\g_2}) \bullet (\eta_{\g_2} * \eta_{\g_1})\]

The equality between these two terms is now a direct consequence of the interchange law and the triangle identities; this shows that $\pma(\alpha)$ is indeed a natural transformation.  \\

Let us now prove that $\pma$ is a pseudofunctor. There are several properties that we must show: 

\begin{enumerate}
\item For any $x,y \in \Ob(\Map(\Bc))$, $\pma$ restricted to $\Bc(x,y)$ is a functor:
\begin{enumerate}
\item $\pma(\beta \bullet \alpha) = \pma(\beta) \bullet \pma(\alpha)$
\item $\pma(1_{\g}) = 1_{\pma(\g)}$
\end{enumerate}
\item $\Id_{\pma(x)} \simeq \pma(\id_x)$, where $\Id$ is the identity functor.
\item There is a natural isomorphism $\pma(\g_2 \cdot \g_1) \simeq \pma(\g_1) \cdot \pma(\g_2)$.
\item Several coherence conditions.
\end{enumerate}

Starting with $1a$, let us fix $x,y \in \Ob(\Bc)$, let there be three maps $\g_1,\g_2,\g_3: x \to y$ in $\Bc$, and let there be two morphisms of maps $\alpha: \g_1 \To \g_2$ and $\beta: \g_2 \To \g_3$. Then:

\begin{eqnarray*}
& & (\pma(\aa) \bullet \pma(\beta))_a\\ 
&=& \pma(\aa)_a \cdot \pma(\beta)_a \ \text{in $\pma(x)$}\\ 
&=& \iota_{a\cdot \g_1 a\cdot\g_2 a\cdot \g_3} \bullet (((1_{\g_1^*} * \rho_a * \alpha) \bullet \eta_{\g_1}) * ((1_{\g_2^*} * \rho_a * \beta) \bullet \eta_{\g_2}))\\
&=& (1_{\g_1^*} * (\iota_{aaa} \bullet (1_{\M(a,a)} * \e_{\g_2} * 1_{\M(a,a)})) * 1_{\g_3}) \bullet\\ 
& & (1_{\g_1^*} * \rho_a * \alpha * 1_{\g_2^*} * \rho_a * \beta) \bullet (\eta_{\g_1} * \eta_{\g_2})\\ 
&=& (1_{\g_1^*} * (\iota_{aaa} \bullet (1_{\M(a,a)} * \rho_a) \bullet (\rho_a * (\e_{\g_2} \bullet (\alpha * 1_{\g_2^*})))) * \beta) \bullet (\eta_{\g_1} * \eta_{\g_2})\\
&=& (1_{\g_1^*} * \rho_a * \beta) \bullet ((\e_{\g_2} \bullet (\aa * 1_{\g_2})) * 1_{\g_2}) \bullet (\eta_{\g_1} * \eta_{\g_2})\\ 
&=& (1_{\g_1^*} * \rho_a * \beta) \bullet (1_{\g_1^*} * \alpha) \bullet \eta_{\g_1}\\ 
&=& (1_{\g_1^*} * \rho_a * (\beta \bullet \alpha)) \bullet \eta_{\g_1}\\ 
&=& \pma(\beta \bullet \alpha)_a
\end{eqnarray*}

The point $1b$ is rather trivial: 

\begin{align*}
\pma(1_{\g})_a &= (1_{\g^{*}} * \rho_a * 1_{\g}) \bullet \eta_{\g}\\
&= \rho_{a\cdot \g}\\
&= \id_{\pma(\g)(a)}\\
&= (1_{\pma(\g)})_a
\end{align*}

Now for the point $2$: 

\begin{align*}
\pma(\id_x)(a) &= a \cdot \id_x\\
&= a
\end{align*}

And, for $\theta: \id_x \To \M(b,a)$:

\begin{align*}
\pma(\id_x)(\theta) &= (1_{\id_x} * \theta * 1_{\id_x}) \bullet \eta_{\id_x}\\
&= \theta
\end{align*}

Hence the unitality isomorphism is even an equality. We go forth to number 3, starting with the remark that whenever it is defined, $a \cdot (\g_2 \cdot \g_1) = (a \cdot \g_2) \cdot \g_1$. This is because $\M(-,a \cdot (\g_2 \cdot \g_1))$ is defined as $\M(-,a) \cdot (\g_2 \cdot \g_1)$, which is always equal to $(\M(-,a)\g_2) \g_1$ (it is a right action). Hence:

\begin{align*}
\pma(\g_2\g_1)(a) &= a \cdot (\g_2\g_1)\\
&= (a\cdot \g_2)\cdot \g_1\\
&= \pma(\g_2)(a) \cdot \g_1\\
&= \pma(\g_1)(\pma(\g_2)(a))\\
&= (\pma(\g_1) \cdot \pma(\g_2))(a)
\end{align*}

And for $\theta: \id_x \To \M(b,a)$:

\begin{align*}
\pma(\g_2\g_1)(\theta) &= (1_{\g_1^*\g_2^*} * \theta * 1_{\g_2\g_1}) \bullet \eta_{\g_2\g_1}\\
&= (1_{\g_1^*} * 1_{\g_2^*} * \theta * 1_{\g_2} * 1_{\g_1}) \bullet (1_{\g_1^*} * \eta_{\g_2} * 1_{\g_1}) \bullet \eta_{\g_1}\\
&= (1_{\g_1^*} * ((1_{\g_2^*} * \theta * 1_{\g_2}) \bullet \eta_{\g_2}) * 1_{\g_1}) \bullet \eta_{\g_1}\\
&= (1_{\g_1^*} * \pma(\g_2)(\theta) * 1_{\g_1}) \bullet \eta_{\g_1}\\
&= \pma(\g_1)(\pma(\g_2)(\theta))\\
&= (\pma(\g_1) \cdot \pma(\g_2))(\theta)
\end{align*}

Again, the compositional isomorphism is an equality; this ensures that the coherence conditions are satisfied and shows that $\pma$ is actually a \textit{strict pseudofunctor}.
\end{proof}

When $(\M,A)$ is a symmetrically complete symmetric $\Bc$-category, the above definition still stands by considering only symmetric maps, which define symmetric singletons $\M(-,a)\g$ of $(\M,A)$. This means that we can also define a pseudofunctor $\pma$ for $(\M,A)$ being symmetrically complete, but this time it is a pseudofunctor $\pma: [\Map_{\s}(\Bc)^{\coop},\Cat]$.

\begin{prop}
Let $f: (\M,A) \to (\N,C)$ be a $\Bc$-functor between two complete $\Bc$-categories. Consider for any $x \in \Ob(\Bc)$ the functor $(P_f)_x: \pma(x) \to \pnc(x)$ which sends an object $a$ to $f(a)$ and an arrow $\theta: \id_x \To \M(b,a)$ to $f_{ab}\bullet \theta: \id_x \To \N(f(b),f(a))$. Then this is indeed a functor, and $P_f$ defines an oplax-natural transformation $\pma \to \pnc$.
\end{prop}

\begin{proof}Let us start by proving that for any $x \in \Ob(\Bc)$, $(P_f)_x$ is a functor. The identity arrow on $a \in A_x$ is $\rho_a$, and we have $(P_f)_x(\rho_a) = f_{aa} \bullet \rho_{a} = \rho_{f(a)f(a)} = \id_{f(a)}$ in $\pnc(x)$. Now for $\theta: a \to b$ and $\chi: b \to c$ in $\pma(x)$, the composite in $\pma(x)$ is the $2$-cell $\iota_{cba} \bullet (\chi * \theta)$. We have
\begin{align*}
(P_f)_x(\iota_{cba} \bullet (\chi * \theta)) &= f_{ca} \bullet \iota_{cba} \bullet (\chi * \theta)\\ 
&= \iota_{f(c)f(b)f(a)} \bullet (f_{cb} * f_{ba}) \bullet (\chi * \theta) \\ 
&= \iota_{f(c)f(b)f(a)} \bullet ((f_{cb} \bullet \chi) * (f_{ba} * \theta)) \\ 
&= (P_f)_x(\chi) \bullet (P_f)_x(\theta) \ \text{in $\pma(x)$}
\end{align*}  

This shows that $(P_f)_x$ is indeed a functor. Now because we want an oplax natural transformation, we must also add for any $\gamma: x \to y$ map in $\Bc$, a natural transformation (a $2$-cell in $\Cat$) $P_f(\g): (P_f)_x \cdot \pma(\g) \To \pnc(\g) \cdot (P_f)_y$, which satisfies several compatibility properties. As both $(P_F)_x \cdot \pma(\g)$ and $\pnc(\g) \cdot (P_F)_y$ are functors from $\pma(y)$ to $\pnc(x)$, we are brought to define for each $a \in A_y$, an arrow in $\pnc(x)$ going from $((P_F)_x \cdot \pma(\g))(a)$ to $(\pnc(\g)\cdot (P_F)_y)(a)$. To clarify, here are the expression of these composite functors on objects: 

\begin{alignat*}{3}
(P_F)_x \cdot \pma(\g): \pma(y)        &\to \pma(x)        &&\to \pnc(x)\\
                                    a   &\mapsto a\cdot \g  &&\mapsto f(a \cdot \g)
\end{alignat*}

\begin{alignat*}{3}
\pnc(\g) \cdot (P_F)_y: \pma(y) &\to \pnc(y)  &&\to \pnc(x)\\
                            a    &\mapsto f(a) &&\mapsto f(a)\cdot \g
\end{alignat*}

In other words, we want to find a $2$-cell from $f(a \cdot \g)$ to $f(a)\cdot \g$ in $\pnc(x)$, i.e. a $2$-cell $\theta: \id_x \To \N(f(a)\cdot \g,f(a\cdot \g))$ in $\Bc$. We define: 
\[\theta:= (1_{\g^*} * f_{a a\cdot \g}) \bullet (1_{\g^*} * \rho_a^{\M} * 1_{\g}) \bullet \eta_{\g}\]
Note that we can also define:
\[\chi:= (f_{a\cdot \g a} * 1_{\g}) \bullet (1_{\g^*} * \rho_a^{\M} * 1_{\g}) \bullet \eta_{\g}\]

Which is a $2$-cell $\id_x \To \N(f(a\cdot \g),f(a) \cdot \g)$ in $\Bc$, i.e. an arrow $f(a)\cdot \g \to f(a \cdot \g)$ in $\pnc(x)$. Now the following computation ensures that the composite $\chi \bullet \theta$ in $\pnc(x)$ yields the identity on $f(a\cdot \g)$: 
\begin{eqnarray*}
\chi \bullet \theta &=& \iota^{\N}_{f(a)\cdot \g f(a\cdot \g) f(a)\cdot \g} \bullet (\chi * \theta)\\
&=& (1_{\g^{*}} * \iota^{\N}_{f(a)f(a\cdot \g)f(a)} * 1_{\g}) \bullet (1_{\g^{*}} * f_{aa\cdot \g} * f_{a\cdot \g a} * 1_{\g})\\ 
& & \bullet (1_{\g^{*}} * \rho_a^{\M} * 1_{\g\g^{*}} * \rho_a^{\M} * 1_{\g}) \bullet (\eta_{\g} * \eta_{\g})\\
&=& (1_{\g^{*}} * (\iota_{f(a)f(a\cdot\g)f(a)}^{\N} \bullet (f_{aa\cdot \g} * f_{a\cdot \g a}) \bullet (\rho_a^{\M} * 1_{\g\g^{*}} * \rho_a^{\M})) * 1_{\g}) \bullet (\eta_{\g} * \eta_{\g})\\
&=& (1_{\g^{*}} * (f_{aa} \bullet \iota_{aa\cdot \g a}^{\M} \bullet (\rho_a^{\M}  * 1_{\g\g^{*}} * \rho_a^{\M})) * 1_{\g}) \bullet (\eta_{\g} * \eta_{\g})\\
&=& (1_{\g^{*}} * (f_{aa} \bullet \iota_{aaa} \bullet (1_{\M(a,a)} * \e_{\g} * 1_{\M(a,a)}) \bullet  (\rho_a^{\M}  * 1_{\g\g^{*}} * \rho_a^{\M})) * 1_{\g}) \bullet (\eta_{\g} * \eta_{\g})\\
&=& (1_{\g^{*}} * (f_{aa} \bullet \iota_{aaa} \bullet (\rho_a^{\M} * \rho_a^{\M}) \bullet \e_{\g}) * 1_{\g}) \bullet (\eta_{\g} * \eta_{\g})\\
&=& (1_{\g^{*}} * (f_{aa} \bullet \rho_a^{\M} \bullet \e_{\g}) * 1_{\g}) \bullet (\eta_{\g} * \eta_{\g})\\
&=& (1_{\g^{*}} * (\rho_{f(a)}^{\N} \bullet \e_{\g}) * 1_{\g}) \bullet (\eta_{\g} * \eta_{\g})\\
&=& (1_{\g^{*}} * \rho_{f(a)}^{\N} * 1_{\g}) \bullet (1_{\g^{*}} * \e_{\g} * 1_{\g}) \bullet (1_{\g^{*}\g} * \eta_{\g}) \bullet \eta_{\g}\\
&=& (1_{\g^{*}} * \rho_{f(a)}^{\N} * 1_{\g}) \bullet (1_{\g^{*}} * ((\e_{\g} * 1_{\g}) \bullet (1_{\g} * \eta_{\g}))) \bullet \eta_{\g}\\
&=& (1_{\g^{*}} * \rho_{f(a)}^{\N} * 1_{\g}) \bullet \eta_{\g}\\
&=& \rho_{f(a)\cdot \g}^{\N}
\end{eqnarray*}

And we recall that in $\pnc(x)$, the identity on $f(a)\cdot \g$ is given by $\rho_{f(a)\cdot \g}^{\N}$. We have here proven that we have $\chi \bullet\theta = \id_{f(a)\cdot \g}$. Now the other sense works in the same way, and we have that $\theta \bullet \chi = \id_{f(a)\cdot \g}$, meaning that $\theta$ is an isomorphism; we actually proved here that we have an isomorphism in $\pnc(x)$: $f(a)\cdot \g \simeq f(a \cdot \g)$, meaning that $P_f$ will actually be a pseudonatural transformation instead of an oplax-transformation. We write $P_f(\g)_a:= \theta$; now we must prove that this defines a natural transformation $P_f(\g)$. Given any $a,b \in A_x$ and any $\theta:  \id_x \To \M(b,a)$, this amounts to prove the commutativity of the following diagram in $\pnc(y)$: 

\[
\begin{tikzcd}[row sep = 50, column sep = 90, labels=description]
f(a\cdot \g)
\arrow[r,"{P_f(\g)_a}"]
\arrow[d,"{f_{b\cdot \g a\cdot \g} \bullet (1_{\g^*} * \theta * 1_{\g}) \bullet \eta_{\g}}"]
& f(a)\cdot \g 
\arrow[d,"{(1_{\g} * (f_{ba} \bullet \theta) * 1_{\g}) \bullet \eta_{\g}}"]\\ 
f(b\cdot \g)
\arrow[r,"{P_f(\g)_b}"]
& f(b)\cdot \g
\end{tikzcd}
\]

We have: 

\begin{eqnarray*}
& &\iota_{f(b)\cdot \g f(b\cdot \g) f(a\cdot \g)}^{\N} \bullet (P_f(\g)_b * (f_{b\cdot \g a\cdot \g} \bullet (1_{\g^*} * \theta * 1_{\g}) \bullet \eta_{\g}))\\ 
&=&(1_{\g^*} * \iota_{f(b)f(b\cdot \g)f(a\cdot \g)}) \bullet (((1_{\g^*} * f_{bb\cdot\g}) \bullet (1_{\g^*} * \rho_b * 1_{\g}) \bullet \eta_{\g}) * \\ 
& &(f_{b\cdot \g a\cdot \g} \bullet (1_{\g^*} * \theta * 1_{\g}) \bullet \eta_{\g}))\\ 
&=& (1_{\g^*} * (\iota_{f(b)f(b\cdot \g)f(a\cdot \g)} \bullet (f_{bb\cdot \g} * f_{b\cdot\g a}))) \bullet (1_{\g^*} * \rho_b * 1_{\g} * 1_{\g^*} * \theta * 1_\g) \bullet\\ 
& & (\eta_\g * \eta_\g)\\
&=& (1_{\g^*} * (f_{ba\cdot \g} \bullet \iota_{bb\cdot\g a})) \bullet (1_{\g^*} * \rho_b * 1_{\g} * 1_{\g^*} * \theta * 1_\g) \bullet \\ 
& & (\eta_\g * \eta_\g)\\
&=& (1_{\g^*} * (f_{ba\cdot \g} \bullet \iota_{bba} \bullet (1_{\M(b,b)} * \e_{\g} * 1_{\M(b,a)}))) \bullet (1_{\g^*} * \rho_b * 1_{\g} * 1_{\g^*} * \theta * 1_\g) \bullet \\ 
& & (\eta_\g * \eta_\g)\\
&=& (1_{\g^*} * (f_{ba\cdot \g} \bullet \iota_{bba})) \bullet (1_{\g^*} * \rho_b * \theta * 1_\g)\ \bullet \eta_\g\\
&=& (1_{\g^*} * f_{ba\cdot \g}) \bullet (1_{\g^*} * \theta * 1_{\g}) \bullet \eta_{\g}\\
\end{eqnarray*}

And for the other sense: 

\begin{eqnarray*}
& & \iota_{f(b)\cdot\g f(a)\cdot\g f(a\cdot\g)} \bullet (((1_{\g^*} * (f_{ba} \bullet \theta ) * 1_{\g})\bullet \eta_{\g}) * P_f(\g)_a)\\ 
&=& (1_{\g^*} * \iota_{f(b)f(a)\cdot\g f(a\cdot\g)}) \bullet (((1_{\g^*} * (f_{ba} \bullet \theta) * 1_{\g})\bullet \eta_{\g}) * ((1_{\g^*} * f_{aa\cdot \g}) \bullet (1_{\g^*} * \rho_a * 1_{\g}) \bullet \eta_{\g}))\\ 
&=& (1_{\g^*} * (\iota_{f(b)f(a)f(a\cdot\g)} \bullet (1_{\N(f(b),f(a))} * \e_{\g} * 1_{\N(f(a),f(a\cdot \g))}))) \bullet \\ 
& & (((1_{\g^*} * (f_{ba} \bullet \theta) * 1_{\g})\bullet \eta_{\g}) * ((1_{\g^*} * f_{aa\cdot \g}) \bullet (1_{\g^*} * \rho_a * 1_{\g}) \bullet \eta_{\g}))\\
&=& (1_{\g^*} * (\iota_{f(b)f(a)f(a\cdot \g)}) \bullet (1_{\g^*} * (f_{ba} \bullet \theta) * (f_{aa\cdot \g} \bullet \rho_a) * 1_{\g}) \bullet \eta_{\g}\\ 
&=& (1_{\g^*} * f_{ba\cdot\g}) \bullet (1_{\g^*} * (\iota_{baa} \bullet (\theta * \rho_a)) * 1_{\g}) \bullet \eta_{\g}\\ 
&=& (1_{\g^*} * f_{ba\cdot \g}) \bullet (1_{\g^*} * \theta * 1_{\g}) \bullet \eta_{\g}\\
\end{eqnarray*}

Hence $P_f(\g)$ is indeed a natural transformation. To conclude, we need to prove the coherence conditions for oplax-natural transformation; the oplax-unity condition is immediate, and the oplax-naturality conditions takes the form, for any maps $\g_1: x \to y$, $\g_2: y \to z$ in $\Bc$: 
\[
(1_{\pnc(\g_1)} * (P_f)(\g_2)) \bullet ((P_f)(\g_1) * 1_{\pma(\g_2)}) = (P_f)(\g_2\g_1)
\]

We do not give here a proof of this oplax-naturality.
\end{proof}

\begin{rk}
The above proof actually showed that $P_f$ is a pseudonatural transformation, which is stronger than just being an oplax-transformation. Of course $\si$ can also be defined by taking pseudonatural transformations and we can thus obtain a pair of functors $\si^{\pseudo}: [\Map(\Bc)^{\coop},\Cat]^{\pseudo} \leftrightarrows \Catk(\Bc): P^{\pseudo}$; however the adjunction we want to prove will not hold in this more restrained case, as transforming a $\Bc$-functor $\si F \to (\M,A)$ will not yield a pseudonatural transformation $F \to \pma$ in general but only an oplax-natural transformation; see Remark~\ref{oplaxtrfAfgg} for more details.
\end{rk}

To summarize the results of our above construction of $P$, we have: 

\begin{enumerate}
\itb For any $\Bc$-category $(\M,A)$, $\pma: \Map(\Bc)^{\coop} \to \Cat$ is a pseudofunctor with: 
\begin{enumerate}
\itb $\Ob(\pma(x)) = A_x$
\itb An arrow $a \to b$ in $\pma(x)$ is a $2$-cell $\theta: \id_x \To \M(b,a)$
\itb For any map $\g: x \to y$ in $\Bc$, $\pma(\g)$ is a functor $\pma(y) \to \pma(x)$ with:
\begin{enumerate}
\itb For any $a \in A_y$, $\pma(\g)(a) = a \cdot \g$
\itb For any $\theta: \id \to \M(b,a)$, $\pma(\g)(\theta) = (1_{\g^*} * \theta * 1_{\g}) \bullet \eta_{\g}$
\end{enumerate}
\itb For any $\aa: \g_1 \To \g_2$ in $\Map(\Bc)$, $\pma(\aa)$ is a natural transformation $\pma(\g_2) \To \pma(\g_1)$ with:
\begin{enumerate}
\itb $\pma(\aa)_a = (1_{\g_1^*} * \rho_a * \aa) \bullet \eta_{\g_1}$
\end{enumerate}
\end{enumerate}
\itb For any $\Bc$-functor $f: (\M,A) \to (\N,C)$, $P_f$ is an oplax-natural transformation $\pma \to \pnc$ with:
\begin{enumerate}
\itb For any $x \in \Ob(\Bc)$, $(P_f)_x: \pma(x) \to \pnc(x)$ is a functor with:
\begin{enumerate}
\itb For any $a \in A_x$, $(P_f)_x(a) = f(a)$
\itb For any $\theta: \id_x \To \M(b,a)$, $(P_f)_x(\theta) = f_{ba} \bullet \theta$
\end{enumerate}
\itb For any map $\g: x \to y$ in $\Bc$, $(P_f)(\g): (P_f)_x \cdot \pma(\g) \To \pnc(\g) \cdot (P_f)_y$ is a natural transformation with:
\begin{enumerate}
\itb For any $a \in A_y$, $(P_f)(\g)(a) = (1_{\g^*} * f_{aa\cdot \g}) \bullet (1_{\g^*} * \rho_a * 1_{\g}) \bullet \eta_{\g}$.
\end{enumerate}
\end{enumerate}
\end{enumerate}

\begin{prop}
The above propositions define a functor: 
\[P: \Catk(\Bc) \to [\Map(\Bc)^{\coop},\Cat]\]
In the symmetric case, we obtain a functor:
\[P_{\s}: \Catsk(\Bc) \to [\Map_{\s}(\Bc)^{\coop},\Cat]\]
\end{prop}

\begin{proof}
Let $(\M,A)$ be a complete $\Bc$-category, denote by $i: (\M,A) \to (\M,A)$ the identity $\Bc$-functor on $(\M,A)$ which sends $a$ to $a$ and such that $i_{ab} = 1_{\M(a,b)}$. Then for any $a \in A_x$ with $x \in \Ob(\Bc)$, $(P_i)_x(a) = i(a) = a$; now for any $\theta: \id_x \To \M(a,b)$ with $a,b \in A_x$, $(P_i)_x(\theta) = i_{ab} \bullet \theta = \theta$. This shows that the functor $(P_f)_x$ is the identity functor on $\pma(x)$. Now for any map $\g: x \to y$ in $\Bc$ and any $a \in A_x$, $(P_i)(\g)_a = (1_{\g^*} * i_{aa\cdot \g}) \bullet (1_{\g^*} * \rho_a * 1_{\g}) \bullet \eta_{\g} = (1_{\g^*} * \rho_a * 1_{\g}) \bullet \eta_{\g} = \rho_{a\cdot \g}$. This proves that $P_i$ is the identity oplax natural transformation on $\pma$.\\ 

Let $f: (\N,C) \to (\Pb,D)$ and $g: (\M,A) \to (\N,C)$ be two $\Bc$-functors between $\Bc$-categories. We want to prove $P_f \cdot P_g = P_{fg}$. For any $x \in \Ob(\Bc)$ and any $a \in A_x$, $(P_f \cdot P_g)_x(a) = (P_f)_x \cdot (P_g)_x(a) = fg(a)$. For any $\theta: \id_x \To \M(b,a)$ with $b \in A_x$, $(P_f \cdot P_g)_x(\theta) = (P_f)_x \cdot (P_g)_x(\theta) = f_{ba} \bullet g_{ba} \bullet \theta = fg_{ba} \bullet \theta$. This proves that for any $x \in \Ob(\Bc)$, $(P_f \cdot P_g)_x = (P_{fg})_x$. Now consider $a \in A_x$, $\gamma: x \to y$ a map in $\Bc$, then $(P_f \cdot P_g)(\g) = (P_f(\g) * 1_{(P_g)_y}) \bullet (1_{(P_f)_x} * P_g(\g))$; we can compute as follows:
\begin{eqnarray*}
& &(P_f \cdot P_g)(\g)_a\\
&=& (P_f(\g) * 1_{(P_g)_y})_a \bullet (1_{(P_f)_x} * P_g(\g))_a\\ 
&=& P_f(\g)_{g(a)} \bullet (P_f)_x(P_g(\g)_a) \ \text{in $P_{\Pb,D}(x)$}\\
&=& \iota_{fg(a)\cdot\g f(g(a)\cdot \g) fg(a\cdot \g)} \bullet (P_f(\g)_{g(a)} * (P_f)_x(P_g(\g)_a))\\ 
&=& \iota \bullet (1_{\g^*} * (f_{g(a)g(a)\cdot \g} \bullet (g_{aa} * 1_{\g})) * (f_{g(a)\cdot\g g(a\cdot \g)} \bullet (1_{\g^*} * g_{aa\cdot \g}))) \bullet (\rho_{a\cdot\g} * \rho_{a\cdot\g})\\ 
&=& (1_{\g^*} * (\iota_{fg(a)f(g(a)\cdot\g)fg(a\cdot\g)} \bullet (f_{g(a)g(a)\cdot\g} * f_{g(a)\cdot\g g(a\cdot\g)}) \bullet \\ 
& & (g_{aa} * 1_{\g\g^*} * g_{a a\cdot\g}))) \bullet (\rho_{a\cdot\g} * \rho_{a\cdot\g})\\ 
&=& (1_{\g^*} * (f_{g(a)g(a\cdot\g)} \bullet \iota_{g(a)g(a)\cdot\g g(a\cdot \g)} \bullet (g_{aa} * 1_{\g\g^*} * g_{a a\cdot\g}) \bullet\\
& & (\rho_a * 1_{\g\g^*} * \rho_a * 1_{\g}))) \bullet (\eta_{\g} * \eta_{\g})\\ 
&=& (1_{\g^*} * (f_{g(a)g(a\cdot\g)} \bullet \iota_{g(a)g(a)g(a\cdot \g)} \bullet (\rho_{g(a)}* \e_{\g} * (g_{a a\cdot\g} \bullet\\ 
& & (\rho_a * 1_{\g}))))) \bullet (\eta_{\g} * \eta_{\g})\\ 
&=& (1_{\g^*} * (f_{g(a)g(a\cdot\g)} \bullet g_{a a\cdot \g} \bullet (\rho_a * 1_{\g}))) \bullet \eta_{\g}\\ 
&=& (1_{\g^*} * (fg)_{aa\cdot\g}) \bullet \rho_{a\cdot\g}\\ 
&=& P_{fg}(\g)_a
\end{eqnarray*}

Everything carries well into the symmetric case.
\end{proof}

\subsection{The adjunction}

\begin{thm}
There is an adjunction: \[\C\si \dashv P\]
\end{thm}

\begin{proof}
We will give the adjunction $\C\si \dashv P$ as the data, for any $F: \Map(\Bc)^{\op} \to \Cat$ and any complete $\Bc$-category $(\M,A)$, of a bijection of sets: 

\[\Cat_{\kappa}(\C\si F,(\M,A)) \simeq [\Map(\Bc)^{\op},\Cat](F,\pma)\]
which is pseudonatural in $F$ and in $(\M,A)$. We therefore start by defining two functions: 

\[\Af: \Hom_{\Cat(\Bc)}(\si F,(\M,A)) \to \Hom_{[\Map^{\op}(\Bc),\Cat]}(F,\pma)\]
\[\Gf: \Hom_{[\Map^{\op}(\Bc),\Cat]}(F,\pma) \to \Hom_{\Cat(\Bc)}(\si F,(\M,A))\]\\

Note that here we used the fact that the Cauchy functor is left adjoint to the inclusion of complete $\Bc$-categories in general $\Bc$-categories, which yields: \[\Hom_{\Cat_{\kappa}(\Bc)}(\C \si F,(\M,A)) \simeq \Hom_{\Cat(\Bc)}(\si F,(\M,A))\] 

Starting with $\Af$, let $g: \si F \to (\M,A)$ be a $\Bc$-functor; we want to define an oplax-natural transformation $\Af(g): F \to \pma$. This means giving:
\begin{enumerate}
\itb For any object $x$ in $\Bc$, a functor (morphism in $\Cat$) $\Af(g)_x: F(x) \to \pma(x)$.
\itb For any map $\g: x \to y$ in $\Bc$, a natural transformation ($2$-cell in $\Cat$) $\Af(g)(\g): \Af(g)_x \cdot F(\g) \To \pma(\g) \cdot \Af(g)_y$.
\itb A set of coherence conditions satisfied by the above data.
\end{enumerate}

Because $g$ is a $\Bc$-functor, it is type-preserving, hence restricting it at $(\si F)_x = \Ob(F(x))$ goes into $A_x = \Ob(\pma(x))$. We define $\Af(g)_x$ on objects as $g_{|F(x)}$. Now consider an arrow $u: a \to b$ in $F(x)$; we want to obtain an arrow $\Af(g)_x(u): g(a) \to g(b)$ in $\pma(x)$, i.e. a $2$-cell $\id_{x} \To \M(g(b),g(a))$. Seeing that the $\Bc$-functor condition of $g$ gives us a $2$-cell $g_{ba}: \N(b,a) \To \M(g(b),g(a))$ (where $\N$ denotes the structural matrix of $\si F$), we are inclined to search for a $2$-cell $\id_x \To \N(b,a)$. But $\N(b,a)$ is defined as the colimits of these $q: x \to x$ in $\Bc$ such that there exists a $u: a \to F(q)(b)$ in $F(x)$. For $q = \id_x$, we have $u: a \to F(\id_x)(b) = b$, yielding an ``injection'' $2$-cell: \[e_u: \id_x \To \N(b,a)\]

Thus we define $\Af(g)_x(u):= g_{ba} \bullet e_u: \id_x \To \M(g(b),g(a))$. Now we must prove that $\Af(g)_x$ is indeed a functor. For that: 

\begin{enumerate}
\itb For any $a \in F(x)$, $\Af(g)_x(\id_a) = g_{aa} \bullet e_{\id_a}$. But by definition of the $\Bc$-category $\si F$, $e_{\id_a} = \rho_a^{\N}$, hence $\Af(g)_x(\id_a) = g_{aa} \bullet \rho^{\N}_{\id_a} = \rho_{g(a)}$ using the property of the $\Bc$-functor $g$; and $\rho_{g(a)}^{\M}$ is the identity on $g(a)$ in $\pma(x)$. Hence $\Af(g)_x(\id_a) = \id_{\Af(g)_x(a)}$.
\itb Consider two arrows: $u: a \to b$, $v: b \to c$ in $F(x)$. Then in $\pma(x)$, we have:
\begin{align*}
\Afgx(v) \cdot \Afgx(u) &= \iota^{\M}_{\Afgx(c)\Afgx(b)\Afgx(a)} \bullet (\Afgx(v) * \Afgx(u))\\ 
&= \iota^{\M}_{g(c)g(b)g(a)} \bullet ((g_{cb} \bullet e_v) * (g_{ba} \bullet e_u))\\ 
&= \iota^{\M}_{g(c)g(b)g(a)} \bullet (g_{cb} * g_{ba}) \bullet (e_v * e_u)\\ 
&= g_{ca} \bullet \iota^{\N}_{cba} \bullet (e_v * e_u)
\end{align*}

And to conclude we can simply recall that the definition of $\iota^{\N}_{cba}$ is so that $\iota^{N}_{cba} \bullet (e_v * e_u) = e_{vu}$. \\ 
\end{enumerate}

Let $\g: x \to y$ be a map in $\Bc$. We want to construct an invertible natural transformation $\Af(g)(\g): \Af(g)_x \cdot F(\g) \To \pma(\g) \cdot \Af(g)_y$. These functors are going from $F(y)$ to $\pma(x)$; hence that means that we ought to give, for any $a \in F(y)$, an arrow $\Afgg_a: (\Af(g)_x \cdot F(\g))(a) \to (\pma(\g) \cdot \Af(g)_y)(a)$ in $\pma(x)$, which must be natural in $a$. First let us explicit what that means: 
\begin{enumerate}
\itb $(\Afgx \cdot F(\g))(a) = g(F(\g)(a))$
\itb $(\pma(\g) \cdot \Afgy)(a) = g(a)\cdot \g$
\end{enumerate}

Because we search for an arrow in $\pma(x)$, that means we want a $2$-cell $\id_x \To \M(g(F(\g)(a)),g(a)\cdot \g)$. Again, we will most certainly use $g_{F(\g)(a)a}$, hence we must try to produce a $2$-cell $\id_x \To \g^*\N(a,F(\g)(a))$. Because there is an arrow $\id: F(\g)(a) \to F(\g)(a)$ in $F(x)$, we have a canonical injection $e_{\id}: \g \To \N(a,F(\g)(a))$ in $\Bc$. Using the unity of the adjunction which defines $\g$, we get $\id_x \To \g^*\g \To \g^* \N(a,F(\g)(a)) \To \g^* \M(g(a),g(F(\g)(a)))$. We therefore can define: \[\Afgg_a:= (1_{\g^*} * (g_{aF(\g)(a)} \bullet e_{\id})) \bullet \eta_{\g}\]

We do not prove the coherence conditions which ensure that $\Af(g)$ is an oplax-natural transformation. We can however give a quick word about the fact that we restrict to oplax-natural transformations: see Remark~\ref{oplaxtrfAfgg} for that matter. \\

Now that we have given the function $\Af$, we turn to $\Gf$. Let there be an oplax natural transformation $\alpha: F \to \pma$, i.e. the data of:

\begin{enumerate}
\item For all $x \in \Ob(\Bc)$, a functor $\alpha_x: F(x) \to \pma(x)$.
\item For any map $\g: x \to y$ in $\Bc$, a natural transformation $\alpha(\g): \alpha_x \cdot F(\g) \To \pma(\g) \cdot \alpha_y$. This is a natural transformation between functors going from $F(y)$ to $\pma(x)$, and which respectively send an $a \in F(y)$ to $\alpha_x(F(\g)(a))$ and to $\alpha_y(a)\cdot \g$.
\item Several coherence conditions.
\end{enumerate}

From that natural transformation we want to define a $\Bc$-functor $\gfa: (\N,\si F) \to (\M,A)$. Its value on elements of $\si F$ is easy enough to give:  \[\gfa(a) = \alpha_{ta}(a)\]

But now we are interested in defining the constitutive family of $2$-cells for a $\Bc$-functor, i.e. for any $a,b \in \si F$, we need a $2$-cell: \[\N(a,b) \To \M(\aa_{ta}(a),\aa_{tb}(b))\]

Because of the definition of $\N(a,b)$ as a colimit, giving $\gfa_{ab}$ is the same thing as giving, for any map $q: tb \to ta$ in $\Bc$ and any arrow $u: b \to F(q)(a)$ in $F(tb)$, a $2$-cell: \[q \To \M(\aa_{ta}(a),\aa_{tb}(b))\]

In a compatible way. For such a pair $(q,u)$, we have: \[\alpha(q)_a: \id_{tb} \To q^*\M(\aa_{ta}(a),\aa_{tb}(F(q)(a)))\]

Hence using the counit of the adjunction of $q$, we easily have a $2$-cell $q \To \M(\aa_{ta}(a),\aa_{tb}(F(q)(a)))$. Now using $u$, which is an arrow in $F(tb)$, we can define $\alpha_{tb}(u)$, which is an arrow $\aa_{tb}(b) \to \aa_{tb}(F(q)(a))$ in $\pma(tb)$, i.e. a $2$-cell in $\Bc$: \[\alpha_{tb}(u): \id_{tb} \To \M(\aa_{tb}(F(q)(a)),\aa_{tb}(b))\]

And using $\iota$ and an horizontal composite, we find a $2$-cell $\mathfrak{a}_q: q \To \M(\aa_{ta}(a),\aa_{tb}(b))$ in $\Bc$, given by:

\[
\mathfrak{a}_q:= (\e_q * \iota_{\aa_{ta}(a)\aa_{tb}(F(q)(a))\aa_{tb}(b)}^{\M}) \bullet (1_{q} * \aa(q)_a * \aa_{tb}(u))
\]

We leave to the reader the proof that this indeed commutes well with the base of the colimit (up to some additional coherence condition to be added if needed), thus defining a $2$-cell $\N(a,b) \To \M(\aa_{ta}(a),\aa_{tb}(b))$, defined by the commutativity of the following diagram:

\[
\begin{tikzcd}[sep = huge, labels=description]
& \M(\Gc(\alpha)(a),\Gc(\aa)(b))\\ 
q 
\arrow[r,"e_q"]
\arrow[ur,"{\mathfrak{a}_q}", bend left = 30, end anchor=real west]
& \N(a,b) 
\arrow[u,"{\Gc(\aa)_{ab}}"]
\end{tikzcd}
\]

Similarly, we do not write the proof that $\Gc(\aa)$ is a $\Bc$-functor; the proof is similar to several other ones in this document.\\

Now we will prove that we have an adjunction, that means that we will show that $\Af$ and $\Gf$ are inverses to one another as functions. We therefore need to prove: 

\begin{enumerate}
\itb For any oplax natural transformation $\alpha: F \to \pma$, $\Af(\Gf(\alpha)) = \aa$.
\itb For any $\Bc$-functor $g: \si F \to \pma$, $\Gf(\Af(g)) = g$. 
\end{enumerate}

Starting with $\Af\Gf$, let $\alpha$ be an oplax natural transformation $F \to \pma$. To prove $\Af(\Gf(\aa)) = \aa$, we need to prove: 

\begin{enumerate}
\item For all $x \in \Ob(\Bc)$, $\Af(\Gf(\aa))_x: F(x) \to \pma(x)$ is the functor $\alpha_x$.
\begin{enumerate}
\item For all $a \in \Ob(F(x))$, $\Af(\Gf(\aa))_x(a) = \aa_x(a)$.
\item For any $u: a \to b$ in $F(x)$, $\Af(\Gf(\aa))_x(u) = \aa_x(u)$.
\end{enumerate}
\item For any map $\g: x \to y$ in $\Bc$, $\Af(\Gf(\aa))(\g)$ is the natural transformation $\aa(\g)$, going from $\alpha_x \cdot F(\g)$ to $\pma(\g) \cdot \aa_y$.
\end{enumerate}

Starting with $1$a: for all $x \in \Ob(\Bc)$, $\Af(\Gf(\aa))_x = \Gf(\aa)(a) = \aa_x(a)$ by definition. For $1$b, let $u:a \to b$ be an arrow in $F(x)$. Then $\Af(\Gf(\aa))_x(u) = \Gf(\aa)_{ba} \cdot e_u$, where $e_u: \id_x \To \N(a,b)$ is associated to the pair $(\id_x,u)$ in the colimit diagram defining $\N(a,b)$. We want to prove $\Gf(\aa)_{ba} \cdot e_u = \aa_x(u)$. Let us draw the following diagram, which is commutative by definition of $\Gf(\aa)_{ba}$:

\[
\begin{tikzcd}[sep = huge]
& \M(\aa_x(a),\aa_x(b)) \\ 
\id_x
\arrow[ru,"(\e_{\id} * \iota) \bullet (1 * \aa(\id)_a * \aa_x(u))"description, bend left=25]
\arrow[r,"e_u"description]
& \N(a,b)
\arrow[u,"\Gf(\aa)_{ba}"description]
\end{tikzcd}
\]

Now a quick computation yields the following, because $\e_{\id} = 1_{\id}$: 

\begin{align*}
(\e_{id_x} * \iota) \bullet (1_{\id_x} * \aa(\id_x)_a * \alpha_x(u))&= \iota_{\aa_x(a)\aa_x(a)\aa_x(b)} \bullet (\aa(\id_x)_a * \aa_x(u))\\ 
&= \iota_{\aa_x(a)\aa_x(a)\aa_x(b)} \bullet (\id_{\aa_x(a)}^{\pma(x)} * \aa_x(u))\\ 
&= \iota_{\aa_x(a)\aa_x(a)\aa_x(b)} \bullet (\rho_{\aa_x(a)} * \aa_x(u))\\ 
&= \aa_x(u)
\end{align*}

We are able to conclude because of two things: first $\aa(\id_x)_a = \id_{\aa_x(a)}$ (this is one of the coherence condition for an oplax natural transformation) and then this identity in $\pma(x)$ identifies as $\rho_{\aa_x(a)}$. Therefore we have proven the point $1$b, and we can embark onto point $2$. Let $\g: x \to y$ be a map in $\Bc$. We have, for any $a \in F(y)$: 

\[
\Af(\Gf(\aa))(\g)_a = (1_{\g^*} * (\Gf(\aa)_{aF(\g)(a)} \bullet e_{\g,\id_{F(\g)(a)}} )) \bullet \eta_{\g}
\] \\ 

Again, we can use the definition of $\Gf(\aa)_{aF(\g)(a)}$ as a colimit to express the composite $\Gf(\aa)_{aF(\g)(a)} \bullet e_{\g,\id}$ as: \[(\e_{\g} * \iota) \bullet (1_{\g} * \alpha(\g)_a * \alpha_{y}(\id_{F(\g)(a)}))\] \\ 

Now we know that $\alpha_{y}(\id_{F(\g)(a)}) = \id_{\aa_x(F(\g)(a))}$ in $\pma(x)$, that is $\rho_{\aa_x(F(\g)(a))}$, we can represent the composite $\Gf(\aa)_{a F(\g)(a)} \bullet e_{\g,\id}$ as follows: 

\[
\begin{tikzcd}[row sep = 30, column sep = 48]
y
&x
\arrow[l,"\g" {name={11}}]
&&&x
\arrow[lll,"\id" {name={12}}]
&&x
\arrow[ll,"\id" {name={13}}]\\ 
y
&x
\arrow[l,"\g" {name={21}}]
&y
\arrow[l,"\g^*" {name={22}}]
\arrow[ll,"" {name={212}}, phantom]
&&x
\arrow[ll,"{\M(\aa_y(a),\aa_x(F(\g)(a)))}" {name={23}}]
\arrow[lll,"" {name={223}}, phantom]
&&x
\arrow[llll,"" {name={234}}, phantom]
\arrow[ll,"{\M(\aa_x(F(\g)(a)),\aa_x(F(\g)(a)))}" {name={24}}] \\ 
y
&& y
\arrow[ll,"\id" {name={31}}]
&&&&x
\arrow[llll,"{\M(\aa_y(a),\aa_x(F(\g)(a)))}" {name = {32}}]
\arrow[Rightarrow, from=11, to=21, "1", shorten >=3, shorten <= 5]
\arrow[Rightarrow, from=12, to=223, "{\aa(\g)_a}", shorten >=3, shorten <= 5]
\arrow[Rightarrow, from=13, to=24, "{\rho^{\M}_{\aa_x(F(\g)(a))}}", shorten >=3, shorten <= 5]
\arrow[Rightarrow, from=212, to=31, "{\e_{\g}}", shorten >=3, shorten <= 5]
\arrow[Rightarrow, from=234, to=32, "{\iota^{\M}_{\aa_y(a)\aa_x(F(\g)(a))\aa_x(F(\g)(a))}}", shorten >=3, shorten <= 5]
\end{tikzcd}
\]

\begin{align*}
\Af(\Gf(\aa))(\g)_a &= (1_{\g^*} * ((\e_{\g} * \iota) \bullet (1_{\g} * \aa(\g)_a * \rho_{\aa_x(F(\g)(a))}))) \bullet \eta_{\g}\\ 
&= (1_{\g^*} * (\e_{\g} \bullet (1_{\g} * \alpha(\g)_a))) \bullet \eta_{\g}\\
&= (1_{\g^*} * \e_{\g} * 1_{\M(\aa_y(a),\aa_x(F(\g)(a)))}) \bullet (\eta_{\g} * \aa(\g)_a)\\ 
&= (((1_{\g^*} * \e_{\g}) \bullet (\eta_{\g} * 1_{\g})) * 1_{\M(\aa_y(a),\aa_x(F(\g)(a)))}) \bullet \alpha(\g)_a\\ 
&= \alpha(\g)_a
\end{align*}

This shows point 2, and therefore $\Af\Gf = \id_{\Hom_{[\Map(\Bc)^{\op},\Cat]^{\mathrm{oplax}}}(F,\pma)}$. Now for the other sense, let $g: \si P \to (\M,A)$ be a $\Bc$-functor. We want to prove $\Gf\Af(g) = g$. There are two things to prove: 

\begin{enumerate}
\item For any $a \in \si P$, $\Gf(\Af(g))(a) = g(a)$.
\item For any $a,b \in \si P$, $\Gf(\Af(g))_{ab} = g_{ab}$.
\end{enumerate}

The first condition is trivial using the definitions: $\Gf(\Af(g))(a) = \Af(g)_{ta}(a) = g(a)$. Now for the point number 2, recall that $\Gf(\Af(g))_{ab}$ is defined as the following colimit, for all $\g: tb \to ta$, and $u: b \to F(\g)(a)$: 

\[
\begin{tikzcd}[row sep = 30, column sep = 80]
& \M(\aa_x(a),\aa_x(b)) \\ 
\g
\arrow[ru,"(\e_{\g} * \iota) \bullet (1 * \aa(\g)_a * \aa_x(u))"description, bend left=25]
\arrow[r,"e_{u,\g}"description]
& \N(a,b)
\arrow[u,"\Gf(\Af(g))_{ab}"description]
\end{tikzcd}
\]

To prove that $\Gf(\Af(g))_{ab}$ is $g_{ab}$, we therefore only need to prove that: 
\[
g_{ab} \bullet e_{u,\g} = (\e_{\g} * \iota) \bullet (1 * \Af(g)(\g)_a * \Af(g)_{tb}(u))
\]

Using the expressions of $\Af(g)$, we can develop this equation; we have: 

\begin{enumerate}
\itb $\Af(g)(\g)_a = (1_{\g^*} * (g_{aF(\g)(a)} \bullet e_{\g,\id})) \bullet \eta_{\g}$
\itb $\Af(g)_{tb}(u) = g_{ab} \bullet e_{u,\id_{tb}}$
\end{enumerate}

Therefore $g_{ab} \bullet e_{u,\g}$ corresponds to the following diagram: 

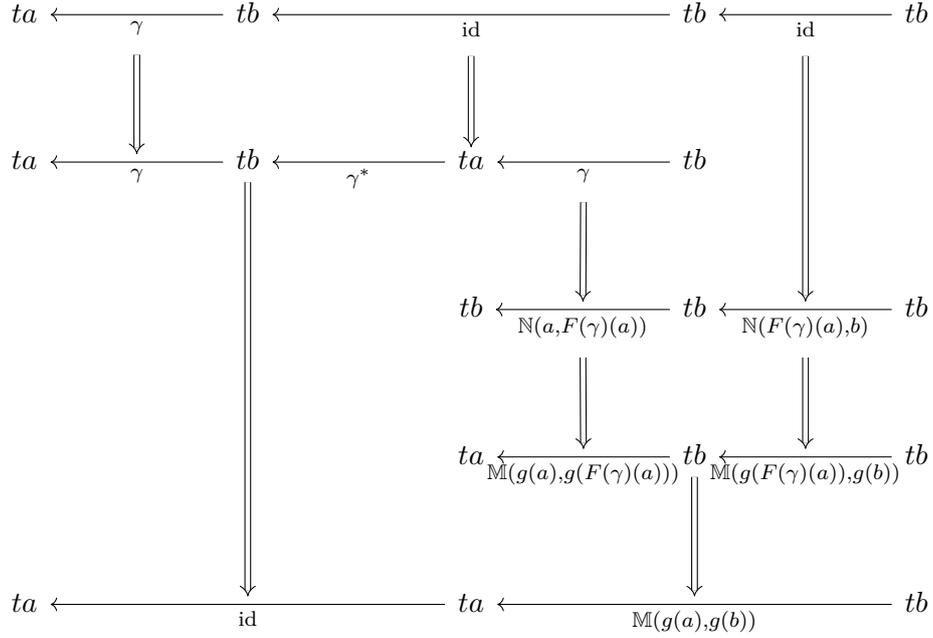
\[
\begin{tikzcd}[row sep = 40, column sep = 65]
ta 
& tb
\arrow[l,"\g" {name={11}}]
&& tb
\arrow[ll,"\id" {name={12}}]
&tb
\arrow[l,"\id" {name={13}}]\\ 
ta 
& tb
\arrow[l,"\g" {name={20}}]
& ta
\arrow[l,"\g^*"]
\arrow[ll,"" {name={21}}, phantom]
& tb
\arrow[l,"\g" {name={23}}]
\arrow[ll,"" {name={22}}, phantom] \\ 
&& tb 
& tb
\arrow[l,"{\N(a,F(\g)(a))}" {name={42}}]
& tb 
\arrow[l,"{\N(F(\g)(a),b)}" {name={43}}]\\ 
&& ta 
& tb 
\arrow[l,"{\M(g(a),g(F(\g)(a)))}" {name={52}}]
& tb 
\arrow[l, "{\M(g(F(\g)(a)),g(b))}" {name={53}}]
\arrow[ll, "" {name={5}}, phantom]\\ 
ta
&& ta 
\arrow[ll,"\id" {name={51}}]
&& tb
\arrow[ll,"{\M(g(a),g(b))}" {name={6}}]
\arrow[Rightarrow, from=11, to=20, "", shorten >=3, shorten <= 5]
\arrow[Rightarrow, from=12, to=22, "", shorten >=3, shorten <= 5]
\arrow[Rightarrow, from=21, to=51, "", shorten >=3, shorten <= 5]
\arrow[Rightarrow, from=13, to=43, "", shorten >=3, shorten <= 5]
\arrow[Rightarrow, from=23, to=42, "", shorten >=3, shorten <= 5]
\arrow[Rightarrow, from=42, to=52, "", shorten >=3, shorten <= 5]
\arrow[Rightarrow, from=43, to=53, "", shorten >=3, shorten <= 5]
\arrow[Rightarrow, from=5, to=6, "", shorten >=3, shorten <= 5]
\end{tikzcd}
\]

Using the interchange law, this diagram simplifies into: 

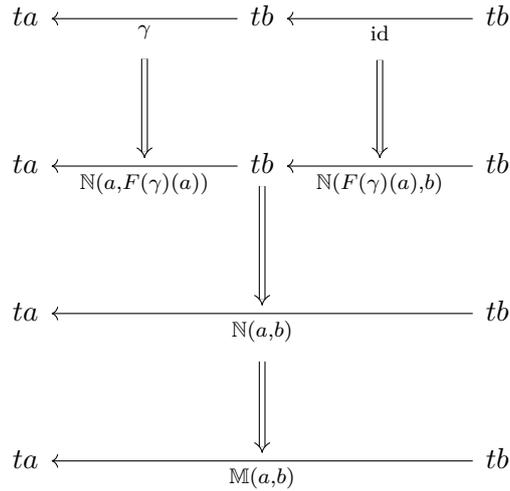
\[
\begin{tikzcd}[row sep = 40, column sep = 70]
ta 
& tb 
\arrow[l,"\g" {name={11}}]
&tb 
\arrow[l,"\id" {name={12}}] \\ 
ta 
& tb 
\arrow[l,"{\N(a,F(\g)(a))}" {name={21}}]
& tb 
\arrow[l,"{\N(F(\g)(a),b)}" {name={22}}] 
\arrow[ll,"" {name={2}}, phantom]\\ 
ta 
&& tb 
\arrow[ll,"{\N(a,b)}" {name={3}}] \\ 
ta 
&& tb
\arrow[ll,"{\M(a,b)}" {name={4}}]
\arrow[Rightarrow, from=12, to=22, "", shorten >=3, shorten <= 5]
\arrow[Rightarrow, from=11, to=21, "", shorten >=3, shorten <= 5]
\arrow[Rightarrow, from=2, to=3, "", shorten >=3, shorten <= 5]
\arrow[Rightarrow, from=3, to=4, "", shorten >=3, shorten <= 5]
\end{tikzcd}
\]

All that is left to see is $\iota_{aF(\g)(a)b}^{\N} \bullet (e_{\id,\g} * e_{u,\id}) = e_{u,\g}$, which is true by definition of $\iota^{\N}$.
\end{proof}

\begin{rk}\label{oplaxtrfAfgg}
Recall that the only limiting factor which prevented us to use pseudonatural transformations was the definition of $\Af(g)$ as an oplax-natural transformation; we gave $\Afgg$ as a non-invertible natural transformation. The reason is mainly that $\g^*$ is not a map in general; more precisely, defining the inverse of $\Afgg_a$ would lead us to give a $\theta: \id_x \To \N(F(\g)(a),a)\g$, and the natural course of action would be to use a canonical injection $\g^* \To \N(F(\g)(a),a)$, which does not exist because $\g^*$ is not a map, even though there does exist a morphism $a \to F(\g\g^*)(a)$ in $F(y)$ (recall that $F$ also reverses $2$-cells). This also means that in the case where we have some kind of denseness condition of maps among $1$-cells of $\Bc$ (something like ``any $1$-cell can be expressed as a colimit of maps''), we can quite probably extend the adjunction to the pseudonatural case; this denseness condition is however not always satisfied: for example in the quantaloid $\R(X)$ of open subsets of a topological space, there is only a map $U \to V$ if one open subset is included in the other, and if not there can still be $1$-cells $U \to V$ if the intersection is not empty. On the other hand, the condition is satisfied in the case $\Bc = \Bc_{\Set}$, where the singleton $\{*\}$ is a map, and any set is a coproduct of singletons. See corollary~\ref{PtFixesDenseness} as well as the discussion before proposition~\ref{FixPtAdjMonCatCat} for more details about this ``denseness'' condition.
\end{rk}

\begin{rk}[Symmetric case]\label{sympasbien}
The above constructions do not all carry too easily to the symmetric case; notably, although the functor $P$ does restrict well to $\Catsk(\Bc)$ and lands in $[\Map_{\s}(\Bc)^{\coop},\Cat]$ in this case, the functor $\si$ does not work as well. It is possible to define $\si F$ for a functor $F: \Map_{\s}(\Bc)^{\coop} \to \Cat$, but then the resulting $\Bc$-category $\si F$ is not symmetric in general; we need to use a symmetrization functor $\Sc$. This symmetrization functor has been described in the quantaloid case in section 3.5 of~\cite{HeymansThesis} as $\M_{\s}(x,y) = \M(x,y) \wedge \M(y,x)^{\circ}$; although it may be possible to extend it to the general case, notably by considering the product $\Mt(x,y) = \M(x,y) \times \M(y,x)^{\circ}$ we will not delve into that subject deeper. Note that even though $(\Mt,A)$ as defined above is a symmetric $\Bc$-category, it cannot be directly defined as the symmetrization of $(\M,A)$ as we would not have an adjunction between the symmetrization functor and $i$ in this case; more precisely, the monad on $\Cat_{\s}(\Bc)$ which sends any symmetric $\Bc$-category $(\M,A)$ to $(\Mt,A)$ is not idempotent, but it is probable that it is only lax-idempotent. Defining the symmetrization as the idempotent monad generated by this lax-idempotent monad would probably yield the correct notion of symmetrization.
\end{rk}

\begin{rk}[Fullness and Faithfulness of $P$]
In general the simple fact that $P_f$ is a pseudonatural transformation instead of an oplax-natural transformation ensures that $P$ is not full in general. Now in general $P$ is not faithful either; let $(\M,A), (\N,C)$ be two $\Bc$-categories, and consider $f,g: (\M,A) \to (\N,C)$ two $\Bc$-functors such that $P_f=P_g$. Then it is easily shown that for any $a \in A$, $f(a) = (P_f)_{ta}(a) = (P_g)_{ta}(a) = g(a)$; what is left to do is to see that $f_{ab} = g_{ab}$ for any $a,b \in A$. In the case of quantaloids, this is always true, yielding that $P$ is always faithful in the quantaloid case, but we will see in~\ref{UndrCatFunc} that $P_-(*)$ in the monoidal case is exactly the ``underlying category'' functor; for any $\theta: I \to \M(a,b)$ we have $f_{ab} \bullet \theta = g_{ab} \bullet \theta$. Therefore the faithfulness of $P$ will be decided in this case by that of $\V(I,-): \V^{\op} \to \Set$; for example when $\V = \Cat$, $P$ is not faithful, but it is for $\V=  \Set, \Ab, R\text{-}\mathbf{mod}, \mathbf{CGTop}$ (the cartesian monoidal category of compactly generated topological spaces). See section 1.3 of~\cite{Kelly} for more details.

We have discussed in Remark~\ref{oplaxtrfAfgg} the fact that in certain particular cases the adjunction can be extended to the pseudonatural case, and not just the oplax-natural one. Notably, proposition~\ref{FixPtAdjMonCatCat} gives a strong sufficient on a monoidal category so that a particular ``denseness'' condition of maps in $\Bv$ applies, ensuring that we can restrict the adjunction to functors $\Bv^{\co} \to \Cat$ and pseudonatural transformations between them; examples include at least $\Set$ and $\Ab$; the condition is that every object of $\V$ is a colimit of copies of $I$. In this case, $P$ becomes full, as we have for any $\Bv$-category $(\N,C)$ that $(\N,C) = \si \pnc$, and given $\alpha: \pma\to \pnc$ we can use $\Gf$ to get the corresponding $\Bv$-functor. In the quantaloid case however, it is even simpler because the $2$-structure is trivial, and $P$ is always full. 

This imply that we have some cases where $P$ is fully faithful: quantaloids on one hand, and the particular cases of monoidal categories expressed above (in general, most concrete monoidal categories should work too). In all these cases, the category $\Cat_{\kappa}(\Bc)$ becomes a reflexive subcategory of $[\Map(\Bv)^{\coop},\Cat]$, which is a strong basis for considering completion of $\Bc$-categories as a sort of ``sheafification''. In the case of quantaloids indeed, see section~\ref{secQuant}, what happens is that $\Cat_{\kappa}(\Bc)$ is a reflexive subcategory of $[\Map(\Q)^{\op},\Poset]$; the symmetric case yields a reflexive subcategory $\Catsk(\Q)$ of $[\Map(\Q)^{\op},\Set]$ and at least in the case of Grothendieck quantaloids~\cite{HSIQGQ} the reflexion is left-exact, i.e. preserves all finite limits. In general, the left-exactness must be checked on all finite bilimits however, as expressed in~\cite{StreetCBS}.
\end{rk}

\section{Examples}\label{secEx}

This section presents two classes of examples in which our adjunction can yield useful results, particularly by studying its fixed points: quantaloids and monoidal categories. These two examples are complementary in that quantaloids have ``trivial'' $2$-structure (without being completely discrete) and monoidal categories have trivial $0$-structure. 

\subsection{Quantaloids}\label{secQuant}

For a general exposition of the theory of quantaloids, refer to~\cite{RosenthalTOQ}. In all this section $\Q$ is a quantaloid, possibly endowed with an involution. Several results of this section can be already found in section 4.2 of~\cite{HeymansThesis} \\ 

One key point in the computation of fixed points of $[\Map(\Bc)^{\op},\Cat]$ is finding a subcategory $\D$ of $\Cat$ such that all fixed points are objects of $[\Map(\Bc)^{\op},\D]$; this is done by remarking that categories of the form $\pma(x)$ are not recovering all categories, because \textit{arrows in $\pma(x)$ are $2$-cells in $\Bc$}. This very important remark can be made into a general principle stating that, as long as we are not talking about compositional properties which are not preserved because the composition of arrows of $\pma(x)$ is not the same as that of $2$-cells in $\Bc$, \textit{the arrow structure of $\D$ corresponds to the $2$-structure of $\Bc$}. In the case of quantaloids, that means that we can restrict the codomain of $P$ to $[\Map(\Q)^{\op},\Poset]$. We will see that the symmetry condition on a bicategory $(\M,A)$ implies that for any $x \in \Ob(\Q)$ the category $\pma(x)$ is discrete, and we thus recover functors from $\Map(\Q)^{\op}$ to $\Set$ through symmetry considerations. The importance of quantaloids as representatives for $\Set$-presheaves comes from the fact that $\Sup$-enriched categories are what we obtain when ``upgrading'' usual $\Set$-categories by asking that they are locally cocomplete and closed. 

\begin{ex}[Complete metric spaces]
Consider the quantale $\overline{\Rb}^+$ of positive real numbers as defined in Example~\ref{realnumbersquantale}; then skeletal (symmetric symmetrically) complete $\Q$-categories are complete metric spaces. The only map in $\Q$ is $0$, hence $[\Map(\Q)^{\op},\Cat]$ is the category $\Cat$. Denoting $\mathbf{CMSp}$ the category of complete metric spaces with a possibly infinite metric, we have an adjunction between $\Cat$ and $\mathbf{CMSp}$. It is given as follows: 

\begin{enumerate}
\itb For any complete metric space $(E,d)$, the corresponding category $P_{(E,d)}$ has as objects the points of $E$; for each pair of points $x,y \in E$, there is an arrow $x \to y$ if and only if $d(x,y)=0$. 
\itb For any category $\D$, we have a $\Q$-category $(\N,\Ob(\D)) = \si \D$ with $\N(x,y)$ being $0$ if there exists some $f: y \to x$ in $\D$, $\infty$ if not. Taking the completion of that category amounts to look at the connected components of $\D$.
\end{enumerate}

This adjunction restricts to an equivalence between sets and discrete metric spaces.
\end{ex}

For more interesting examples, we will characterize the fixed points of the adjunction, and we start by giving descriptions of the form that a pseudofunctor must take if it is to be a fixed point. The following proposition is an immediate consequence of the definition of $\pma$ and assesses the fact that the arrow structure of each category $\pma(q)$ for $q \in \Ob(\Q)$ is modelled by the $2$-structure of $\Q$, i.e. is a poset.

\begin{prop}
Let $(\M,A)$ be a complete $\Q$-category. Then for any $q \in \Ob(\Q)$, $\pma(q)$ is a posetal category.
\end{prop}

This means that our adjunction can be directly restricted to an adjunction between the categories $[\Map(\Q)^{\coop},\Poset]$ and $\Catk(\Q)$; in other words the setting of $\Q$-categories is only appropriate to describe posetal sheaves over $\Map(\Q)^{\coop}$ and not general functors going into $\Cat$. The following proposition is even more important, as it underlies the role of symmetry as the crucial element which enables us to recover only set-theoretical sheaves instead of posetal ones.

\begin{prop}
Let $(\M,A)$ be a symetrically complete $\Q$-category. Then for any $q \in \Ob(\Q)$, $\pma(q)$ is equivalent to a discrete category.
\end{prop}

\begin{proof}
Let us recall the description of the category $\pma(q)$.
\begin{enumerate}
\itb Objects of $\pma(q)$ are elements of $A$ which are of type $q$: $\Ob(\pma(q)) = A_q$.
\itb For any two $a,b \in \Ob(\pma(q))$, there is at most one arrow $a \to b$ in $\pma(q)$; there is exactly one if and only if $\id_q \leq \M(b,a)$.
\end{enumerate}
This proves that $\pma(q)$ is a posetal category. Now because $\M$ is symmetric, we have that if there is an arrow $a \to b$, meaning that $\id_q \leq \M(b,a)$, then taking the involute yields $\id_q \leq \M(a,b)$, meaning that there is also an arrow $b \to a$. Because $\pma(q)$ is posetal that implies that those arrows are isomorphisms: in other words, all arrows of $\pma(q)$ are isomorphisms. This mean that taking the equivalence classes of objects (i.e. generated by the relations $a \sim b$ if and only if $\id_q \leq \M(a,b)$), one obtains a discrete category which is equivalent to $\pma(q)$.
\end{proof}

This proves that any symetrically complete $\Q$-category, $\pma$ is an element of the subcategory $[\Map(\Q)^{\op},\mathbf{GrpdPoset}]$ of $[\Map(\Q)^{\op},\Cat]$ of functors taking values in posetal groupoids; this subcategory is equivalent to $[\Map(\Q)^{\op},\Set]$ if we assume the axiom of choice, which is necessary to construct the inverse image of a connected component. From now on we can therefore consider that $\pma$ is a (set-valued) presheaf over $\Map(\Q)^{\op}$. Note that the notion of symmetry is crucial there because it is precisely what encapsulates the restriction from posets to sets; moreover we can drop the ``$\coop$'' part simply to a $\op$ part, as the $2$-structure (which was already posetal in the non-symmetric case) disappears completely here.

Now going back to the fixed points of $P\C\si$, the above proposition immediatly implies that any fixed point $F$ can already be imposed to go only to $\Set$ instead of general categories. In what follows, we therefore consider a presheaf $F: \Map_{\s}(\Q)^{\op} \to \Set$. \\ 

As we have pointed out above in Remark~\ref{sympasbien}, even if $F$ goes from $\Map_{\s}(\Q)^{\op}$ to $\Set$, there is no guarantee that $(\M,\si F)$ will be a symmetric $\Q$-category. This can be verified easily as $\M(a,b) = \bigvee \{f: tb \to ta \in \Map_{\s}(\Q)(tb,ta) \: \ f \leq \M(a,b)\}$ and since the involution is not asked to commute with colimits; note that even in the case where it would commute with colimits, there is no guarantee that $\bigvee \{f^{\circ} \leq \M(b,a) \: \  \text{$f$ map}\}$ is going to be $\M(b,a)$. The solution is to use a symmetrization functor, which is introduced in~\cite{HeymansStubbeSCCQEC} and is defined by $\M_{\s}(a,b) = \M(a,b) \wedge \M(b,a)^{\circ}$; this symmetrization is left adjoint to the inclusion functor, which enables us to use the same proof we gave to produce a symmetric version of our adjunction, going between $[\Map_{\s}(\Q)^{\op},\Set]$ and $\Catsk(\Q)$. In all that follows, we denote by $\Sc$ the symmetrization functor.

\begin{deft}
Consider the following notion of covering in $\Map(\Q)$: a family $(f_i)_{i\in I}$ of maps going into an object $q \in \Ob(\Q)$ is said to be \textbf{covering} if: \[\id_q \leq \bigvee_{ i\in I} f_if_i^{\circ}\]
\end{deft}

\begin{rk}
In general, this does not define a Grothendieck topology on $\Map(\Q)$. Indeed, while the easiest two axioms of the definition (maximality and stability under composition), it seems that this definition does not satisfy the property of being stable under pullbacks.
\end{rk}

From then we can state the following proposition from section 4.2 of~\cite{HeymansThesis}:

\begin{prop}[Corollary 4.2.11 of~\cite{HeymansThesis}]\label{ShQ}
Let $F: \Map(\Q)^{\op} \to \Set$ be a presheaf over $\Map(\Q)$. Then $\Sc \si F$ is symetrically complete if and only if:
\begin{enumerate}
\itb For every symmetric singleton $\s$ of $\Sc\si F$, there exists some covering family $(x_i, f_i)_{i\in I}$ with $x_i \in \Ob(F(tx_i))$, $f_i: tx_i \to t\s$ symmetric map for all $i \in I$, such that for any $a \in \si F$, $\s(a)f_i = \M(a,x_i)$ (we say that $\s$ is \textbf{locally representable}).
\itb For any covering family $(x_i,f_i)_{i\in I}$ with $x_i \in F(tx_i)$, $f_i: tx_i \to q$ with a fixed $q \in \Ob(\Q)$ which is compatible in the sense that $f_i^{\circ}f_j \leq \M(x_i,x_j)$, there exists a unique $x \in \si F$ such that for all $i\in I$, $F(f_i)(x_i) = x$.
\end{enumerate}
\end{prop}

\begin{rk}
These two conditions are to be understood by comparing with the case where $F$ is a presheaf on a topological space $X$. Then a covering family of maps of some $U \in \Ob(\R(X))$ is simply a covering in the usual sense. The first condition on $\si F$ is essentially a restriction property while the second condition is directly a glueing property. In Example~\ref{exFaiscfixes}, sheaves on $X$ will appear precisely as fixed points of the adjunction in the case $\Bc = \R(X)$.
\end{rk}

Now if $\Sc \si F$ is symetrically complete, then $\C_{\s} \Sc \si F = \Sc \si F$, and therefore we are brought to look at $P \Sc \si F$. 

\begin{prop}
Let $F$ be a presheaf on $\Map(\Q)$ such that $P\C_{\s}\Sc\si F = F$. Then $\Sc\si F$ is symetrically complete.
\end{prop}

\begin{proof}
For any $x \in \Ob(\Q)$, we have $(\si F)_x = \Ob(F(x)) = \Ob(P\C\si F(x)) = (\C\si F)_x$, meaning that any singleton corresponds exactly to an element of $(\si F)_x$, meaning that $\si F$ is symetrically complete.
\end{proof}

This shows that the above condition is necessary. To conclude, we only need to prove the following: 

\begin{prop}
If $\si F$ is symetrically complete, then $P\C_{\s} \si F = F$.
\end{prop}

\begin{proof}
First we have already seen that if $\si F$ is symetrically complete, then we have $\Ob(F(x)) = \Ob(P\C\si F(x))$ for all $x \in \Ob(\Q)$. Now recall that $F$ is a set-valued presheaf, hence we need not to compute arrows in $P\C\si F(x)$. We only need to prove that the transition functions are the same, i.e. that for any $f: q_1 \to q_2$ in $\Q$, $P\C\si F(f) = F(f)$. Now $P\C\si F (f)$ goes from $P\C\si F(q_2)$ to $P\C\si F(q_1)$, that is from $F(q_2)$ to $F(q_1)$ and it sends any $a \in F(q_2)$ to $a\cdot f$ defined as the element of $\si F$ which represents the singleton $\M(-,a)f$, where $\M$ is the structural matrix of $\si F$; here we have used the fact that $\C \si F = \si F$ because $\si F$ is symetrically complete. Therefore, what we want to prove is that $a \cdot f = F(f)(a)$, i.e. that the singletons $\M(-,a)f$ and $\M(-,F(f)(a))$ are equal. We recall both definitions: 
\begin{align*}
\M(x,F(f)a)) &= \bigvee \{g: q_1 \to tx \ \mathrm{map} \ | \ \exists u: F(f)(a) \to F(g)(x) \ \mathrm{in} \ F(q_1) \}\\
\M(x,a)f &= \bigvee \{hf \ | \ h: q_2 \to tx \ \mathrm{map} \ \mathrm{s.t.} \ \exists v: a \to F(h)(x) \ \mathrm{in} \ F(q_2)\}\\
\end{align*}

Now for any $h$ in the second join-defining set, there is some $v: a \to F(h)(x)$ in $F(q_2)$, and taking through $F(f): F(q_2) \to F(q_1)$, one obtains $F(f)(v): F(f)(a) \to F(hf)(x)$, meaning that $hf$ is actually in the first join-defining set, meaning that we have: 

\[\M(-,F(f)(a)) \leq \M(-,a)f\]

This is an inequality between maps in the quantaloid $\Dist(\Q)$; therefore it is an equality, which concludes the proof.

\end{proof}

\begin{ex}[Sheaves on a topological space]\label{exFaiscfixes}
Let $\R(X)$ be the quantaloid of open subsets of a topological space $X$ as defined in Example~\ref{secpreshf}, and let us describe the category $[\Map_{\s}(\R(X))^{\op},\Set]_f$ of fixed points in this case. A symmetric map $W: U \to V$ in $\R(X)$ is an open subset $W \leq U \wedge V$ such that $U \leq W$ and $V \leq W$; because $W \leq U \wedge V$ we necessarily have $W \leq U$ hence there is a map $U \to V$ if and only if $U \leq V$, therefore we recover the usual lattice of open subsets $\Map_{\s}(\R(X)) = \Omega(X)$. This means that the category $[\Map(\R(X))^{\op},\Set]$ is the usual presheaf category over $X$; this is a common procedure: to identify a bicategory $\Bc$ such that $\Map(\Bc)$ is our category of interest, often through a ``bicategory of relations'' techniques. Now all is left to do is to remark that the compatibility conditions of~\ref{ShQ} are in this case the same as sheaf conditions over $X$, proving that the presheaves on $X$ which are fixed points of the adjunction are precisely those that are sheaves. 
\end{ex}

Now for fixed $\Q$-categories, we have the following characterization, which is a direct reformulation of the definitions:

\begin{prop}\label{cpmama}
Let $(\M,A)$ be a complete $\Bc$-category. Then $(\M,A)$ is a fixed point of the adjunction if and only if the following equality is satisfied for all $a,b \in A$, $\M(a,b)$ is given by the following suprema:
\[\bigvee \{fg \ | \ x \in A, f: tx \to ta \text{ map}, g: tb \to tx \text{ map}, f\leq \M(a,x), g \leq \M(x,b)\}\]
\end{prop}

As $fg$ is a map when $f$ and $g$ are, and $\id$ is also a map, we have the following corollary, which builds on the ideas mentioned at the end of Remark~\ref{oplaxtrfAfgg}:

\begin{coro}\label{PtFixesDenseness}
Suppose that maps are dense in $\Q$, meaning that for any arrow $f: a \to b$ in $\Q$, we have $f = \bigvee \{\g \leq f \ | \ f \text{ map}\}$. Then any complete $\Q$-category is a fixed point of the adjunction.
\end{coro}

This condition is sufficient but it is of course non-necessary, as proven by the following example:

\begin{ex}[$\R(X)$-categories]
Let $(\M,A)$ be a complete $\R(X)$-category, with $X$ a topological space. Then $\si\pma$ has the same sections $A$ as $(\M,A)$, but the matrix $\N$ on $\si\pma$ is defined by: 
\[
\N(a,b) = 
\left\{
\begin{array}{lll}
    ta          &\text{ if $ta \subseteq tb$ and $a = b_{|ta}$.}\\ 
    tb          &\text{ if $tb \subseteq ta$ and $b = a_{|tb}$.}\\
    \emptyset   &\text{ else.}\\ 
\end{array}
\right.
\]

In other words $\N$ only encodes the restriction properties. Now it is only a quick check to verify that indeed we have:

\[
\M(a,b) = \bigvee_{x \in A} \N(a,x)\N(x,b)
\]

Which is the same condition as the one in proposition~\ref{cpmama}. In other words any complete $\R(X)$-category is a fixed point of the adjunction.

\end{ex}

\begin{coro}
The adjunction $\C\si \dashv P$ restricts to an equivalence of categories
\[\Catsk(\Q)_f \simeq [\Map_{\s}(\Q)^{\coop},\Set]_f,\] 
Where $\Catsk(\Q)_f$ is the full subcategory of $\Catsk(\Q)$ consisting of $\Q$-categories satisfying the property of~\ref{cpmama}, while $[\Map_{\s}(\Q)^{\op},\Set]_f$ is the full subcategory of $[\Map_{\s}(\Q)^{\op},\Set]$ consisting of presheaves which satisfy the compatibility property of proposition~\ref{ShQ}.
\end{coro}

\begin{ex}[Recovering Walters' result from~\cite{WaltersCahiers}]
In the case where $\Q = \R(X)$, we have seen that $\Catsk(\Q)_f = \Catsk(\Q)$ and that $$[\Map_{\s}(\Q)^{\op},\Set]_f = \Sh(X).$$ We thus recover the result of Walters which stated that sheaves on $X$ are the same as symmetric Cauchy-complete categories on $\R(X)$.
\end{ex}

\begin{ex}[Sheaves on a site]
In the case of a site $(\C,J)$ and of the quantaloid $\R(\C,J)$, this recovers the other result of Walters~\cite{Walters}, representing Grothendieck toposes by quantaloids.
\end{ex}

We can now give the following diagram which summarizes the case of quantaloids: 

\[
\begin{tikzcd}[row sep = huge, column sep = 35]
\Catk(\Q)_f
\arrow[r,"i"]
\arrow[d,"\simeq"description]
& \Cat_{\kappa}(\Q)
\arrow[d,"P"]
& \Cat(Q)
\arrow[l,"\C"]\\
{[\Map(\Q)^{\op},\Poset]_f}
\arrow[u,"\simeq" description]
\arrow[r,"i"]
&{[\Map(\Q)^{\op},\Poset]}
\arrow[r,"i"]
&{[\Map(\Q)^{\op},\Cat]}
\arrow[u,"\si"]
\end{tikzcd}
\]

And in the case of involutive quantaloids and symmetric maps, we get: 

\[
\begin{tikzcd}[row sep = huge, column sep = 35]
\Catsk(\Q)_f
\arrow[r,"i"]
\arrow[d,"\simeq"description]
& \Catsk(\Q)
\arrow[d,"P"]
& \Cat_{\s}(Q)
\arrow[l,"\C_{\s}"]\\
{[\Map(\Q)_{\s}^{\op},\Set]_f}
\arrow[u,"\simeq" description]
\arrow[r,"i"]
&{[\Map(\Q)_{\s}^{\op},\Set]}
\arrow[r,"i"]
&{[\Map(\Q)_{\s}^{\op},\Cat]}
\arrow[u,"\Sc\si"]
\end{tikzcd}
\]

\subsection{Monoidal categories}

In all this section, $\V$ will denote a Bénabou cosmos, i.e. a symmetric monoidal closed category which is complete and cocomplete. We denote by $\Bv$ the $1$-object bicategory obtained by delooping; its $1$-cells are objects of $\V$ and its $2$-cells are arrows of $\V$; horizontal composition of $2$-cells in $\Bv$ corresponds to the usual composition of arrows in $\V$ and composition of $1$-cells in $\Bv$ corresponds to the monoidal product in $\V$. Note that because $\Bv$ has only one object, $\Bv^{\coop}$ is actually the same thing as $\Bv^{\mathrm{co}}$, which is the same thing as $\V^{\op}$. For that reason, functors $\Map(\Bv)^{\coop} \to \Cat$ will be denoted as going from $\Map(\Bv)^{\co} \to \Cat$.

\begin{ex}[The functor $\C$]
The completion functor $\C$ is already known and correspond in this case to the usual notion of \textit{Cauchy-completeness} of $\V$-categories~\cite{BorceuxDejean,LackTendasFFCEC}. This notion is closely related to the splitting of idempotents, see 3.16 and 4.23 of~\cite{LackTendasFFCEC} for a characterization of Cauchy-complete $\V$-categories for monoidal categories $\V$ satisfying particular conditions. Several general well-known examples include (this is taken directly from~\cite{nlab:cauchy_complete_category}):
\begin{enumerate}
\item For $\V = \Set$, the completion of a $\V$-category (i.e. in this case a usual category) is the Karoubi enveloppe of the category and the notion of (Cauchy-)completion is the same as the notion of Karoubi-completion.
\item For $\V = \Ab$, the completion of a $\V$-category is a completion under finite direct sums and idempotent splitting.
\end{enumerate}
\end{ex}

\begin{prop}\label{UndrCatFunc}
The functor: 
\[P_{-}(*): \VCat \to \Cat\]
Is the ``underlying category'' functor $\VCat(\mathcal I,-)$, where $\cc I$ is the one-object $\V$-category with $\cc I(*,*) = I$ (with $I$ the neutral element of $\otimes$ in $\V$).
\end{prop}

\begin{proof}
By $P_-(*)$ we mean the functor which to any $\V$-category $(\M,A)$ associates the category $\pma(*)$, where $*$ is the only object of $\Bv$ (this can be defined even if $(\M,A)$ is not complete). Indeed, arrows $b \to a$ of $\pma(*)$ correspond exactly to arrows $I \to \M(a,b)$. Now let $f: (\M,A) \to (\N,C)$ be a $\V$-functor, we have for any $a \in A$, $(P_f)_*(a) = f(a)$, for any $\theta: I \to \M(a,b)$ in $\pma(*)$, $(P_f)_*(\theta) = f_{ab} \bullet \theta$; hence $(P_f)_* = \VCat(\cc I,f)$, which is the same thing as $\V(I,f)$. Note that $P$ also carries information about maps in $\Bv$, as we have a natural transformation $P_f(\g): P_{\D}(\g) \cdot (P_f)_* \simeq (P_f)_* \cdot P_{\C}(\g)$ for any map $\g$.
\end{proof}

In general, there seems to be no easy way to compute maps in $\Bv$; the symmetric case is slightly easier as the involution on $\Bv$ is the identity on $1$-cells. A symmetric map is thus given by:

\begin{enumerate}
\itb An object $x$ of $\V$.
\itb A morphism $\eta: 1 \to x \otimes x$ in $\V$.
\itb A morphism $\e: x \otimes x \to 1$ in $\V$.
\end{enumerate}

Such that:  

\begin{enumerate}
\itb $(\e \otimes \id_x) \circ (\id_x \otimes \eta) = \id_x$
\itb $(\id_x \otimes \e) \circ (\eta \otimes \id_x) = \id_x$
\end{enumerate}

\begin{ex}[Karoubi-complete categories]
Consider the case $\V = \Set$; then a symmetric map in $\Bc_{\Set}$ is the data of a set $X$, the choice of an element $(\eta_1,\eta_2)$ of $X \times X$ such that for any $x \in X$, $x = \eta_2$ and $x = \eta_1$. The only possibility is that $X = \{*\}$ is a singleton, or is empty. In this case the category of functors $[\Map(\Bc_{\Set})^{\op},\Cat]$ is the same as the category $\Cat$. On the other side, we have seen that $\Catsk(\Set)$ is the category of Karoubi-complete categories; our adjunction $\C\si \dashv P$ here becomes the usual adjunction between Karoubi-complete categories and usual categories. 
\end{ex}

Recall that in Remark~\ref{oplaxtrfAfgg}, we made a brief mention of some ``denseness condition'' of maps in general arrows in $\Bc$; this denseness condition is notably false in the elementary case of the locale of open subsets of a topological space. In the case of $\Bc_{\Set}$, however, this condition is true as any set can indeed be written as a colimit of copies of the singleton. As pointed out in the aforementioned remark, that means that we can extend our adjunction to a category of pseudonatural transformations between indexed categories (instead of oplax-natural transformation). However, this very useful property also yields the fact that all complete $\Bc$-categories are fixed points of the adjunction: the following proposition is essentially the same as corollary~\ref{PtFixesDenseness} in the case of one-object bicategories.

\begin{prop}\label{FixPtAdjMonCatCat}
Let $\V$ be a bicategory in which the diagram of all morphisms $I \to X$ is a colimit for any object $X$. Then any complete $\Bv$-category is a fixed point for the adjunction.
\end{prop}

\begin{proof}
This is immediate once we see that, denoting $\si \pma $ by $(\N,A)$, we have $\N = \M$. This does also hold in the symmetric case.
\end{proof}

Recall from section 2.5 of~\cite{Kelly} that the ``underlying category'' functor that we showed to be given by $P_-(*)$ sometimes has a left adjoint which is the ``free $\V$-category on a category'' functor. This can also be recovered with our formalism in certain cases in which we can construct a functor $F: \Map(\Bv)^{\co} \to \Cat$ from the data of its underlying category $F(*)$. This is the object of the following proposition in which we give a sufficient condition:

\begin{prop}
Suppose that $\V$ has a $0$ object such that $0 \times x = 0$ for any $x \in \Ob(\V)$ and that the only maps in $\Bv$ are $0$ and $\id$. Then we can extend the functor $\si$ to $\VCat$ into a functor $(-)_{\V}$ which is left adjoint to $P_-(*)$.
\end{prop}

\begin{proof}
Let $\C$ be a category and define $\C_{\V}$ to be the $\Bv$-category with same objects as $\C$ and with $\N(a,b)$ being the colimit of copies of the identity, taken as much as there are arrows $b \to a$ in $\C$, commutating up to the same $2$-cells as in the definition of $\si F$. Then proving the adjunction $(-)_{\V} \dashv P_-(*)$ becomes very straightforward and is already done in 2.5 of~\cite{Kelly}.
\end{proof}

\begin{rk}
The above proposition also works when there is no $0$ object if the only map is $\id$. In this case however, the proposition becomes even simpler because there is an equivalence between $[\Map(\Bv)^{\co},\Cat]$ and $\Cat$, and we do not even need to extend $\si$ as the free construction is directly given by $\si$.
\end{rk}

The above setup where the only maps are $0$ and $\id$ and in which objects of $\V$ are colimits of $\id$ are typical of concrete cases such as $\V = \Set$ and $\V = \Ab$. We can apply this setup to the study of the fixed points of the adjunction, notably these functors $F: \Map(\Bv)^{\co} \to \Cat$ which are so that $F = P\C\si F$. Recall that in the quantaloid case described in the previous section, fixed functors are exactly these for which $\si F$ is complete; now in the case of a one-object bicategory $\Bv$, suppose that $\si F$ is Cauchy-complete, then $\si F = \C \si F$. The fact that $F$ is a fixed point is now expressed by: $F = P\si F$; and we can easily describe $P\si F$: $P\si F(*)$ is the underlying category of $\si F$, meaning that it has the same objects as $F(*)$, while an arrow $a \to b$ in $P\si F$ is a morphism $I \to \M(a,b)$ in $\V$; in turn, $\M(a,b)$ is the colimit of maps $\g$ in $\V$ for which there is an arrow $b\to F(\g)(a)$ in $F(*)$. Now coming back to the concrete case where $\V = \Set$, the only map is $\id$, meaning that $\M(a,b)$ is the colimit in $\Set$ of several times the singleton, one for each arrow $b \to a$ in $F(*)$. Note that because $\id$ is always a map, we will always have at least as many copies of $\id$ in $\M(a,b)$ as there are arrows $b \to a$ in $F(*)$. 

This means that in the case where arrows $I \to \M(a,b)$ correspond to something like ``elements of $\M(a,b)$'' in the sense that the object of $\V$ corresponding to ``elements of $\M(a,b)$'' is the colimit of copies of $I$ for each arrow $I \to \M(a,b)$ (or copies of $\g$ for maps $\g$ with arrows $b \to F(\g)(a)$), we have that $F$ is a fixed point of the adjunction whenever $\si F$ is Cauchy-complete. This is precisely what happens in the case $\V = \Set$, where a set is always given by a colimit of copies of the singleton for each element of the set. Generally this consideration depends on the ``concreteness'' of $\V$, i.e. on how much the ``underlying category'' functor retains information about $\V$. As expressed in 1.3 of~\cite{Kelly}, this is decided by the representable functor $\V(I,-): \V^{\op} \to \Set$ and by how much this representable functor is faithful. More concretely, the above discussion yields the following proposition:

\begin{prop}~\label{FixPtAdjMonoCatFunc}
Let $\V$ be a monoidal category with maps only $\id$ (and possibly $0$) for which any object $X$ is the colimit the diagram of all arrows $I \to X$, with arrows $I \to I$ making the triangles commute. Then any functor $F: \Map(\Bv)^{\co} \to \Cat$ is a fixed point of the adjunction if and only if $\si F$ is Cauchy-complete.
\end{prop}

\begin{ex}
Let $\V$ be the monoidal category $\Ab$ of abelian groups. Here a map is a group $G$ together with the choice of an element $(\eta_1,\eta_2)$ of $G \otimes G$ and of a morphism $G \otimes G \to \Z$ (i.e. a bilinear morphism $G \times G \to \Z$) such that for any $g \in G$, we have: 
\begin{align*}
g   &= g \otimes 1 \otimes 1\\ 
    &= (\e \otimes \id_g)(g \otimes \eta_1 \otimes \eta_2)\\ 
    &= \e(g \otimes \eta_1) \cdot \eta_2
\end{align*}
This implies that $G$ is generated by $\eta_2$, in other words $G$ is a \textit{cyclic group}. Reciprocally, it is obvious that $0$ and $\Z$ are maps; let $n \in \N$ be a natural number greater or equal to $2$, then as there are no homomorphisms of groups $\e: \Zp n \to \Z$ apart for the trivial one, $\Zp n$ cannot be a map. The category $\Map(\Bv)$ here is reduced to one object and two $1$-cells $0$ and $\Z$; this gives us an easy description of functors $\Map(\Bv)^{\co} \to \Cat$ as the data of a category together with a ``zero morphism'' which is idempotent, and with natural transformations from the identity functor to itself (and to $0$) corresponding to elements of $\Z$. Now $\si F$ is going to be the ``free'' $\Ab$-category generated by this category, glueing things along $2$-cells $\Z \to \Z$ corresponding to the form of the diagram described in definition~\ref{GrothConstrDef}, meaning that for any $f: b \to a$ in $F(*)$, we have for any $n \in \Z$, that $(n\cdot \id_a) f$ in $F(*)$ is sent to $n \cdot f$ in $\si F$, where $n\cdot \id_a$ is $F(\aa_n)(a)$ with $\aa_n: \Z \to \Z$ the morphism sending $1$ to $n$. Now an above proposition~\ref{FixPtAdjMonoCatFunc} applies and $F$ is a fixed point whenever $\si F$ is Cauchy-complete, i.e. complete under direct sums and idempotent splittings.

Note that any complete $\Bv$-category has an element $a_0$ which is such that $\M(-,a_0) = 0$, which is equal to $a\cdot 0$ for any $a$. Here $\Ab$ does satisfy the condition of~\ref{FixPtAdjMonCatCat} and therefore we can deduce that any complete $\Ab$-category is a fixed point of the adjunction.
\end{ex}

\textbf{Acknowledgements}: Olivia Caramello has benefited for this work of the support of the Université Paris-Saclay in the framework on the Jean D’Alembert 2024 Programme.

\newpage
\printbibliography

\newpage

\textsc{Olivia Caramello} 

\vspace{0.2cm}
{\small \textsc{Dipartimento di Scienza e Alta Tecnologia, Universit\`a degli Studi dell'Insubria, via Valleggio 11, 22100 Como, Italy.}\\
	\emph{E-mail address:} \texttt{olivia.caramello@uninsubria.it}}

\vspace{0.2cm}

{\small \textsc{Istituto Grothendieck ETS, Corso Statuto 24, 12084 Mondovì, Italy.}\\
	\emph{E-mail address:} \texttt{olivia.caramello@igrothendieck.org}}

\vspace{0.2cm}

\small \textsc{Université Paris-Saclay, CentraleSupélec, Mathématiques et Informatique pour la Complexité et les Systèmes, 91190, Gif-sur-Yvette, France.}\\
	\emph{E-mail address:} \texttt{olivia.caramello@centralesupelec.fr}

\vspace{0.6cm}

\textsc{Elio Pivet} 

\vspace{0.2cm}
{\small \textsc{Department of Mathematics, ETH Zurich, Rämistrasse 101
8092 Zurich, Switzerland.}\\
\emph{E-mail address:} \texttt{epivet@ethz.ch}

\vspace{0.2cm}

{\small \textsc{Istituto Grothendieck ETS, Corso Statuto 24, 12084 Mondovì, Italy.}\\
	\emph{E-mail address:} \texttt{elio.pivet@ctta.igrothendieck.org}}

\end{document}